\documentclass[11pt,a4paper,reqno]{amsart}
\usepackage[english]{babel}
\usepackage{array}
\usepackage{blkarray}
\usepackage{makecell}
\usepackage{extarrows}
\usepackage{amssymb}
\usepackage{dsfont}
\usepackage[hidelinks]{hyperref} 
\usepackage[mathscr]{eucal} 
\theoremstyle{plain}
\usepackage{mathtools}
\usepackage{tikz}
\usepackage{stmaryrd}
\usepackage[all]{xy}
\UseComputerModernTips
\usepackage{comment}

\usepackage{fourier-orns}
\usepackage{txfonts}

\usepackage{geometry}
\usepackage[T1]{fontenc}

\newtheorem{theoremenonum}{Theorem}

\newcommand{\tens}[1][]{\mathbin{\otimes_{\raise1.5ex\hbox to-.1em{}{#1}}}}
\newcommand{\lltens}[1][]{{\mathop{\tens}\limits^{\rm \mathbb{L}}}_{#1}}

\newtheorem{theorem}{Theorem}[section]

\newtheorem{lemma}[theorem]{Lemma}

\newtheorem{proposition}[theorem]{Proposition}

\newtheorem{corollary}[theorem]{Corollary}

{\theoremstyle{definition}\newtheorem{notation}[theorem]{Notation}}

{\theoremstyle{definition}}

{\theoremstyle{definition}}

{\theoremstyle{definition}\newtheorem{example}[theorem]{Example}}

{\theoremstyle{definition}\newtheorem{definition}[theorem]{Definition}}

{\theoremstyle{definition}}

{\theoremstyle{definition}\newtheorem{remark}[theorem]{Remark}}

\newcommand{\UN}[4][r]{%
    \ar@/^1pc/[#1]^{#2}_*=<0.3pt>{}="HAUT"
    \ar@/_1pc/[#1]_{#3}^*=<0.3pt>{}="BAS"
    \ar @{=>} "HAUT";"BAS" ^{#4}
  }

\author{Damien Calaque}
\author{Julien Grivaux}
\thanks{D. Calaque acknowledges the financial support of the Institut Universitaire de France and of the ANR grant ``SAT'' ANR-14-CE25-0008}
\thanks{J. Grivaux acknowledges the financial support of the ANR grant ``MicroLocal''
ANR-15-CE40-0007 and ANR Grant ``HodgeFun'' ANR-16-CE40-0011}

\address{
IMAG, Univ Montpellier, CNRS, Montpellier, France \& Institut Universitaire de France
}

\address{
CNRS, I2M (Marseille) \& IH\'{E}S
}

\date \today
\email{damien.calaque@umontpellier.fr}
\email{jgrivaux@math.cnrs.fr}

\title{The Ext algebra of a quantized cycle}

\begin{document}

\begin{abstract}
Given a quantized analytic cycle $(X, \sigma)$ in $Y$, we give a categorical Lie-theoretic interpretation of a geometric condition, discovered by Shilin Yu, that involves the second formal neighbourhood of $X$ in $Y$. If this condition (that we call tameness) is satisfied, we prove that the derived Ext algebra $\mathcal{RH}om_{\mathcal{O}_Y}(\mathcal{O}_X, \mathcal{O}_X)$ is isomorphic to the universal enveloping algebra of the shifted normal bundle $\mathrm{N}_{X/Y}[-1]$ endowed with a specific Lie structure, strengthening an earlier result of C\u{a}ld\u{a}raru, Tu, and the first author This approach allows to get some conceptual proofs of many important results in the theory: in the case of the diagonal embedding, we recover former results of Kapranov, Markarian, and Ramadoss about (a) the Lie structure on the shifted tangent bundle $\mathrm{T}_X[-1]$ (b) the corresponding universal enveloping algebra (c) the calculation of Kapranov's big Chern classes. We also give a new Lie-theoretic proof of Yu's result for the explicit calculation of the quantized cycle class in the tame case: it is the Duflo element of the Lie algebra object $\mathrm{N}_{X/Y}[-1]$. 
\end{abstract}

\vspace*{1.cm}
\maketitle
\setcounter{tocdepth}{6}
\tableofcontents

\section{Introduction}

\subsection{The diagonal embedding case}

Let be $X$ a complex manifold or a smooth algebraic variety over a field of characteristic zero. Thanks to the celebrated result of Hochschild, Kostant and Rosenberg (\textit{cf}. \cite{HKRoriginal}), the Hochschild homology and cohomology groups of the structural sheaf $\mathcal{O}_X$ are given by $\mathrm{HH}_i(\mathcal{O}_X)=\Omega^i_X$ and $\mathrm{HH}^i(\mathcal{O}_X)=\Lambda^i \mathrm{T}_X$. These isomorphisms can be upgraded at the level of derived categories, and are called (geometric) HKR isomorphisms. For more history on this topic, we refer the reader to the paper \cite{Grivaux-HKR}, as well as references therein. In the present paper, we will be especially interested in the geometric HKR isomorphism involving Hochschild cohomology: this isomorphism is an additive sheaf isomorphism between the sheaf of polyvector fields  $\mathrm{S}(\mathrm{T}_X[-1])=\bigoplus_{i \geq 0} \Lambda^i \mathrm{T}_X [-i]$ and the derived Hom sheaf of the diagonal $p_{1*} \mathcal{RH}om_{\mathcal{O}_{X\times X}}(\mathcal O_X,\mathcal O_X)$. Here we view $X$ as the diagonal inside $X\times X$ (which is the geometric counterpart of looking at an algebra as a bimodule over itself) and $p_1$ is the first projection.
\par \medskip
Both members of this isomorphism have multiplicative structure: the wedge product on polyvector fields, and the Yoneda product on the derived Hom complex. Taking the cohomology on both sides, for any non-negative integer $p$, the corresponding isomorphism between the algebras $\bigoplus_{i=0}^p \mathrm{H}^{p-i}(X, \Lambda^i \mathrm{T}_X)$ and $\mathrm{Ext}^i_{X \times X}(\mathcal{O}_X, \mathcal{O}_X)$ is not multiplicative in general (contrarily to the homological geometric HKR). It has been conjectured by Kontsevich \cite{Kontsevich} and later proved by Van den Bergh and the first author \cite{CV} that one can use the square root of the Todd class $\mathrm{Td}\,(X)$ on $\bigoplus_i \mathrm{H}^i(X,\Omega^i_X)$ of the tangent sheaf $\mathrm{T}_X$ to ``twist'' the global cohomological HKR isomorphism in order to get an isomorphism of algebras. 
\par \medskip
The Todd class of $X$ is also intimately related to the geometry of the diagonal of $X$ in another way: it is the correction term appearing on the Grothendieck-Riemann-Roch theorem, so thanks to the Lefschetz formalism it can be interpreted as a restriction of the Grothendieck cycle class to the diagonal itself. This has been hinted in the pioneering work of Toledo-Tong (\textit{see} \cite[\S 6]{toledo_duality_1978}), and formalized using HKR isomorphisms in an unpublished manuscript of Kashiwara around 1992. Kashiwara's account can be found in \cite{conjecture}, where the second author proves in fact that the Todd class is the Euler class of $\mathcal O_X$, proving a conjecture of Kaschiwara--Schapira \cite{KS1}. 
\par \medskip 
The above results can be re-interpreted very naturally in Lie-theoretic terms after the works of Kapranov \cite{Kapranov} and Markarian \cite{Markarian}: 
\begin{itemize}
\item[--] The shifted tangent sheaf $\mathrm{T}_X[-1]$ is a Lie algebra object in the derived category $\mathrm{D}^{\mathrm{b}}(X)$. 
\item[--] Thanks to results of Markarian \cite{Markarian} and Ramadoss \cite{Ramadoss2}, the universal enveloping algebra of this Lie algebra object is indeed the derived Hom sheaf, and that the geometric HKR isomorphism can be re-interpreted as the Poincaré-Birkhoff-Witt (PBW) isomorphism. 
\item[--] Every element $\mathcal{F}$ in $\mathrm{D}^{\mathrm{b}}(X)$ is naturally a representation of $\mathrm{T}_X[-1]$, and via the PBW isomorphism the character of this representation (which is a central function on $\mathrm{U}(\mathrm{T}_X[-1])$) can be identified with the Chern character of $\mathcal{F}$.
\item[--] The Todd class becomes the derivative of the multipication map in the universal enveloping algebra, and is therefore the Duflo element of $\mathrm{T}_X[-1]$. 
\item[--] The isomorphism $\mathrm{HKR} \circ\iota_{\sqrt{\mathrm{Td}\,(X})}$ from \cite{Kontsevich,CV} can be seen as a Duflo isomorphism (\textit{see} \cite{duflo_operateurs_1977}) for the Lie algebra object $\mathrm{T}_X[-1]$. 
\end{itemize}
We refer to \cite{CalaqueCT}, \cite{CalaqueCT2} for further analogies between Lie theory and algebraic geometry. 

\subsection{More general embeddings: tame quantized cycles}

In the present paper, we are interested in the more general situation where we replace the diagonal embedding $\Delta_X \hookrightarrow X \times X$ by an arbitrary closed immersion $X \hookrightarrow Y$, where $X$ is a smooth closed subscheme of an ambiant smooth scheme $Y$. 
\par \medskip
In \cite{Arinkin-Caldararu}, Arinkin and C\u{a}ld\u{a}raru gave a necessary and sufficient condition for an additive generalized geometric HKR isomorphism to exist between $\mathcal{RH}om_{\mathcal O_Y}(\mathcal O_X,\mathcal O_X)$ and $\mathrm{S}(\mathrm{N}_{X/Y}[-1])$:  the condition is that $\mathrm{N}_{X/Y}$ extends to a locally free sheaf on the first infinitesimal neighborhood of $X$ in $Y$. The Lie theoretic interpretation of the first order neighborhood and of the above geometric condition has been given in \cite{CalaqueCT2} by C\u{a}ld\u{a}raru, Tu, and the first author. 
\par \medskip
Earlier on, in Kashiwara's 1992 unpublished manuscript, a more restrictive condition is introduced: Kashiwara deals with subschemes with split conormal sequence, which means that the map from $X$ to its first infinitesimal neiborhood in $Y$ admits a global retraction (in this case, \textit{any} locally free sheaf on $X$ extends at order one in $Y$). On the Lie side, this corresponds to pairs $\mathfrak{h} \subset \mathfrak{g}$ that split as $\mathfrak{h}$-modules, these are usually called reductive pairs. In \cite{Grivaux-HKR}, the second author developed Kashiwara's construction in this framework. The data of a subscheme $X$ of $Y$ together with such a retraction $\sigma$ is called a quantized cycle, and to such a cycle it is possible to associate geometric HKR isomorphisms, as well as a quantized cycle class $q_{\sigma}(X)$ living in $\bigoplus_{i \geq 0} \mathrm{H}^i(X, \Lambda^i \mathrm{N}^*_{X/Y}) $ that generalizes the Todd class in the diagonal case. Recently, answering a question raised by the second author in the article \cite{Grivaux-HKR}, Yu has shown in \cite{Yu} the following result: given a quantized cycle $(X, \sigma)$ in $Y$ such that $\sigma^* \mathrm{N}_{X/Y}$ extends to a locally free sheaf on the second infinitesimal neighborhood of $X$ in $Y$, 
\begin{enumerate}
\item[--] The quantized cycle $q_{\sigma}(X)$ class is completely determined by the geometry of the second infinitesimal neighborhood of $X$ in $Y$. 
\item[--] It can be expressed by an explicit formula similar to that of the usual Todd class.
\end{enumerate}
Yu's proof is based on direct calculation using the dg Dolbeault complex as well as homological perturbation theory. In this paper we provide a Lie theoretic explanation of Yu's results. We introduce the notion of \textit{tame quantized cycle} corresponding to Yu's condition: a quantized cycle $(X, \sigma)$ is tame if $\sigma^* \mathrm{N}_{X/Y}$ extends to a locally free sheaf at the second order. We can list all the conditions that can be investigated on the cycle $X$ (each condition being more restrictive than the previous one), and the corresponding conditions for Lie algebra pairs:
\par \medskip
\begin{center}
\begin{tabular}{|c|c|}
\hline 
cycles \,$X \subset Y$ & Lie pairs \, $\mathfrak{h} \subset \mathfrak{g}$ \\ 
\hline
$\mathrm{N}^*_{X/Y}$ extends at the first order & delicate condition, \textit{cf.} \cite{CalaqueCT} \\
\hline 
$X$ admits a retraction at the first order in $Y$ & $\mathfrak{g}=\mathfrak{h} \oplus \mathfrak{m}$ \\
\textit{i.e.} $X$ can be quantized  & as $\mathfrak{h}$-modules \\ 
\hline 
$X$ is a tame quantized cycle & $[\pi_{\mathfrak{h}} ([\mathfrak{m}, \mathfrak{m}]), \mathfrak{m}]=0$ \\ 
\hline 
$X$ admits a retraction at the second order in $Y$ & $\mathfrak{g}=\mathfrak{h} \rtimes \mathfrak{m}$ \\ 
\hline 
\end{tabular} 
\end{center}
\par \medskip
If one of the two last conditions is satisfied, the object $\mathrm{N}_{X/Y}[-1]$ is naturally a Lie object in $\mathrm{D}^{\mathrm{b}}(X)$, but this is no longer the case if we drop the tameness assumption. In full generality (that is without any specific quantization conditions), the algebraic structure of $\mathrm{N}_{X/Y}[-1]$ has been investigated in \cite{CalaqueCT}: it is a derived Lie algebroid, whose anchor map is given by the extension class of the normal exact sequence of the pair $(X, Y)$. Hence, our setting can be understood as the weaker universal hypotheses for which this derived Lie algebroid is a true Lie object in the symmetric monoidal category $\mathrm{D}^{\mathrm{b}}(X)$.

\subsection{The principal results}
The two main geometric results we prove in this paper deals with tame quantized analytic cycle. The first result is the explicit computation of the enveloping algebra of the Lie algebra object $\mathrm{N}_{X/Y}[-1]$.
\begin{theoremenonum}
Let $(X, \sigma)$ be a tame quantized cycle in $Y$.
The class $\alpha$ defines a Lie coalgebra structure on $\mathrm{N}^*_{X/Y}[1]$, hence a Lie algebra structure on $\mathrm{N}_{X/Y}[-1]$. Besides, the objects $\mathcal{RH}om^{\ell}_{\mathcal{O}_Y}(\mathcal{O}_X, \mathcal{O}_X)$ and $\mathcal{RH}om^{r}_{\mathcal{O}_Y}(\mathcal{O}_X, \mathcal{O}_X)$ are naturally algebra objects in the derived category $\mathrm{D}^{\mathrm{b}}(X)$, and there are commutative diagrams
\[
\xymatrix{
\sigma_* \mathcal{RH}om_{\mathcal{O}_S}(\mathcal{O}_X, \mathcal{O}_X) \ar[r] \ar[dd]_-{\mathrm{HKR}} & \mathcal{RH}om^{\ell}_{\mathcal{O}_Y}(\mathcal{O}_X, \mathcal{O}_X) \ar[d]^-{\mathrm{HKR}} \\
&  \mathrm{S}(\mathrm{N}_{X/Y}[-1]) \ar[d]^-{\mathrm{PBW}} \\
\mathrm{T}(\mathrm{N}_{X/Y}[-1]) \ar[r]& \mathrm{U}(\mathrm{N}_{X/Y}[-1])
}
\] 
and
\[
\xymatrix{
\sigma_* \mathcal{RH}om_{\mathcal{O}_S}(\mathcal{O}_X, \mathcal{O}_X) \ar[r] \ar[dd]_-{\mathrm{dual \,\, HKR}} & \mathcal{RH}om^{r}_{\mathcal{O}_Y}(\mathcal{O}_X, \mathcal{O}_X) \ar[d]^-{\mathrm{dual \, \, HKR}} \\
&  \mathrm{S}(\mathrm{N}_{X/Y}[-1]) \ar[d]^-{\mathrm{PBW}} \\
\mathrm{T}(\mathrm{N}_{X/Y}[-1]) \ar[r]& \mathrm{U}(\mathrm{N}_{X/Y}[-1])
}
\] 
where all horizontal arrows are algebra morphisms, and all vertical arrows are isomorphisms.
\end{theoremenonum}
At first glance, this result looks similar to the main result of \cite{CalaqueCT}, which says that in full generality, the universal envelopping algebra of the Lie algebroid $\mathrm{N}_{X/Y}[-1]$ is the object $\mathcal{RH}om_{\mathcal{O}_Y}(\mathcal{O}_X, \mathcal{O}_X)$, considered as an element of $\mathrm{D}^{\mathrm{b}}_{\Delta_X}(X \times X)$\footnote{This means considered as a kernel supported in the diagonal.}. The main subtlety here lies in the fact that if the Lie algebroid structure of $\mathrm{N}_{X/Y}[-1]$ in in fact a true Lie structure in $\mathrm{D}^{\mathrm{b}}(X)$, the universal envelopping algebras of $\mathrm{N}_{X/Y}[-1]$ as a Lie algebra object or as a derived Lie algebroid are not the same; they don't even live in the same categories. In the setting of \cite{CalaqueCT}, $\mathcal{RH}om_{\mathcal{O}_Y}(\mathcal{O}_X, \mathcal{O}_X)$ is naturally an algebra object in $\mathrm{D}^{\mathrm{b}}_{\Delta_X}(X \times X)$. In our setting, it is not at all obvious that $\mathcal{RH}om^{\ell}_{\mathcal{O}_Y}(\mathcal{O}_X, \mathcal{O}_X)$ and $\mathcal{RH}om^{r}_{\mathcal{O}_Y}(\mathcal{O}_X, \mathcal{O}_X)$ carry natural multiplicative structures. This is where the tameness condition becomes crucial.
\par \medskip
The second result we prove is in fact originally due to Shilin Yu \cite{Yu}, although our proof is completely different and Lie-theoretic. Yu's original statement is slightly different, because the Lie algebra structure on $\mathrm{N}_{X/Y}[-1]$ is not considered at all in his paper. The statement is the following:
\begin{theoremenonum}
Let $(X, \sigma)$ be a tame quantized cycle in $Y$. Then the quantized cycle class of $(X, \sigma)$ defined in \cite{Grivaux-HKR} is the Duflo element of the Lie algebra object $\mathrm{N}_{X/Y}[-1]$.
\end{theoremenonum}
The proofs of these two results combine two different types of ingredients: purely geometrical considerations linked to the geometry of formal neighbourhoods (mainly the two first ones) of a subscheme, as well as the use of abstract results on Lie algebra objects on symmetric monoidal categories. 

\subsection{Plan of the paper}

The paper is organized as follows: 
\begin{itemize}
\item[--] \S2 is devoted to some recollection about Lie algebra objects in a categorical setting. We claim no originality for this material which seems to be well known among experts in representations theory, but we were not locate the desired results in the form we needed in the literature. For instance, most references are written down for abelian categories while we work in the way more general Karoubian framework. The proof of the categorical PBW, which seems to be folklore among people in representation theory, is given using an operadic method in Appendix \ref{AAAA}.
\item[--] \S3 deals with three different topics. \S3.1 gives universal formulas for the multiplication map $\mathrm{U}(\mathfrak{g}) \otimes \mathfrak{g} \rightarrow \mathrm{U}(\mathfrak{g})$ via the PBW isomorphism. The proof is again provided in Appendix \ref{AAAA} using the operadic method. In \S3.2, we define an algebraic condition that characterizes uniquely the Duflo element of a Lie object in a symmetric monoidal category. Up to our knowledge, this result is new in this degree of generality (in the case $\mathfrak{g}=\mathrm{T}_X[-1]$, this is \cite[Lemma 4]{Markarian}). This will be the key ingredient to our Lie-theoretic proof of Yu's result. In \S 3.3, we introduce the ``tame condition'' for pairs of Lie algebras, which is a Lie theoretic analog of Yu's condition.
\item[--] \S4 recollects previous results on (first and second order) infinitesimal neighborhoods, HKR isomorphisms and quantized cycles. 
This is were we state the geometric tameness condition, a-k-a Yu's condition. 
\item[--] In \S5 we explain that the geometric tameness condition coincides with the Lie theoretic tameness condition for the pair $\big(\mathrm{T}_X[-1],\mathrm{T}_Y[-1]_{|X}\big)$ of Lie algebra objects in $\mathrm{D}^\mathrm{b}(X)$. We get in particular that in the tame case, $\mathrm{N}_{X/Y}[-1]$ is a Lie algebra object in $\mathrm{D}^{\mathrm{b}}(X)$ and we describe its universal enveloping algebra in geometric terms. Using the Lie-theoretic results of \S2-3, we are able to give enlightening and simple proofs of results of Ramadoss about Kapranov big Chern classes (diagonal case), and Yu's formula for the quantized cycle class. 
\item[--] In \S6 we use the above to get a description of the Ext algebra of a tame quantized cycle. We show in particular that it is completely determined by the second order infinitesimal neighborhood of $X$ in $Y$.
\item[--] We conclude the paper with a few perspectives in \S7. 
\end{itemize}

\section{Universal enveloping algebras and the categorical PBW theorem}

\subsection{Preliminary results of linear algebra}

\subsubsection{Partially antisymmetric tensors}

Let $\mathbf{k}$ be a field of characteristic zero and let $\mathcal{C}$ be a $\mathbf{k}$-linear symmetric monoidal category that is Karoubian (\textit{i.e.} every idempotent splits\footnote{Therefore, every multiple of a projector has a kernel and an image; we will repeatedly use this property. }) and such that countable direct sums exist and commute with the product. 
We can assume without loss of generality that $\mathcal{C}$ is a strict monoidal category (\textit{i.e.} it is harmless to drop the parenthesizations of iterated tensor products from the notation). 
For any non-negative integer $n$, the symmetric group $\mathfrak{S}_n$ acts naturally on $V^{\otimes n}$, where $V$ is an object of $\mathcal{C}$. 
Let $\pi_n$ be the element $(n!)^{-1} \sum_{g \in \mathfrak{S}_n} g$, considered as an idempotent element of the group algebra $\mathbf{k}[\mathfrak{S}_n]$. 
It induces a natural idempotent\footnote{Abusing notation, we will denote by the same symbol an element of $\mathbf{k}[\mathfrak{S}_n]$ and the induced endomorphism of $V^{\otimes n}$ in $\mathcal{C}$. } 
on $V^{\otimes n}$, whose kernel is denoted by $\widetilde{\Lambda}^n V$ and whose image is denoted by $\mathrm{S}^n V$. We therefore have a decomposition
\[
V^{\otimes n}= \mathrm{S}^n V \oplus \widetilde{\Lambda}^n V\,.
\]

Assume that $n$ is at least two and let 
\[
\Psi_n \colon \bigoplus_{i=1}^{n-1}\,  \left(V^{\otimes i-1} \otimes \Lambda^2 V \otimes V^{\otimes n-i-1} \right) \rightarrow V^{\otimes n}
\]
be the map obtained by embedding each $\Lambda^2 V:=\widetilde{\Lambda}^2 V$ in $V^{\otimes 2}$. 
We can provide a concrete description of $\widetilde{\Lambda}^n V$ via the map $\Psi_n$:
\begin{lemma} \label{cric}
The image of $\Psi_n$ exits and is canonically isomorphic to $\widetilde{\Lambda}^n V$. 
\end{lemma}

\begin{proof}
For $i$ in $\llbracket 1, n-1 \rrbracket$, let $\tau_{i}$ be the transposition in the group  $\mathfrak{S}_n$ that switches $i$ and $i+1$. 
We first observe that $(\mathbf{1}+\tau_{i})\pi_n=\pi_n(\mathbf{1}+\tau_{i})=\pi_n$. 
Hence the kernel of $\mathbf{1}+\tau_{i}$ acting on $V^{\otimes n}$ is a (split) sub-object of the kernel of $\pi_n$ acting on $V^{\otimes n}$: 
\[
V^{\otimes(i-1)} \otimes \Lambda^2 V \otimes V^{\otimes (n-i-1)}\subset \widetilde{\Lambda}^n V\,.
\]
In other words, the map $\Psi_n$ factors through $\widetilde{\Lambda}^n V$. 
In order to conclude, it is then sufficient to prove that $\bigoplus_{i=1}^{n-1}\,  \left(V^{\otimes i-1} \otimes \Lambda^2 V \otimes V^{\otimes n-i-1} \right)\to \widetilde{\Lambda}^n V$ admits a section.
We claim that the right ideal in $\mathbf{k}[\mathfrak{S}_n]$ generated by $\mathbf{1}-\tau_{i}$ ($1 \leq i \leq n-1$) contain all elements $\mathbf{1}-\tau$ for arbitrary  $\tau$ in $\mathfrak{S}_n$. 
Indeed, for any elements $g_1, \ldots, g_d$ in the group algebra $\mathbf{k}[\mathfrak{S}_n]$, we have
\[
\mathbf{1}-\Pi_{i=1}^d g_i=(\mathbf{1}-\Pi_{i=1}^{d-1} g_i)g_d+(\mathbf{1}-g_d).
\]
The claim follows again from the fact that the $\tau_i$ generate $\mathfrak{S}_n$. As a corollary, $\mathbf{1}-\pi$ sits in this ideal, so that we can choose elements $(a_i)_{1 \leq i \leq n-1}$ in the group algebra such that
\[
\sum_{i=1}^{n-1} (\mathbf{1}-\tau_i)\, a_i=1-\pi_n
\]
As a consequence we get that the map 
\[
\Phi_n:=\prod_{i=1}^{n-1} (\mathbf{1}-\tau_i)\, a_i:V^{\otimes n}\longrightarrow \bigoplus_{i=1}^{n-1}V^{\otimes n}
\] 
factors through $\bigoplus_{i=1}^{n-1}\,  \left(V^{\otimes i-1} \otimes \Lambda^2 V \otimes V^{\otimes n-i-1} \right)$ and $\Psi_n\circ\Phi_n$ is the projection onto the direct factor $\widetilde{\Lambda}^n V$ in $V^{\otimes n}$. 
\end{proof}

\begin{example} \label{explicit}
The simplest nontrivial case is  $n=3$. In this case we can take for instance
\[
\begin{cases}
6 a_1=3\mathbf{1}+\tau_2+\tau_2 \tau_1 \\
3 a_2=\mathbf{1} + \tau_1
\end{cases}
\]
\end{example}

\subsubsection{Jacobi identity}
Let us consider a morphism $\alpha \colon V^{\otimes 2} \rightarrow V$ in $\mathcal{C}$ which vanishes on $\mathrm{S}^2 V$ (\textit{i.e.} such that $\alpha\circ\tau_1=-\alpha$).
\begin{proposition} \label{Jacobi}
The pair $(V, \alpha)$ is a Lie algebra object if and only if there exists
\[
\beta \colon \widetilde{\Lambda}^3 V \rightarrow V
\]
such that $\beta \circ \Psi_3=(\alpha \circ(\alpha \otimes \mathrm{id}), \alpha \circ(\mathrm{id} \otimes \alpha))$. 
\end{proposition}

\begin{proof}
Let $u=\alpha \circ(\alpha \otimes \mathrm{id})$. The Jacobi identity is equivalent to the identity
\[
u- u \circ \tau_2+u \circ \tau_2 \tau_1=0\,.
\]
Remark now that $\alpha \circ(\mathrm{id} \otimes \alpha)=-u \circ \tau_2\tau_1$, and that $u \circ \tau_1=-u$. 
We now pre-compose $(u, -u \circ \tau_2 \tau_1)$ with the right inverse of ${\Psi}_3$ given by Lemma \ref{cric}. 
This gives a morphism $\beta \colon \widetilde \Lambda^3 V \to V$. Example \ref{explicit} and a short calculation provide the following explicit formula: 
$$
\beta=\dfrac{1}{3}(3u-u \circ \tau_2-u \circ \tau_2 \tau_1)\,.
$$
This morphism $\beta$ is the only possible candidate to fulfil the desired condition $\beta \circ \Psi_3=(u,-u\circ\tau_2\tau_1)$. 
Then we observe that 
\[
\begin{cases}
\beta_{| \Lambda^2V \otimes V}=u\\
\beta_{|V \otimes \Lambda^2V}=-u \circ \tau_2 \tau_1 + \dfrac{2}{3}(2u+u \circ \tau_2 \tau_1)
\end{cases}
\]
Hence $\beta\circ \Psi_3=(u, -u \circ \tau_2 \tau_1)$ if and only if $2u+u \circ \tau_2 \tau_1$ vanishes on $V \otimes \Lambda^2V$. 
This condition is clearly implied by the Jacobi identity, but it is in fact equivalent to it. 
Indeed, $2u+u \circ \tau_2 \tau_1$ vanishes on $V \otimes \Lambda^2V$ if and only if 
$(2u+u \circ \tau_2 \tau_1) \circ (\mathbf{1}-\tau_2)=0$, and  
\begin{align*}
(2u+u \circ \tau_2 \tau_1) \circ (\mathbf{1}-\tau_2)&=2u+u \circ \tau_2\tau_1-2u \circ \tau_2-u \circ \tau_2 \tau_1 \tau_2\\
&=2u+u \circ \tau_2\tau_1-2u \circ \tau_2+u \circ \tau_2 \tau_1\\
&=2(u-u \circ \tau_2+u \circ \tau_2 \tau_1).
\end{align*}
\end{proof}

\subsection{Algebras satisfying the PBW isomorphism} \label{hello2}

\subsubsection{The morphisms \texorpdfstring{$c_p^k$}{Cpk}} \label{hello}

Let $V$ be an object in $\mathcal{C}$ and let $\mathscr{A}$ be a unital augmented algebra object in $\mathcal{C}$ together with an algebra morphism $\Delta \colon \mathrm{T}(V) \rightarrow \mathscr{A}$. 
From now on, we require
\par \medskip
\begin{quote}
\textbf{Assumption (A1)} \,\,
\fbox{The composition
$
\Delta_+ \colon \mathrm{S} (V) \rightarrow \mathrm{T} (V) \rightarrow \mathscr{A}
$
is an isomorphism\footnotemark.}\footnotetext{Only as objects of $\mathcal{C}$, not as algebras.}
\end{quote}
\par \medskip
Here $\mathrm{T} (V):=\bigoplus_{n\geq0}V^{\otimes n}$ is equipped with the concatenation product, $\mathrm{S} (V):=\bigoplus_{n\geq0}\mathrm{S}^n V$ and the map $\mathrm{S} (V) \to \mathrm{T} (V)$ is the direct sum of direct factor embeddings $\mathrm{S}^n V \subset V^{\otimes n}$. 
Note that $\mathscr{A}$ carries a split increasing filtration: $\mathrm{F}^p\mathscr{A}:=\Delta_+ \big(\bigoplus_{i=0}^p\mathrm{S}^i V \big)$. 
Moreover, $\Delta$ is a filtered morphism for the obvious degree filtration on $\mathrm{T} (V)$. We now consider the restriction $\Delta^p$ of the filtered morphism $\Delta_+^{-1}\circ\Delta$ on each homogeneous component $V^{\otimes p}$; it decomposes as follows: 
\[
\Delta^p:=\Delta_+^{-1}\circ\Delta_{|V^{\otimes p}}=\sum_{k=0}^pc_p^k\,,
\]
where $c_p^k \in \mathrm{Hom}_{\mathcal{C}}(V^{\otimes p}, \mathrm{S}^k V)$. 
Since $\mathscr{A}$ is augmented, $c_p^0$ vanishes. 
We further make the following 
\par \medskip
\begin{quote}
\textbf{Assumption (A2)} \,\,
\fbox{
\begin{minipage}[]{10.5cm}
For any non-negative integer $p$ the morphism $V^{\otimes p} \rightarrow \mathrm{Gr}^p \mathscr{A} \simeq \mathrm{S}^p V$\\
is the canonical projection $\pi_p$. 
\end{minipage}
}
\end{quote}
\par \medskip
Note that this morphism is nothing but $\Delta^p$ followed by the projection onto $\mathrm{S}^p V$.
This second assumption on $\mathscr{A}$ ensures that the restriction of $c_p^k$ to $\mathrm{S}^p V$ is zero if $k \leq p-1$, and $c_p^p$ is the canonical projection $\pi_p$ from $V^{\otimes p}$ to $\mathrm{S}^p V$. 
In particular, if $1 \leq k \leq p-1$, then $c_p^k$ lives naturally in $\mathrm{Hom}_{\mathcal{C}}(\widetilde{\Lambda}^p V, \mathrm{S}^k V)$.

\subsubsection{The Lie bracket}

We define $\alpha=c_2^1 \colon \Lambda^2V \rightarrow V$. Since $\mathscr{A}$ is augmented, we can identify $\mathrm{F}^1 \mathscr{A}$ with $ \mathds{1} \oplus V$ (\textit{via} $\Delta_+$), where $\mathds{1}$ is the monoidal unit of $\mathcal{C}$. 
\begin{lemma} \label{pff}
Let $m \colon \mathscr{A}^{\otimes 2}\to\mathscr{A}$ be the associative product. Then the morphism $V^{\otimes 2}\to\mathscr{A}$ defined by 
$$
m\circ\Delta_{|V}^{\otimes 2}\circ (\mathbf{1}-\tau_1)
$$
factors through $\Delta_+(V)\subset \mathrm{F}^1 \mathscr{A}$, and coincides with $2\Delta_+\circ\alpha$.
\end{lemma}

\begin{proof}
As $\Delta$ is an algebra morphism, we have that 
$$
m\circ\Delta_{|V}^{\otimes 2}\circ (\mathbf{1}-\tau_1)=\Delta_{|V^{\otimes 2}}\circ (\mathbf{1}-\tau_1)=\Delta_+\circ(c_2^1+\pi_2)\circ(\mathbf{1}-\tau_1)=2\Delta_+\circ c_2^1=2\Delta_+\circ \alpha\,.
$$
We are done. 
\end{proof}

\begin{corollary} \label{ehoui}
The morphism $\alpha$ defines a Lie structure on $V$.
\end{corollary}

\subsubsection{Induction formul\ae~for \texorpdfstring{$c_p^k$}{Cpk}}
We can now provide explicit induction formul\ae~for the morphisms $c_{p}^k$. 
Recall that for $0 \leq k \leq p-1$, we consider $c_p^k$ as an element of $\mathrm{Hom}_{\mathcal{C}}(\widetilde{\Lambda}^p V, \mathrm{S}^k V)$.
\begin{proposition} \label{chalvet}
Given a pair $(V, \mathscr{A})$ as above, the coefficients $c_p^k$ are determined as follows:
\[
\left\{\!\begin{aligned}
\begin{minipage}{12cm}
$c_{p}^p=\pi_p:V^{\otimes p}\to\mathrm{S}^p V$. \\
$c_{p}^k \circ \Psi_{p}=\{c_{p-1}^k \circ (\mathrm{id}_{V^{\otimes i-1}} \otimes \alpha \otimes \mathrm{id}_{V^{\otimes p-i-1}}) \}_{1 \leq i \leq p-1} \,\,\, \textrm{if} \,\,\, 1 \leq  k \leq p-1$.
\end{minipage}
\end{aligned}
\right.
\]
Recall here that $\alpha=c_{2}^1$. 
\end{proposition}

\begin{proof}
For $1 \leq i \leq p-1$ we have
\begin{align*}
\Delta_{|V^{\otimes p}}\circ \frac12(\mathbf{1}-\tau_i)
& = m^{(2)}\circ \Delta_{|V^{\otimes i-1}}\otimes\big(\Delta_{|V^{\otimes 2}}\circ \frac12(\mathbf{1}-\tau_1)\big)\otimes\Delta_{|V^{\otimes p-i-1}} \\
& = m^{(2)}\circ \Delta_{|V^{\otimes i-1}}\otimes(\Delta\circ \alpha)\otimes\Delta_{|V^{\otimes p-i-1}} \qquad\textrm{(by Lemma \ref{pff})}\\
& = \Delta_{|V^{\otimes p-1}} \circ (\mathrm{id}_{V^{\otimes i-1}} \otimes \alpha \otimes \mathrm{id}_{V^{\otimes p-i-1}})\,,
\end{align*}
where $m^{(2)}:=m\circ (m\otimes\mathrm{id})$. Applying $\Delta_+^{-1}$ followed by the projection on the direct factor $\mathrm{S}^k V$ we get 
\[
c_p^k\circ \frac12(\mathbf{1}-\tau_i) = c_{p-1}^k \circ (\mathrm{id}_{V^{\otimes i-1}} \otimes \alpha \otimes \mathrm{id}_{V^{\otimes p-i-1}})\,.
\]
Hence both members of the induction relation agree on $V^{\otimes i-1} \otimes \Lambda^2 V \otimes V^{\otimes p-i-1}$ for every $1 \leq i \leq p-1$. 
\end{proof}

\subsection{Universal algebras in the categorical setting}

\subsubsection{Reverse PBW theorem} 

We can now prove our first main result: assuming that the algebra $\mathscr{A}$ satisfies the PBW theorem (\textit{i.e.} Assumptions (A1) and (A2)), 
we prove that it is the universal enveloping algebra of $V$ endowed with the Lie bracket $2c_2^1$.
\begin{proposition}[\textbf{Reverse categorical PBW}] \label{coupdepute}
If $\mathscr{A}$ satisfies Assumptions (A1-A2) of \S \ref{hello}, then $\mathscr{A}$ is a universal enveloping algebra of the Lie algebra $(V, 2\alpha)$.
\end{proposition}

\begin{proof}
For any associative algebra object $B$ in $\mathcal{C}$, with product $m_B:B^{\otimes 2}\to B$, we denote by $B_{\mathrm{Lie}}=(B,\mu_B)$ the Lie algebra object 
which is $B$ endowed with the Lie bracket $\mu_B=m_B\circ(\mathbf{1}-\tau_1)$. 
Lemma $\ref{pff}$ tells us that the direct factor inclusion $(\Delta_+)_{|V}:V \hookrightarrow \mathscr{A}$ is a morphism of Lie algebra objects from $(V,2\alpha)$ to $\mathscr{A}_{\mathrm{Lie}}$ in $\mathcal{C}$. 
\par \medskip
Assume now to be given a morphism $f:V\to B$. By the universal property of the tensor algebra, it defines an algebra morphism $\widetilde{f} \colon \mathrm{T}(V) \rightarrow B$. 
\begin{lemma}
If $f$ is a Lie algebra morphism (from $(V,2\alpha)$ to $B_{\mathrm{Lie}}$) then $\widetilde{f}$ factors through a unique morphism $g:\mathscr{A}\to B$. 
\end{lemma}
\begin{proof}[Proof of the Lemma]
It is sufficient to prove that $\widetilde{f}=\widetilde{f}\circ s\circ\Delta$, where $s$ is a section of $\Delta$. Indeed, the only possible choice is $g=\widetilde{f}\circ s$. 
Here we use the section $s$ given by $\Delta_+^{-1}:\mathscr{A}\to\mathrm{S}(V)\subset \mathrm{T}(V)$. 
\par \medskip
We will prove by induction on $p$ that $\widetilde{f}_{|V^{\otimes p}}=\widetilde{f}\circ \Delta_+^{-1}\circ\Delta_{|V^{\otimes p}}$. 
Note that this identity is obviously satisfied when restricted to $\mathrm{S}^pV\subset V^{\otimes p}$. 
The only thing left to prove is thus that  $\widetilde{f}_{|\widetilde{\Lambda}^pV}=\widetilde{f}\circ s\circ\Delta_{|\widetilde{\Lambda}^pV}$. 
\begin{enumerate}
\item[--] For $p\in\{0,1\}$ the result is obvious. 
\item[--] For $p=2$, we have 
\[
\widetilde{f}_{|\Lambda^2V}=\frac12\mu_B\circ (f\otimes f)_{|\Lambda^2V}=f\circ \alpha=f\circ c_2^1=\widetilde{f}\circ \Delta_+^{-1}\circ\Delta_{|\Lambda^2 V}\,.
\]
\item[--] Let us now assume that the required equality holds for a given $p\geq 2$. We compute  
\begin{align*}
\widetilde{f}\circ s\circ\Delta_{|V^{\otimes i-1}\otimes\Lambda^2V\otimes V^{\otimes p-i}} 
& = \widetilde{f}\circ\left(\sum_{k=1}^{p} c_{p+1}^k\right)_{|V^{\otimes i-1}\otimes\Lambda^2V\otimes V^{\otimes p-i}} \\
& = \widetilde{f}\circ\left(\sum_{k=1}^{p} c_{p}^k\right)\circ(\mathrm{id}_{V^{\otimes i-1}}\otimes\alpha\otimes\mathrm{id}_{V^{\otimes p-i}}) \quad\textrm{(by~Proposition \ref{chalvet})}\\
& = \widetilde{f}\circ(\mathrm{id}_{V^{\otimes i-1}}\otimes\alpha\otimes\mathrm{id}_{V^{\otimes p-i}}) \quad\textrm{(by induction)}\,.
\end{align*}
Finally one can prove that 
\[
\widetilde{f}_{|V^{\otimes i-1}\otimes\Lambda^2V\otimes V^{\otimes p-i}}=\widetilde{f}\circ(\mathrm{id}_{V^{\otimes i-1}}\otimes\alpha\otimes\mathrm{id}_{V^{\otimes p-i}})
\]
in the same way as for the case when $p=2$. 
\end{enumerate}
\end{proof}
\textit{End of the proof of Proposition \ref{coupdepute}.} In order to conclude one has to prove that $g$ is indeed a morphism of algebras: 
\begin{align*}
g\circ m & =\widetilde{f}\circ s\circ m=\widetilde{f}\circ s\circ m\circ(\Delta\circ s)^{\otimes 2}\\ 
& =\widetilde{f}\circ s\circ\Delta\circ m\circ s^{\otimes 2}=\widetilde{f}\circ m\circ s^{\otimes 2} \\
& = m\circ (\widetilde{f}\circ s)^{\otimes 2}=m\circ g^{\otimes 2}\,.
\end{align*}
\end{proof}

\subsubsection{Categorical PBW theorem}

In this section, we will give the definitive form of the categorical PBW theorem. This results relies heavily on the following fact, that we will prove in Appendix \ref{AAAA}:
\begin{theorem}  \label{canif}
There exists an algebra $\mathscr{A}$ that satisfies Assumptions (A1-A2).
\end{theorem}

\begin{remark}
In \cite[Lemma 1.3.7.5]{DM}, the authors provide an explicit multiplication law $m_\star$ on $\mathrm{S}(V)$ in the case when $\mathcal{C}$ is the category of graded vector spaces (or more generally any abelian category), and prove that it is associative. Then the algebra $\mathscr{A}=(\mathrm{S}(V), m_\star)$ satisfies all required properties. Their proof should carry on as well for a general $\mathcal{C}$, but the formulas defining $m_{\star}$ are daunting to write down in the categorical setting. This is why we provided an operadic approach in Appendix \ref{AAAA} (the result is written down in \S \ref{audiA5}).
\end{remark}

\begin{theorem}[\textbf{Categorical PBW}] \label{saroumane}
Let $V$ be an object of $\mathcal{C}$ and $\alpha$ be an element of $\mathrm{Hom}_{\mathcal{C}}(V, \Lambda^2 V)$.
\begin{enumerate}
\item[--] The system of equations
\begin{equation} \label{hell}
\left\{\!\begin{aligned}
\begin{minipage}{12cm}
$c_p^k \in \mathrm{Hom}_{\mathcal{C}}(\widetilde{\Lambda}^p V, \mathrm{S}^k V)$ for $1 \leq k \leq p-1$ \\
$c_{p}^p=\pi_p:V^{\otimes p}\to\mathrm{S}^p V$ \\
$c_{p}^k \circ \Psi_{p}=\{c_{p-1}^k \circ (\mathrm{id}_{V^{\otimes i-1}} \otimes \alpha \otimes \mathrm{id}_{V^{\otimes p-i-1}}) \}_{1 \leq i \leq p-1} \,\,\, \textrm{if} \,\,\, 1 \leq  k \leq p-1$
\end{minipage}
\end{aligned}
\right.
\end{equation}
has a solution $(c_p^k)_{1 \leq k \leq p}$ if and only if $\alpha$ is a Lie bracket on $V$. If it exists, this solution is unique.
\vspace{2pt}
\item[--] Given a Lie algebra object $(V, \alpha)$ in $\mathcal{C}$, let $(c_p^k)_{1 \leq k \leq p}$ be coefficients satisfying \eqref{hell}. If we define a product $m_\star$ on $\mathrm{S}(V)$ by the formula
\[
(m_\star)_{|\mathrm{S}^p V\otimes \mathrm{S}^q V}:=\left(\sum_{k=1}^{p+q} c_{p+q}^k\right)_{|\mathrm{S}^p V\otimes \mathrm{S}^q V},
\]
then $m_{\star}$ is associative and $\mathscr{A}=(\mathrm{S}(V),m_\star)$ is a universal enveloping algebra of $(V, 2\alpha)$. Besides, the PBW theorem holds.
\end{enumerate}
\end{theorem}

\begin{proof}
Assume that $\alpha$ is a Lie bracket. Theorem \ref{canif} gives an algebra $\mathscr{A}$ satisfying Assumptions (A1-A2) of \S \ref{hello}. 
Hence Proposition \ref{chalvet} provides the existence of the $c_p^k$.  Uniqueness is clear. Conversely, if the equations \eqref{hell} are satisfied, then we have
\[
c_3^1 \circ \Psi_3=(\alpha \circ(\alpha \otimes \mathrm{id}), \alpha\circ (\mathrm{id} \otimes \alpha))\,.
\]
Thanks to Proposition \ref{Jacobi}, $\alpha$ satisfies the Jacobi identity.
\par \medskip
Now assume that $\alpha$ is a Lie bracket, and let $\mathscr{A}$ be an algebra satisfying Assumptions (A1-A2) of \S \ref{hello}. 
Then we can transport the algebra structure on $\mathrm{S}(V)$ via the isomorphism $\Delta_+$: we have
\begin{align*}
(m_\star)_{|\mathrm{S}^p V\otimes \mathrm{S}^q V} 
& = \Delta_+^{-1}\circ m\circ (\Delta_{|\mathrm{S}^p V}\otimes \Delta_{|\mathrm{S}^q V}) \\
& =\Delta_+^{-1} \circ \Delta_{|\mathrm{S}^p V\otimes \mathrm{S}^q V} \\
& = \left(\sum_{k=1}^{p+q} c_{p+q}^k\right)_{|\mathrm{S}^p V\otimes \mathrm{S}^q V}\,.
\end{align*}
This gives the result.
\end{proof}

\section{Distinguished elements in the universal enveloping algebra}

\subsection{The derivative of the multiplication map} \label{parrain}
\subsubsection{The Todd series} \label{oural}
We borrow the notation from the previous Section: $(V,\alpha)$ is a Lie algebra object in $\mathcal{C}$ and $m_\star$ is the associative product on $\mathrm{S}(V)$ from Theorem \ref{saroumane}. 
Our aim is to give a closed formula for the restriction $\varphi$ of $m_\star$ to $\mathrm{S}(V)\otimes V\subset \mathrm{S}(V)\otimes \mathrm{S}(V)$.
\begin{enumerate}
\item[--] On $\mathrm{S}(V)$ we have an associative and commutative product $m_0$ defined as the composition
\[
\mathrm{S}(V)^{\otimes 2}\subset\mathrm{T}(V)^{\otimes 2}\to \mathrm{T}(V) \twoheadrightarrow  \mathrm{S}(V)\,.
\]
\par \smallskip
\item[--] We write $\tau_{p,q}$ for the transposition $(p,q)$ in the symmetric group, as well as for the corresponding action on $V^{\otimes n}$. 
\par \smallskip
\item[--] We consider the morphism $\omega_{(V,\,\mu)}:\mathrm{S}^nV\otimes V\to \mathrm{S}^{n-1}V\otimes V$ defined as 
\begin{equation}\label{omega0}
\omega_{(V,\,\mu)}:=(\mathrm{id}_{V^{\otimes n-1}}\otimes\mu)\circ\sum_{i=1}^n \tau_{i,n}=n\,(\mathrm{id}_{V^{\otimes n-1}}\otimes\mu)\,,
\end{equation}
where $\mu:=2\alpha$. 
We leave it as an exercise to check that the image of $\omega_{(V,\,\mu)}$ indeed factors through $\mathrm{S}^{n-1}V\otimes V\subset V^{\otimes n}$. 
We often simply write $\omega$ as the choice of $(V,\,\mu)$ is clear from the context. 
\item[--] We consider the morphism $\varpi_V:\mathrm{S}^nV\otimes V\to \mathrm{S}^{n-1}V\otimes\Lambda^2V\subset V^{\otimes n+1}$ defined as the restriction of $\omega_{\mathrm{T}(V)_{\mathrm{Lie}}}$ to $\mathrm{S}^nV\otimes V$. 
In other words, since $\tau_{i, n}=\mathrm{id}$ on $\mathrm{S}^n V \otimes V$, 
\[
\varpi_V=(\mathbf{1}-\tau_n)\circ\sum_{i=1}^n\tau_{i,n}=n\,(\mathbf{1}-\tau_n)\,.
\]
We often simply write $\varpi$ as the choice of $V$ is clear from the context. 
\end{enumerate}

\begin{theorem}\label{proputile}
The map $\varphi=m_{\star {|\mathrm{S}(V)\otimes V}}$ is the composition of $\dfrac{\omega}{1-\mathrm{exp}(-\omega)}$ with the multiplication morphism $m_0$ of the symmetric algebra $\mathrm{S}(V)$.
\end{theorem}
For the proof, we refer the reader to Appendix A (more specifically in \S \ref{audiA6}). In the classical case of ordinary Lie algebras, \textit{see} \cite[\S 5.2.1]{Ramadoss2}.

\subsubsection{Linear algebra computations} \label{miss}

All along this Subsection, we assume that $V$ is a dualizable object in $\mathcal{C}$ and we denote by $V^*$ its dual\footnote{The category being symmetric monoidal, every left dual is a right dual as well, so that we allow ourselves to simply speak about duals.}: in particular, 
we have a coevaluation map $\epsilon:\mathbf{1}_{\mathcal{C}}\to V^*\otimes V$ and an evaluation map $\delta:V\otimes V^*\to\mathbf{1}_{\mathcal{C}}$ that satisfy the ``snake'' identity 
\[
(\delta\otimes\mathrm{id}_V)\circ (\mathrm{id}_V\otimes\epsilon)=\mathrm{id}_V\,.
\]
\begin{enumerate}
\item[--] One shows easily that the restriction of $\mathrm{id}_{V^{\otimes n-1}}\otimes\delta$ to $V^{\otimes n}\otimes V^*$ to $\mathrm{S}^nV\otimes V^*$ 
factors through $\mathrm{S}^{n-1}V$. As a consequence the direct factor, inclusion $\mathrm{S}^nV\to \mathrm{S}^{n-1}V\otimes V$ can be re-written as the restriction of
\[
(\mathrm{id}_{V^{\otimes n-1}}\otimes\delta\otimes\mathrm{id}_V)\circ (\mathrm{id}_{V^{\otimes n}}\otimes\epsilon)\,.
\]
to $\mathrm{S}^n V$. Therefore\footnote{The morphism $\omega$ has been defined at the beginning of \S \ref{parrain}.} $\omega:\mathrm{S}^nV\otimes V\to \mathrm{S}^{n-1}V \otimes V$ equals $n$ times
\[
\left(\mathrm{id}_{V^{\otimes n-1}}\otimes\big(\mu\circ(\delta\otimes\mathrm{id}_{V^{\otimes2}})\big)\right)\circ 
(\mathrm{id}_{V^{\otimes n}}\otimes\epsilon\otimes \mathrm{id}_V)\,,
\]
which also equals
\[
\left(\mathrm{id}_{V^{\otimes n-1}}\otimes\big((\delta\otimes\mathrm{id}_{V})\circ(\mathrm{id}_{V\otimes V^*}\otimes\mu)\big)\right)\circ 
(\mathrm{id}_{V^{\otimes n}}\otimes\epsilon\otimes \mathrm{id}_V)\,.
\]
\item[--] We have an adjunction between the functor $V\otimes-$ and the functor $V^*\otimes-$: for any two objects $Y,Z$ in $\mathcal{C}$, 
\[
\mathrm{Hom}_{\mathcal{C}}(V\otimes Y,Z)\cong \mathrm{Hom}_{\mathcal{C}}(Y,V^*\otimes Z)\,,
\]
where the bijection is given by sending $\phi\in\mathrm{Hom}_{\mathcal{C}}(V\otimes Y,Z)$ to 
$\phi^*:=(\mathrm{id}_{V^*}\otimes\phi)\circ(\epsilon\otimes \mathrm{id}_Y)$. 
The inverse bijection sends $\psi\in\mathrm{Hom}_{\mathcal{C}}(Y,V^*\otimes Z)$ to $(\delta\otimes\mathrm{id}_Z)\circ(\mathrm{id}_V\otimes\psi)$. 
For instance, the element $\mu^*$ in $\mathrm{Hom}_{\mathcal{C}}(V,V^*\otimes V)$ is understood as the adjoint action, and we have 
\begin{equation} \label{olan}
\omega=n\times\mathrm{id}_{V^{\otimes n-1}}\otimes\big((\delta\otimes\mathrm{id}_{V})\circ(\mathrm{id}_V\otimes\mu^*)\big)\,.
\end{equation}
\item[--]
We also have an adjunction in the reverse way between the functor $-\otimes V$ and the functor $-\otimes V^*$: 
\[
\mathrm{Hom}_{\mathcal{C}}(Y\otimes V^*,Z)\cong \mathrm{Hom}_{\mathcal{C}}(Y,Z\otimes V)\,,
\]
where the bijection is given by sending $\phi\in\mathrm{Hom}_{\mathcal{C}}(Y\otimes V^*,Z)$ to 
$\phi^*:=(\phi\otimes\mathrm{id}_V)\circ(\mathrm{id}_Y\otimes\epsilon)$. \vspace{0.2cm}
\item[--] Taking into account that $\mathrm{S}^p(V^*)$ is canonically isomorphic to $\mathrm{S}^p(V)^*$, 
we get a canonical ``contraction map'' element $\mathfrak{c}_p\in\mathrm{Hom}_{\mathcal{C}}(\mathrm S^nV\otimes \mathrm{S}^pV^*,\mathrm{S}^{n-p}V)$ 
given as $\displaystyle \frac{n!}{(n-p)!}$ times the adjoint to the direct factor inclusion $\mathrm{S}^nV\hookrightarrow\mathrm{S}^{n-p}V\otimes\mathrm S^pV$. More explicitly, on $\mathrm S^nV\otimes \mathrm{S}^pV^*$
\[
\mathfrak{c}_p = \frac{n!}{(n-p)!}\times(\mathrm{id}_{V^{\otimes n-p}}\otimes\delta)\circ\cdots\circ
(\mathrm{id}_{V^{\otimes n-1}}\otimes\delta\otimes\mathrm{id}_{(V^*)^{\otimes p-1}})\,.
\]
\item[--]
Notice that $\mathfrak{c}_p=\mathfrak{c}_1\circ(\mathfrak{c}_1\otimes\mathrm{id}_{V^*})\circ\cdots\circ(\mathfrak{c}_1\otimes\mathrm{id}_{(V^*)^{\otimes p-1}})$. 
One the other hand, using \eqref{olan}, we obtain that $\omega$ equals $(\mathfrak{c}_1\otimes\mathrm{id}_V)\circ(\mathrm{id}_{V^{\otimes n}}\otimes\mu^*)$ and thus 
\begin{equation} \label{kir}
\omega^{\circ p}=(\mathfrak{c}_1\otimes\mathrm{id}_V)\circ(\mathrm{id}_{V^{\otimes n-p+1}}\otimes\mu^*)
\circ\cdots\circ(\mathfrak{c}_1\otimes\mathrm{id}_V)\circ(\mathrm{id}_{V^{\otimes n}}\otimes\mu^*)\,.
\end{equation}
\end{enumerate}
We now introduce a convenient notation. Let $X,Y,Z$ be three objects in $\mathcal{C}$ and let $(A,m_A)$ be an associative algebra object in $\mathcal{C}$. 
We then have a $\mathbf{k}$-linear associative composition product 
\[
-\bullet-:\mathrm{Hom}_{\mathcal{C}}(Y,A\otimes Z)\times\mathrm{Hom}_{\mathcal{C}}(X,A\otimes Y)\to \mathrm{Hom}_{\mathcal{C}}(X,A\otimes Z)
\] 
defined as follows: for every $\phi:Y\to A\otimes Z$ and every $\psi:X\to A\otimes Y$ we set
\[
\phi\bullet\psi:=(m_A\otimes\mathrm{id}_Z)\circ(\mathrm{id}_A\otimes\phi)\circ\psi\,.
\]
In particular, if $X=Y=Z$ then we get that $\bullet$ turns $\mathrm{Hom}_{\mathcal{C}}(Y,A\otimes Y)$ into a $\mathbf{k}$-linear associative algebra. 
We are interested in the case $(A,m_A)=(\mathrm{S}(V^*),m_0)$. 
\begin{lemma}\label{jardin}
Let $\phi\in \mathrm{Hom}_{\mathcal{C}}(Y,\mathrm{S}^pV^*\otimes Z)$ and $\psi\in \mathrm{Hom}_{\mathcal{C}}(X,\mathrm{S}^qV^*\otimes Y)$. Then 
\[
(\mathfrak{c}_p\otimes\mathrm{id}_Z)\circ(\mathfrak{c}_q\otimes\phi)\circ(\mathrm{id}_{\mathrm{S}^nV}\otimes\psi)
=(\mathfrak{c}_{p+q}\otimes\mathrm{id}_Z)\circ(\mathrm{id}_{\mathrm{S}^nV}\otimes(\phi\bullet\psi))\,,
\]
as morphisms from $\mathrm{S}^nV\otimes X$ to $\mathrm{S}^{n-p-q}V\otimes Z$. 
\end{lemma}
\begin{proof}
First of all, in view of the expression for $\mathfrak{c}_p$ in terms of $\mathfrak{c}_1$'s, it is sufficient to prove the Lemma for $p=1$. 
Then, in view of the definition of the product $\bullet$, we have that the r.h.s.~
\[
(\mathfrak{c}_{1+q}\otimes\mathrm{id}_Z)\circ(\mathrm{id}_{\mathrm{S}^nV}\otimes(\phi\bullet\psi))
\]
equals
\[
(\mathfrak{c}_{1+q}\otimes\mathrm{id}_Z)\circ(\mathrm{id}_{\mathrm{S}^nV}\otimes m_0\otimes\mathrm{id}_Z)
\circ(\mathrm{id}_{\mathrm{S}^nV\otimes\mathrm{S}^qV^*}\otimes\phi)\circ(\mathrm{id}_{\mathrm{S}^nV}\otimes\psi)\,.
\]
Hence it is sufficient to show the following identity
\begin{align*}
(\mathfrak{c}_1\otimes\mathrm{id}_Z)\circ(\mathfrak{c}_q\otimes\phi) &=
(\mathfrak{c}_{1+q}\otimes \mathrm{id}_{V^*\otimes Z})\circ(\mathrm{id}_{\mathrm{S}^nV}\otimes m_0\otimes\mathrm{id}_Z)
\circ(\mathrm{id}_{\mathrm{S}^nV\otimes \mathrm{S}^qV^*}\otimes\phi)
\end{align*}
in $\mathrm{Hom}_{\mathcal{C}}(\mathrm{S}^nV\otimes \mathrm{S}^qV^*\otimes Y,\mathrm{S}^{n-q}V\otimes Z)$. 
Then observe that, from the very definition $\mathfrak{c}_{1+q}$, we have that 
\[
\mathfrak{c}_{1+q}\circ(\mathrm{id}_{\mathrm{S}^nV}\otimes m_0)=\mathfrak{c}_{1}\circ (\mathfrak{c}_{q}\otimes \mathrm{id}_{V^*})
\]
in $\mathrm{Hom}_{\mathcal C}(\mathrm{S}^nV\otimes \mathrm{S}^qV^*\otimes V^*,\mathrm{S}^{n-q}V)\,.$
As a consequence, it is sufficient to have that 
\[
\mathfrak{c}_q\otimes\phi=
(\mathfrak{c}_{q}\otimes \mathrm{id}_{V^*\otimes Z})\circ(\mathrm{id}_{\mathrm{S}^nV\otimes \mathrm{S}^qV^*}\otimes\phi)\,,
\]
which is obvious. 
\end{proof}

\begin{corollary} \label{pastis}
If we consider  $\mu^*$ in $\mathrm{Hom}_{\mathcal{C}}(V,V^*\otimes V)$, we have 
$
\omega^{\circ p}=(\mathfrak{c}_p\otimes\mathrm{id}_V)\circ\big(\mathrm{id}_{\mathrm{S}^nV}\circ(\mu^*)^{\bullet p}\big)\,.
$
\end{corollary}
\begin{proof}
Setting $Y=V$, we get a family of maps
$(\mu^*)^{\bullet p}\in \mathrm{Hom}_{\mathcal{C}}(V,\mathrm{S}^pV^*\otimes V)$. Lemma \ref{jardin} gives the result.
\end{proof}
\begin{lemma}\label{lemkitu}
For every $p\geq1$, $(\mu^*)^{\bullet p}\bullet\epsilon=0$. 
\end{lemma}
\begin{proof}
First observe that it is sufficient to prove it for $p=1$ (using the associativity of $\bullet$). 
Then note that since $\mu \colon \Lambda^2 V \to V$ then  
$$
(\mathrm{id}_{V^*}\otimes\mu^*)\circ\epsilon=(\mu^*)^*
$$
lies in $\mathrm{Hom}_{\mathcal{C}}(\mathbf{1}_{\mathcal{C}},\wedge^2V^*\otimes V)$. 
Therefore 
$$
\mu^*\bullet\epsilon=(\tau_1\otimes\mathrm{id}_V)\circ(\mu^*\bullet\epsilon)=\big((m_0\circ\tau_1)\otimes\mathrm{id}_V\big)\otimes(\mu^*)^*=-(m_0\otimes\mathrm{id}_V\big)\otimes(\mu^*)^*=-\mu^*\bullet\epsilon\,.
$$
Therefore $\mu^*\bullet\epsilon=0$, and we are done. 
\end{proof}

\subsubsection{The trace identity}

For every two objects $Y,Z$ in $\mathcal{C}$ one has a linear map 
\[
\mathrm{Tr}:\mathrm{Hom}_{\mathcal{C}}(Y\otimes V,Z\otimes V)\to \mathrm{Hom}_{\mathcal{C}}(Y,Z)
\] 
defined as follows: 
\[
\mathrm{Tr}(\phi)=(\mathrm{id}_Z \otimes \delta) \circ (\phi \otimes \mathrm{id}_{V^*}) \circ (\mathrm{id}_Y \otimes\bar\epsilon)\,,
\]
where $\bar\epsilon$ is given by $\epsilon$ followed by the symmetry morphism $V^*\otimes V\to V\otimes V^*$. 

\begin{proposition}\label{vanoise}
For every $p$, we have that 
$$
\mathfrak{c}_1\circ\big((m_0\circ\omega^{\circ p})\otimes \mathrm{id}_{V^*}\big)\circ (\mathrm{id}_{\mathrm{S}^nV} \otimes\bar\epsilon)
=\mathfrak c_p\circ(\mathrm{id}_{\mathrm{S}^nV}\otimes\mathrm{Tr}((\mu^*)^{\bullet p}))\,.
$$
\end{proposition}
\begin{proof}
Using the fact that $\mathfrak c_1$ defines an action of $V^*$ on $\mathrm{S}(V)$ by derivations, we get that the r.h.s.~is 
$$
m_0\circ\big((\mathfrak{c}_1\otimes\mathrm{id}_V)\circ (\omega^{\circ p})^*+\mathrm{Tr}(\omega^{\circ p})\big)\,.
$$
Then observe that we have 
$$
m_0\circ(\mathfrak{c}_1\otimes\mathrm{id}_V)\circ (\omega^{\circ p})^*
=(\mathfrak{c}_{p+1}\otimes\mathrm{id})\circ (\mathrm{id}\otimes (\mu^*)^{\bullet p}\bullet\epsilon)=0\,.
$$
Finally, 
$$
m_0\circ\mathrm{Tr}(\omega^{\circ p})=\mathrm{Tr}(\omega^{\circ p})=\mathfrak c_p\circ(\mathrm{id}_{\mathrm{S}^nV}\otimes\mathrm{Tr}((\mu^*)^{\bullet p}))\,.
$$
We are done. 
\end{proof}

\subsection{The Lie-theoretic cycle class}

\subsubsection{The Duflo element} \label{dudu}

We borrow the notation from the previous paragraphs and introduce the {\it Duflo element} 
$\mathfrak{d}=\sum_{n=0}^{+ \infty} \mathfrak{d}_n \in \prod_n\mathrm{Hom}_{\mathcal{C}}(\mathbf{1}_{\mathcal{C}},\mathrm{S}^nV^*)$: 
\[
\mathfrak{d}:=\mathrm{det} \left(\dfrac{\mu^*}{1-\mathrm{exp}(-\mu^*)} \right)\,.
\]
This has to be understood as a formal expression in terms of the ``invariant polynomials'' 
$\nu_k:=\mathrm{Tr}\, \big((\mu^*)^{\bullet k}\big)$ that live in  $\mathrm{Hom}_{\mathcal{C}}(\mathbf{1}_{\mathcal{C}},\mathrm{S}^kV^*)$. 
For instance, 
\[
\mathfrak{d}=1 + \frac{\nu_1}{2}+\left(\frac{ 3\nu_1^{\bullet2}- \nu_2}{24} \right) + \frac{\nu_1^{\bullet3}-\nu_1\bullet\nu_2}{48} + \cdots
\]
More formally, for any $N \geq 0$, we write formally
\[
\prod_{n=1}^{N} \dfrac{x_i}{1-\mathrm{exp}(-x_i)}=\sum_{i \geq 0} P_i(y_1, y_2, \ldots, y_i)
\]
where $y_i=\displaystyle \sum_{n=0}^N x_n^i$, and $P_i$ is a polynomial independant from $N$, and of total degree $i$ if each variable $y_k$ has degree $k$. For instance, 
\[
\begin{cases}
P_0=1\\
P_1(y_1)=\displaystyle\frac{y_1}{2}\\
P_2(y_1, y_2)=\displaystyle \frac{ 3y_1^{2}-y_2}{24} \vspace{0.2cm}\\
P_3(y_1, y_2, y_3)=\displaystyle \frac{y_1^{3}-y_1 y_2}{48}
\end{cases}
\]
Then for any $p \geq 0$, we have
\[
\mathfrak{d}_p=P_p(\nu_1, \ldots, \nu_p)\,.
\]

\subsubsection{Torsion morphisms}

Let $\ell \in \mathbb{N}$ and let ${a}$ be a morphism from an arbitrary object $X$ to $\mathrm{S}^{\leq \ell} V$ and $(a_n)_{0\leq n \leq \ell}$ be the graded components of $a$.

\begin{definition}
We say that such a morphism ${a}$ is an $\ell$-torsion morphism if $m_\star\circ(a\otimes\mathrm{id}_V)$ factors through 
$\mathrm{S}^{\ell +1}V\subset \mathrm{S}^{\leq \ell+1} V$. \end{definition}
Our main result is:
\begin{theorem} \label{serpette}
If $a$ is an $\ell$-torsion morphism, then $a=\mathfrak{c}(\mathfrak{d}\otimes a_{\ell})$, where $\mathfrak{c}(\mathfrak d\otimes -)$ means $\sum_p\mathfrak{c}_p(\mathfrak{d}_p\otimes -)$, the sum being in fact automatically finite. 
\end{theorem}
Remark that this theorem tells nothing about the \textit{existence} of $\ell$-torsion morphisms.
\begin{proof}
Thanks to Theorem \ref{proputile}, ${a}$ is an $\ell$-torsion morphism if and only if the system of equations
\begin{equation} \label{dv}
\sum_{i=0}^{\ell}  \frac{\mathrm{B}_i}{i!} \left\{m_0 \circ\omega^{\circ i}\right\} (a_{\ell-k+i} \otimes \mathrm{id}_V)=0
\end{equation}
holds for $1 \leq k \leq \ell$. Each condition corresponds to the vanishing of the $\ell-k+1^{\mathrm{th}}$ graded piece of the element $m_\star\circ (a\otimes \mathrm{id}_V)$. Thanks to Proposition \ref{vanoise}, we get
\[
k a_{\ell-k} + \sum_{i=1}^k (-1)^i \frac{\mathrm{B}_i}{i!} \times \mathfrak{c}_i(a_{\ell-k+i}\otimes \nu_i)=0.
\]
Let us explain how it works on the first terms. 
\par \medskip
\begin{enumerate}
\item[--] For $k=1$, the first equation is
\[
a_{\ell-1}-\frac{1}{2} \mathfrak{c}_1(a_\ell\otimes\nu_1)=0,
\]
so 
\[
a_{\ell-1}=\mathfrak{c}_1 \left(a_\ell\otimes \frac{\nu_1}{2} \right)\,.
\] 
\par \smallskip
\item[--] For $k=2$, the second equation is
\[
2a_{\ell-2}-\frac{1}{2} \mathfrak{c}_1(a_{\ell-1}\otimes\nu_1) + \frac{1}{12} \mathfrak{c}_2(a_\ell\otimes\nu_2)=0
\]
and we get
\begin{align*}
a_{\ell-2}&=\frac{1}{4} \mathfrak{c}_1(a_{\ell-1}\otimes\nu_1) - \frac{1}{24} \mathfrak{c}_2(a_\ell\otimes\nu_2) \\
&=\frac{1}{8} \mathfrak{c}_2(a_\ell\otimes \nu_1^{\bullet2}) -\frac{1}{24} \mathfrak{c}_2(a_\ell\otimes\nu_2) \\
&=\mathfrak{c}_2 \left( a_{\ell} \otimes \frac{3 \nu_1^{\bullet 2}-\nu_2}{24} \right) \,.
\end{align*}
\par \smallskip
\item[--] For $k=3$, the third equation is
\[
3a_{\ell-3}-\frac{1}{2} \mathfrak{c}_1(a_{\ell-2}\otimes\nu_1)+ \frac{1}{12} \mathfrak{c}_2(a_{\ell-1}\otimes\nu_2)=0
\]
Hence
\begin{align*}
a_{\ell-3}
&=\frac{1}{3}\left( 
\frac{1}{2} \mathfrak{c}_1\left(\mathfrak{c}_2\left( \frac{1}{8} a_\ell\otimes\nu_1^{\bullet2}-\frac{1}{24} a_r\otimes\nu_2\right)\otimes\nu_1\right) 
-\frac{1}{12} \mathfrak{c}_2\left(\mathfrak{c}_1\left( \frac{1}{2} a_{\ell}\otimes\nu_1 \right) \right)\otimes \nu_2\right) \\
&= \mathfrak{c}_3\left(a_{\ell} \otimes \frac{\nu_1^{\bullet3}-\nu_1\bullet\nu_2}{48} \right)
\end{align*}
\end{enumerate}
\par \medskip

To conclude, it suffices to prove the induction relation
\[
k P_{k}(y_1 \ldots, y_k)  + \sum_{i=1}^k (-1)^i \frac{\mathrm{B}_i}{i!} \times y_i P_{k-i}(y_1 \ldots, y_{k-i})=0.
\]
involving the polynomials $P_i$. For this we fix the variables $x_1, \ldots, x_k$, and put $y_i=\sum_{n=0}^k x_n^i$. Since $y_0=k$, the identity is equivalent to 
\[
\sum_{i=0}^k (-1)^i \frac{\mathrm{B}_i}{i!} \times y_i P_{k-i}(y_1 \ldots, y_{k-i})=0
\]
The left-hand side is the homogeneous term of degree $k$ (in the variables $x_i)$ in the product
\[
\sum_{i \geq 0}  (-1)^i \frac{\mathrm{B}_i}{i!} y_i  \times \sum_{i \geq 0} P_i(y_1, \ldots, y_i) ,
\]
which is
\[
\sum_{p=1}^k \frac{x_p}{\mathrm{exp}(x_p)-1} \times \prod_{q=1}^k \frac{x_q}{1-\mathrm{exp}(-x_q)} 
\]
Let $\phi(x)=\displaystyle \frac{x}{1-\mathrm{exp}(-x)} \cdot$ Then $\phi(-x) \phi(x)=\phi(x)-x \phi'(x)$ so that for any $p$ with $1 \leq p \leq k$, the homogeneous coefficient of degree $k$ in 
$\displaystyle \frac{x_p}{\mathrm{exp}(x_p)-1} \times \prod_{q=1}^k \frac{x_q}{1-\mathrm{exp}(-x_q)}$ is
\[
\sum_{\alpha_1 + \ldots + \alpha_{k}=k} \frac{B_{\alpha_1}}{\alpha_1 !} \times \ldots \times \frac{B_{\alpha_{p-1}}}{\alpha_{p-1} !} \times \frac{B_{\alpha_p} (1-\alpha_p)}{\alpha_p !} \times \frac{B_{\alpha_{p+1}}}{\alpha_{p+1} !} \times \ldots \times  \frac{B_{\alpha_k}}{\alpha_k !} \cdot
\]
Taking the sum in $p$, we get 
\[
\sum_{\alpha_1 + \ldots + \alpha_{k}=k} \left( \frac{B_{\alpha_1}}{\alpha_1 !} \times \ldots \times \frac{B_{\alpha_k}}{\alpha_k !} \times  \sum_{p=1}^k (1-\alpha_p) \right)
\]
which is zero.
\end{proof}

\subsection{Tameness for pairs of Lie algebras} \label{mignon}
\subsubsection{Setting}
Assume to be given a triplet $(\mathfrak{g}, \mathfrak{h}, \mathfrak{n})$, where $\mathfrak{g}$ is a Lie algebra, $\mathfrak{h}$ is a Lie subalgebra of $\mathfrak{g}$, and $\mathfrak{g}=\mathfrak{h} \oplus \mathfrak{n}$ as $\mathfrak{h}$-module. This makes perfect sense in any abstract $\mathbf{k}$-linear symmetric monoidal category $\mathcal{C}$. We denote by $\mu_{\mathfrak{g}}$ and $\mu_{\mathfrak{h}}$ the Lie brackets on $\mathfrak{g}$ and $\mathfrak{h}$ respectively; and by $\pi_{\mathfrak{h}}$ and $\pi_{\mathfrak{n}}$ the two projections from $\mathfrak{g}$ to $\mathfrak{h}$ and $\mathfrak{n}$ respectively.
We also define $\alpha$ and $\beta$ in $\mathrm{Hom}_{\mathcal{C}}(\mathfrak{n}^{\otimes 2}, \mathfrak{n})$ and $\mathrm{Hom}_{\mathcal{C}}(\mathfrak{n}^{\otimes 2}, \mathfrak{h})$ respectively
by the formulas $\alpha=\pi_{\mathfrak{n}} \circ {\mu_{\mathfrak{g}}}_{| \mathfrak{n}^{\otimes 2}}$ and $\beta=\pi_{\mathfrak{h}} \circ {\mu_{\mathfrak{g}}}_{| \mathfrak{n}^{\otimes 2}}$.
\begin{definition} \label{defalg}
The triplet $(\mathfrak{g}, \mathfrak{h}, \mathfrak{n})$ is called \textit{tame} if the morphism
\[
\mu_{\mathfrak{g}} \circ (\beta \otimes \mathrm{id}_{\mathfrak{n}}) \colon \mathfrak{n}^{\otimes 3} \rightarrow \mathfrak{g}
\]
vanishes.
\end{definition}

\begin{lemma}
Given a tame triplet $(\mathfrak{g}, \mathfrak{h}, \mathfrak{n})$, the morphism $\alpha$ defines a Lie structure on $\mathfrak{n}$. Besides, $\mathfrak{n}$ becomes a Lie object in the symmetric monoidal category of $\mathfrak{h}$-modules.
\end{lemma}

\begin{proof}
We claim that we have $\pi_{\mathfrak{n}} \circ (\mu_{\mathfrak{g}} \circ ({\mu_{\mathfrak{g}}}_{| \mathfrak{n}^{\otimes 2}} \otimes \mathrm{id}_{\mathfrak{n}}))=\alpha \circ (\alpha \otimes \mathrm{id}_{\mathfrak{n}})$. Indeed, 
\begin{align*}
\pi_{\mathfrak{n}} \circ \mu_{\mathfrak{g}} \circ ({\mu_{\mathfrak{g}}}_{| \mathfrak{n}^{\otimes 2}} \otimes \mathrm{id}_{\mathfrak{n}})&= 
\pi_{\mathfrak{n}} \circ  \mu_{\mathfrak{g}} \circ (\alpha \otimes \mathrm{id}_{\mathfrak{n}}) +  \pi_{\mathfrak{n}} \circ \underbrace{\mu_{\mathfrak{g}} \circ (\beta \otimes \mathrm{id}_{\mathfrak{n}})}_{0} \\
&= \pi_{\mathfrak{n}} \circ \mu_{\mathfrak{g}} \circ (\alpha \otimes \mathrm{id}_{\mathfrak{n}}) \\
&= \alpha \circ (\alpha \otimes \mathrm{id}_{\mathfrak{n}}) \,.
\end{align*}
Using this, the Jacobi identity for $(\mathfrak{g}, \mu_{\mathfrak{g}})$ restricted to $\mathfrak{n}^{\otimes 3}$ at the source, and projected to $\mathfrak{n}$ at the target, gives the Jacobi identity for $(\mathfrak{n}, \alpha)$.
\end{proof}

\subsubsection{The envelopping algebra \texorpdfstring{$\mathrm{U}(\mathfrak{n})$}{U(n)}} \label{castelnau}

Since $\mathfrak{n}$ is a Lie algebra object in the category of $\mathfrak{h}$-modules in $\mathcal{C}$, the algebra object $\mathrm{U}(\mathfrak{n})$ is naturally endowed with an action by derivation of $\mathfrak{h}$. This action is simply induced by the adjoint action of $\mathfrak{h}$ on the tensor algebra $\mathrm{T}(\mathfrak{n})$. We define a morphism $\mathfrak{g} \otimes \mathrm{U}(\mathfrak{n}) \rightarrow \mathrm{U}(\mathfrak{n})$ componentwise as follows:

\begin{enumerate}
\item[--] The morphism $p \colon \mathfrak{n} \otimes \mathrm{U}(\mathfrak{n}) \rightarrow \mathrm{U}(\mathfrak{n})$ is the multiplication in $\mathrm{U}(\mathfrak{n})$.
\item[--] The morphism $q \colon \mathfrak{h} \otimes \mathrm{U}(\mathfrak{n}) \rightarrow \mathrm{U}(\mathfrak{n})$ is the action of $\mathfrak{h}$ on $\mathrm{U}(\mathfrak{n})$.
\end{enumerate}

\begin{lemma}
Given a tame triplet $(\mathfrak{g}, \mathfrak{h}, \mathfrak{n})$, the above morphism endows $\mathrm{U}(\mathfrak{n})$ with a $\mathfrak{g}$-module structure.
\end{lemma}

\begin{proof}
For any elements $x$ and $y$ in $\mathcal{C}$, be denote by $\tau$ the symmetry isomorphism $x \otimes y \xrightarrow{\sim} y \otimes x$.
We check componentwise (that is on $\mathfrak{n} \otimes \mathfrak{n}$, $\mathfrak{n} \otimes \mathfrak{h}$ and on $\mathfrak{h} \otimes \mathfrak{h}$) that the map $(p, q)$ defines a $\mathfrak{g}$-action. 
\begin{enumerate}
\item[--] We have on $\mathfrak{n} \otimes \mathfrak{n} \otimes \mathrm{U}(\mathfrak{n})$
\[
\underbrace{p \circ (\mathrm{id}_{\mathfrak{n}}\otimes p) - p \circ (\mathrm{id}_{\mathfrak{n}}\otimes p) \circ (\tau \otimes \mathrm{id}_{\mathrm{U}(\mathfrak{n})})- p \circ (\alpha \circ \mathrm{id}_{\mathrm{U}(\mathfrak{n})})}_{0 \,\,\textrm{since}\,\,\mathfrak{n}\,\,\textrm{acts on}\,\, \mathrm{U}(\mathfrak{n})}-\underbrace{q \circ (\beta \circ  \mathrm{id}_{\mathrm{U}(\mathfrak{n})})
}_{0 \,\, \textrm{by tameness}}=0 \,.
\]
\item[--] Since $\mathfrak{h}$ acts by derivation on $\mathrm{U}(\mathfrak{n})$, we have on $\mathfrak{n} \otimes \mathfrak{h} \otimes \mathrm{U}(\mathfrak{n})$ the equality
\[
q \circ ((\mathrm{id}_{\mathfrak{h}}\otimes p) \circ (\tau \otimes  \mathrm{id}_{\mathrm{U}(\mathfrak{n})}))=q \circ (\mu_{\mathfrak{g}} \circ  \mathrm{id}_{\mathrm{U}(\mathfrak{n})})+ 
p \circ (\mathrm{id}_{\mathfrak{n}}\otimes q),
\]
that is
\[
p \circ (\mathrm{id}_{\mathfrak{n}}\otimes q) - q \circ (\mathrm{id}_{\mathfrak{h}}\otimes p) \circ (\tau \otimes  \mathrm{id}_{\mathrm{U}(\mathfrak{n})}) -q \circ (\mu_{\mathfrak{g}} \circ  \mathrm{id}_{\mathrm{U}(\mathfrak{n})})=0\,.
\]
\item[--] Lastly, on $\mathfrak{h} \otimes \mathfrak{h} \otimes \mathrm{U}(\mathfrak{n})$, we have 
\[
q \circ (\mathrm{id}_{\mathfrak{h}}\otimes q) - q \circ (\mathrm{id}_{\mathfrak{h}}\otimes q) \circ (\tau \otimes  \mathrm{id}_{\mathrm{U}(\mathfrak{n})}) -q \circ (\mu_{\mathfrak{g}} \circ  \mathrm{id}_{\mathrm{U}(\mathfrak{n})})=0
\]
since $\mathfrak{h}$ acts on $\mathrm{U}(\mathfrak{n})$.
\end{enumerate}
\end{proof}

\subsubsection{The induced representation} \label{chili}

The aim of this section is to prove the following theorem:

\begin{theorem} \label{harpe}
Given a tame triplet $(\mathfrak{g}, \mathfrak{h}, \mathfrak{n})$, the induced $\mathfrak{g}$-module $\mathrm{Ind}_{\mathfrak{h}}^{\mathfrak{g}} \, \mathbf{1}_{\mathcal{C}}$ of the trivial $\mathfrak{h}$-module $\mathbf{1}_{\mathcal{C}}$ exists, and is naturally isomorphic to the $\mathfrak{g}$-module $\mathrm{U}(\mathfrak{n})$\footnote{The $\mathfrak{g}$-module structure on $\mathrm{U}(\mathfrak{n})$ has been introduced in \S \ref{castelnau}.}.
\end{theorem}

\begin{proof}

The proof goes in several steps. We want to prove that $\mathrm{U}(\mathfrak{n})$ satisfies the universal property of the induced representation, that is that for any $\mathfrak{g}$-module $V$, $\mathrm{Hom}_{\mathfrak{h}}(\mathbf{1}_{\mathcal{C}}, V) \simeq \mathrm{Hom}_{\mathfrak{g}}(\mathrm{U}(\mathfrak{g}), V)$.
\par \medskip
First, we claim that the induced representation $\mathrm{Ind}\,_0^{\mathfrak{g}} \, \mathbf{1}_{\mathcal{C}}$ exists and is isomorphic to $\mathrm{U}(\mathfrak{g})$. This means that  $\mathrm{Hom}_{\mathcal{C}}(\mathbf{1}_{\mathcal{C}}, V) \simeq \mathrm{Hom}_{\mathfrak{g}}(\mathrm{U}(\mathfrak{g}), V)$. The morphism is obtained by attaching to each $\varphi$ in $\mathrm{Hom}_{\mathcal{C}}(\mathbf{1}_{\mathcal{C}}, V)$ the map
\[
\mathrm{U}(\mathfrak{g}) \simeq \mathrm{U}(\mathfrak{g}) \otimes \mathbf{1}_{\mathcal{C}} \xrightarrow{\mathrm{id} \,\otimes \,\varphi} \mathrm{U}(\mathfrak{g}) \otimes V \rightarrow V.
\]
Its inverse if simply the precomposition with the map $\mathbf{1}_{\mathcal{C}} \rightarrow \mathrm{U}(\mathfrak{g})$. Let us consider the following diagram:
\[
\xymatrix{ \mathbf{1}_{\mathcal{C}} \ar[rd] \ar[rr] \ar@{-->}@/_/[rdd] && V \\
& \mathrm{U}(\mathfrak{n}) & \\
& \mathrm{U}(\mathfrak{g}) \ar@{=>}[u]  \ar@{=>}@/_/[ruu] & \\
}
\]
The vertical arrow is simply given by the left action of $\mathrm{U}(\mathfrak{g})$ on the unit element element $\mathds{1}$ of $\mathrm{U}(\mathfrak{n})$. Here the dashed arrow is a morphism in $\mathcal{C}$, the plain arrows are morphism of $\mathfrak{h}$-modules and the double arrows are morphisms of $\mathfrak{g}$-modules.  We have a diagram of morphism spaces
\[
\xymatrix{
\mathrm{Hom}_{\mathcal{C}}(\mathbf{1}_{\mathcal{C}}, V) \ar[r]^-{\sim} & \mathrm{Hom}_{\mathfrak{g}}(\mathrm{U}(\mathfrak{g}), V) \\
\mathrm{Hom}_{\mathfrak{h}}(\mathbf{1}_{\mathcal{C}}, V) \ar@{_{(}->}[u]& \ar[l] \ar[u] \mathrm{Hom}_{\mathfrak{g}}(\mathrm{U}(\mathfrak{n}), V)
}
\]
We now claim the following:
\begin{enumerate}
\item[--] The map $\mathrm{U}(\mathfrak{g}) \rightarrow \mathrm{U}(n)$ admits a section, in particular it has a kernel $\mathfrak{N}$ in the category $\mathcal{C}$.
\item[--] The map $\mathrm{U}(\mathfrak{g}) \otimes \mathfrak{h} \rightarrow \mathrm{U}(\mathfrak{g})$ factors through the kernel $\mathfrak{N}$, and the induced map $\mathrm{U}(\mathfrak{g}) \otimes \mathfrak{h} \rightarrow \mathfrak{N}$ is an isomorphism.
\end{enumerate}
To prove the two claims, we make heavy use of the categorical PBW theorem (Theorem \ref{saroumane}). 
If we endow $\mathrm{U}(\mathfrak{n})$ and $\mathrm{U}(\mathfrak{g})$ with their natural filtrations, then the action of $\mathfrak{g}$ on $\mathrm{U}(\mathfrak{n})$ is of degree $1$ with respect to this filtration. The induced map $\mathfrak{g} \otimes \mathrm{Gr}^p \mathrm{U}(\mathfrak{n}) \rightarrow \mathrm{Gr}^{p+1} \mathrm{U}(\mathfrak{n})$  is simply given by the projection $\mathfrak{g} \rightarrow \mathfrak{n}$ followed by the multiplication map $\mathfrak{n} \otimes \mathrm{S}^p \mathfrak{n} \rightarrow \mathrm{S}^{p+1} \mathfrak{n}$. Therefore, we see that the natural map $\mathrm{U}(\mathfrak{g}) \rightarrow \mathrm{U}(\mathfrak{n})$ is a filtered map of degree zero, and that the induced graded map $\mathrm{S}(\mathfrak{g}) \rightarrow \mathrm{S}(\mathfrak{n})$ is induced by the projection from $\mathfrak{g}$ to $\mathfrak{n}$.
We can now consider the map
\[
\delta: \mathrm{U}(\mathfrak{n}) \simeq \mathrm{S}(\mathfrak{n}) \hookrightarrow \mathrm{T}(\mathfrak{n}) \rightarrow \mathrm{T}(\mathfrak{g}) \rightarrow \mathrm{U}(\mathfrak{g})
\]
Then the composite map $\mathrm{U}(\mathfrak{n}) \xrightarrow{\delta} \mathrm{U}(\mathfrak{g}) \rightarrow \mathrm{U}(\mathfrak{n})$ is an isomorphism, since the associated graded map is the identity, and the filtration splits (as objects of $\mathcal{C}$).
\par \medskip
Let us prove the second claim. First we remark that the composite map
\[
\mathfrak{h} \hookrightarrow \mathfrak{g} \xrightarrow{ . \mathds{1}} \mathrm{U}(\mathfrak{n})
\]
vanishes. This gives the factorization of the map $\mathrm{U}(\mathfrak{g}) \otimes \mathfrak{h} \rightarrow \mathrm{U}(\mathfrak{g})$ by $\mathfrak{N}$. We can endow $\mathfrak{N}$ with the filtration induced by $\mathrm{U}(\mathfrak{g})$. Since for each non-negative integer $p$ the morphism $\mathrm{F}^p \mathrm{U}(\mathfrak{g}) \rightarrow \mathrm{F}^p \mathrm{U}(\mathfrak{n})$ admits a section, the natural map from $\mathrm{Gr}^p \mathfrak{N}$ to the kernel of $\mathrm{Gr}^p \mathrm{U}(\mathfrak{g}) \rightarrow \mathrm{Gr}^p \mathrm{U}(\mathfrak{n})$ is an isomorphism. Thus, using the PBW isomorphism, we have a commutative diagram
\[
\xymatrix{
\mathrm{Gr}^{p-1} \mathrm{U}(\mathfrak{g}) \otimes \mathfrak{h} \ar[r] \ar[dd]^-{\sim} & \mathrm{Gr}^p \mathfrak{N} \ar[d]^-{\sim} \\
& \mathrm{Ker} \left(\mathrm{Gr}^p \mathrm{U}(\mathfrak{g}) \rightarrow \mathrm{Gr}^p \mathrm{U}(\mathfrak{n}) \right) \ar[d]^ -{\sim} \\
\mathrm{S}^{p-1} \mathfrak{g} \otimes \mathfrak{h} \ar@{=}[r] & \mathrm{S}^{p-1} \mathfrak{g} \otimes \mathfrak{h}
}
\]
Hence the map $\mathrm{U}(\mathfrak{g}) \otimes \mathfrak{h} \rightarrow \mathfrak{N}$ is a degree one map whose associated graded map is an isomorphism. Therefore it is an isomorphism.
\par \bigskip
We now come back to the main proof. The only thing that remains to prove is that the composite morphism
\[
\mathrm{Hom}_{\mathfrak{h}} (\mathbf{1}_{\mathcal{C}}, V) \rightarrow \mathrm{Hom}_{\mathcal{C}} (\mathbf{1}_{\mathcal{C}}, V) \simeq \mathrm{Hom}_{\mathfrak{g}} (\mathrm{U}(\mathfrak{g}), V)
\]
factors through $\mathrm{Hom}_{\mathfrak{g}} (\mathrm{U}(\mathfrak{n}), V)$. This is equivalent to prove that the composition
\[
\mathrm{Hom}_{\mathfrak{h}} (\mathbf{1}_{\mathcal{C}}, V) \rightarrow \mathrm{Hom}_{\mathcal{C}} (\mathbf{1}_{\mathcal{C}}, V) \simeq \mathrm{Hom}_{\mathfrak{g}} (\mathrm{U}(\mathfrak{g}), V) \rightarrow \mathrm{Hom}_{\mathfrak{g}} (\mathrm{U}(\mathfrak{g}) \otimes \mathfrak{h}, V)
\]
vanishes. The image of a map $\varphi$ is the morphism
\[
\mathrm{U}(\mathfrak{g}) \otimes \mathfrak{h} \rightarrow \mathrm{U}(\mathfrak{g}) \simeq  \mathrm{U}(\mathfrak{g}) \otimes \mathbf{1}_{\mathcal{C}} \xrightarrow{\mathrm{id}_{\mathrm{U}(\mathfrak{g})}\, \otimes \, \varphi} \mathrm{U}(\mathfrak{g}) \otimes V \rightarrow V
\]
which can also be written as the composition
\[
\mathrm{U}(\mathfrak{g}) \otimes \mathfrak{h} \simeq  \mathrm{U}(\mathfrak{g}) \otimes \mathfrak{h} \otimes \mathbf{1}_{\mathcal{C}} \xrightarrow{\mathrm{id}_{\mathrm{U}(\mathfrak{g}) \,\otimes \,\mathfrak{h}}\, \otimes \, \varphi} \mathrm{U}(\mathfrak{g}) \otimes V \rightarrow V.
\]
Now we have a commutative diagram
\[
\xymatrix{
\mathfrak{h} \otimes \mathbf{1}_{\mathcal{C}} \ar[r]^{\mathrm{id}_{\mathfrak{h}}\, \otimes \, \varphi} \ar[d]_-{\textrm{zero map}} &\mathfrak{h} \otimes V \ar[d]\\
\mathbf{1}_{\mathcal{C}}\ar[r]^-{\varphi} &V 
}
\]
which allows to finish the proof.
\end{proof}
Remark now that we have a priori two algebra structures on the algebra $\mathrm{U}(\mathfrak{n})^{\mathfrak{h}}$: the first one is the natural one induced by the algebra structure on $\mathrm{U}(\mathfrak{n})$, and the second one is obtained using Frobenius duality, and Theorem \ref{harpe}, and the natural composition on $\mathrm{Hom}_{\mathfrak{g}} (\mathrm{U}(\mathfrak{n}), \mathrm{U}(\mathfrak{n}))$:
\[
\mathrm{U}(\mathfrak{n})^{\mathfrak{h}} = \mathrm{Hom}_{\mathfrak{h}} (\mathbf{1}_{\mathcal{C}}, \mathrm{Res}^{\mathfrak{g}}_{\mathfrak{h}} \mathrm{U}(\mathfrak{n})) \simeq \mathrm{Hom}_{\mathfrak{g}} ( \mathrm{Ind}^{\mathfrak{g}}_{\mathfrak{h}} \mathbf{1}_{\mathcal{C}}, \mathrm{U}(\mathfrak{n})) \simeq \mathrm{Hom}_{\mathfrak{g}} (\mathrm{U}(\mathfrak{n}), \mathrm{U}(\mathfrak{n})).
\]
These two structure are in fact compatible, this is the content of the following:

\begin{theorem} \label{yolo}
Given a tame triple $(\mathfrak{g}, \mathfrak{h}, \mathfrak{n})$, the natural isomorphism $\mathrm{U}(\mathfrak{n})^{\mathfrak{h}} \simeq \mathrm{Hom}_{\mathfrak{g}} (\mathrm{U}(\mathfrak{n}), \mathrm{U}(\mathfrak{n}))$ is an anti-morphism of algebra objects.
\end{theorem}

\begin{proof}
Let us first observe that the above isomorphism $f:\mathrm{U}(\mathfrak{n})^{\mathfrak{h}}\tilde\longrightarrow\mathrm{Hom}_{\mathfrak{g}} (\mathrm{U}(\mathfrak{n}), \mathrm{U}(\mathfrak{n}))$ can be described as follows: 
for every $P\in \mathrm{U}(\mathfrak{n})^{\mathfrak{h}}=\mathrm{Hom}_{\mathfrak{h}} (\mathbf{1}_{\mathcal{C}}, \mathrm{Res}^{\mathfrak{g}}_{\mathfrak{h}} \mathrm{U}(\mathfrak{n}))$,  
\[
f(P)=a\circ(\delta\otimes P)=m\circ(id\otimes P)\,,
\]
where $\delta:\mathrm{U}(\mathfrak{n})\hookrightarrow \mathrm{U}(\mathfrak{g})$ is as in \S \ref{castelnau}, $a:\mathrm{U}(\mathfrak{g})\otimes \mathrm{U}(\mathfrak{n})\to \mathrm{U}(\mathfrak{n})$ is the $\mathfrak{g}$-module structure of $\mathrm{U}(\mathfrak{n})$, and $m$ is the multiplication on $\mathrm{U}(\mathfrak{n})$. 
Hence we have that 
\[
f(m(P\otimes Q))=m\circ\big(id\otimes m(P\otimes Q)\big)=m\circ\big(m(id\otimes P)\otimes Q\big)=f(Q)\circ f(P).
\] 
\end{proof}

\section{Generalities on formal neighbourhoods}

\subsection{Sheaves on split square zero extensions}
\subsubsection{Theta morphisms}
In this section, we study some properties related to sheaves on a trivial first order thickening of a smooth scheme. Let us fix the setting: $X$ is a smooth $\mathbf{k}$-scheme, $\mathcal{V}$ is a locally free sheaf on $X$, and $S$ is the (split) first order thickening of $X$ by $V$; that is $\mathcal{O}_S=\mathcal{V} \oplus \mathcal{O}_X$ and $\mathcal{V}$ is a square zero ideal in $\mathcal{O}_S$. We denote by 
$j \colon X \rightarrow S$ and $\sigma \colon S \rightarrow X$ the natural morphisms. Let $\mathcal{F}$ be an element of $\mathrm{D}^{-}(X)$. Then the exact sequence
\begin{equation} \label{hop}
0 \longrightarrow \mathcal{V} \otimes \mathcal{F} \longrightarrow \sigma^* \mathcal{F} \longrightarrow \mathcal{F} \longrightarrow 0
\end{equation}
defines a morphism $
\Theta_{\mathcal{F}} \colon \mathcal{F} \longrightarrow \mathcal{V} \otimes \mathcal{F} [1]
$
in $\mathrm{D}^-(S)$. 
\begin{proposition} \label{matisse} The following properties are valid:
\begin{enumerate}
\item[(i)]
For any $\varphi \colon \mathcal{F} \rightarrow \mathcal{G}$ in $\mathrm{D}^-(X)$, we have
\[
\Theta_{\mathcal{G}} \circ \varphi = (\mathrm{id}_{\mathcal{V}} \otimes \varphi) \circ \Theta_{\mathcal{F}}.
\]
\item[(ii)]
For any $\mathcal{F}$ in $\mathrm{D}^-(X)$, we have \[
\Theta_{\mathcal{F}}= \Theta_{\mathcal{O}_X} \lltens{}_{\mathcal{O}_S} \, \mathrm{id}_{\sigma^* \mathcal{F}}.
\]
\item[(iii)] For any $\mathcal{F}$ in $\mathrm{D}^-(X)$, the composition 
\[
\mathcal{F} \xrightarrow{\mathrm{at}_S(\mathcal{F})} \Omega^1_S \,\lltens{}_{\mathcal{O}_S}\, \mathcal{F} [1] \rightarrow \Omega^1_S \otimes_{\mathcal{O}_S} \mathcal{F} [1] \simeq \mathcal{V} \otimes \mathcal{F}[1] \oplus \Omega^1_X \otimes \mathcal{F}[1]
\] 
is the couple $(\Theta_{\mathcal{F}}, \mathrm{at}_X(\mathcal{F}))$.
\vspace{0.2cm}
\item[(iv)] The morphism $\sigma_* \Theta_{\mathcal{F}}$ vanishes.
\end{enumerate}
\end{proposition}

\begin{proof}
(i), (ii) and (iv) are straightforward. For (iii), $\Theta_{\mathcal{F}}$ is a special occurrence of a residual Atiyah morphism, and we apply \cite[Prop. 4.9]{Grivaux-chernloc}.
\end{proof}

\subsubsection{Infinitesimal HKR isomorphism}
In this section, we describe the infinitesimal cohomological HKR isomorphism attached to a split square-zero extension of a smooth scheme.
\par \medskip
For any non-negative integer $p$, we define a morphism $\Delta_p \colon \mathcal{O}_X \rightarrow  \mathrm{T}^p (\mathcal{V}[1]) $ in the derived category $\mathrm{D}^{\mathrm{b}}(S)$ as follows: $
\Delta_0=\mathrm{id}$ and $\Delta_p=- \,\Theta_{\mathrm{T}^{p-1}(\mathcal{V}[1])} \circ \Delta_{p-1}$ for $p \geq 1$.
\begin{proposition} \label{hkrbabe}
For any vector bundles $\mathcal{E}_1$ and $\mathcal{E}_2$ on $X$, there is a canonical isomorphism
\begin{align*}
\mathrm{T} (\mathcal{V}[-1]) \,\lltens{} \, \mathcal{RH}om(\mathcal{E}_1, \mathcal{E}_2)  
&\simeq \bigoplus_{p \in \mathbb{N}} \,\mathcal{RH}om(\mathcal{V}^{\otimes p} \otimes \mathcal{E}_1, \mathcal{E}_2) [-p] \\
& \rightarrow \bigoplus_{p \in \mathbb{N}} \, \sigma_{*} \mathcal{RH}om_{\mathcal{O}_S} (\mathcal{V}^{\otimes p}[p] \otimes \mathcal{E}_1, \mathcal{E}_2) \\
& \rightarrow \bigoplus_{p \in \mathbb{N}} \, \sigma_{*} \mathcal{RH}om_{\mathcal{O}_S} (\mathcal{E}_1, \mathcal{E}_2)
\end{align*}
obtained by precomposing by $ \Delta_p \otimes  \mathrm{id}_{\sigma^* \mathcal{E}_1}$. Besides, this isomorphism is compatible with the Yoneda product for the pair $(\mathcal{E}_1, \mathcal{E}_2)$.
\end{proposition}
\begin{proof}
The first part of the proof is well known and follows from the existence of a canonical locally $\mathcal{O}_S$-free resolution of $\mathcal{V}$ on $S$ (\textit{see} \cite{Arinkin-Caldararu}). The compatibility with the Yoneda product follows from routine calculations using Proposition \ref{matisse}.
\end{proof}
It is also possible to derive the internal Hom with respect to the \textit{second} variable instead of the first one: this gives the infinitesimal counterpart of Kashiwara's \textit{dual} HKR isomorphism (\textit{see} \cite{Grivaux-HKR}). To do so, we replace the morphism $\Delta_1$ by the \textit{dual Atiyah morphism} $\Delta'_1 \colon \mathcal{V}^* [-1] \rightarrow \mathcal{O}_X$, which is obtained by the composition \[
 \mathcal{V}^*[-1] \xrightarrow{\mathrm{id}_{ \mathcal{V}^*[-1]} \otimes \Delta_1}  \mathcal{V}^*[-1] \otimes  \mathcal{V}[1] \xrightarrow{\mathrm{ev}} \mathcal{O}_X.
\]
Then for any integer $p$ we construct the morphism $\Delta^{'}_p \colon \mathrm{T}^{p} (\mathcal{V}^*[-1]) \rightarrow \mathcal{O}_X$, as well as the symmetric components $\Delta^{' +}_p \colon \mathrm{S}^{p} (\mathcal{V}^*[-1]) \rightarrow \mathcal{O}_X$ and $\Delta^{' -}_p \colon \widetilde{\Lambda}^{p} (\mathcal{V}^* [-1]) \rightarrow \mathcal{O}_X$ as we did before. Then the dual infinitesimal HKR isomorphism takes the following form:
\begin{proposition} \label{hkrbabe2}
For any vector bundles $\mathcal{E}_1$ and $\mathcal{E}_2$ on $X$, there is a canonical isomorphism
\begin{align*}
\mathrm{T} (\mathcal{V}[-1]) \,\lltens{} \, \mathcal{RH}om(\mathcal{E}_1, \mathcal{E}_2)  
&\simeq \bigoplus_{p \in \mathbb{N}} \,\mathcal{RH}om(\mathcal{E}_1, (\mathcal{V}^{*})^{\otimes p} \otimes \mathcal{E}_2) [-p] \\
& \rightarrow \bigoplus_{p \in \mathbb{N}} \, \sigma_{*} \mathcal{RH}om_{\mathcal{O}_S} ( \mathcal{E}_1, (\mathcal{V}^{*})^{\otimes p} [-p] \otimes\mathcal{E}_2) \\
& \rightarrow \bigoplus_{p \in \mathbb{N}} \, \sigma_{*} \mathcal{RH}om_{\mathcal{O}_S} (\mathcal{E}_1, \mathcal{E}_2)
\end{align*}
obtained by postcomposing by $ \Delta_p^{'} \otimes  \mathrm{id}_{\sigma^* \mathcal{E}_2}$. Besides, this isomorphism is compatible with the Yoneda product for the pair $(\mathcal{E}_1, \mathcal{E}_2)$.
\end{proposition}

\begin{remark}
The object $\sigma_{*} \mathcal{RH}om_{\mathcal{O}_S} (\mathcal{O}_X, \mathcal{O}_X)$ is a ring object in $\mathrm{D}^{+}(X)$.
Propositions \ref{hkrbabe} and \ref{hkrbabe2} give two a priori different isomorphisms between this ring object and $\mathrm{T} \mathcal{V}^* [-1]$. In the next Lemma, we will provide an identity relating $\Delta_p$ and $\Delta^{'}_p$ that shows that these two isomorphisms are in fact the same. 
\end{remark}

\begin{lemma} \label{bienpenible}
Let $\mathcal{W}$ be any element in $\mathrm{D}^{\mathrm{b}}(X)$, let $\mathcal{W}^*=\mathcal{RH}om_{\mathcal{O}_X}(\mathcal{W}, \mathcal{O}_X)$ be the (na\"{i}ve) derived dual of $\mathcal{W}$, and let $\varphi$ be in $\mathrm{Hom}_{\mathrm{D}^{\mathrm{b}}(X)}(\mathcal{W}, \mathrm{T}^k \mathcal{V}^*[-k])$. Then the following diagram
\[
\xymatrix@C=50pt{
\mathcal{W}  \ar[rr]^-{\varphi} \ar[d]^{\mathrm{id}_{\mathcal{W}} \otimes \Delta_k} && \mathrm{T}^k \mathcal{V}^* [-k] \ar[d]^-{\Delta^{'}_k}\\
\mathcal{W} \otimes \mathrm{T}^k \mathcal{V} [k] \ar[r]^-{\varphi \otimes \mathrm{id}_{\mathrm{T}^k \mathcal{V} [k]}}& \mathrm{T}^k \mathcal{V}^* [-k] \otimes \mathrm{T}^k \mathcal{V} [k] \ar[r]& \mathcal{O}_X
}
\]
commutes in $\mathrm{D}^{\mathrm{b}}(X)$. In particular, the composition
\[
\mathcal{O}_X \xrightarrow{\mathrm{co-ev}} \mathcal{W}^* \otimes \mathcal{W} \xrightarrow{\mathrm{id}_{\mathcal{W}^*} \otimes (\Delta'_k \circ \varphi)} \mathcal{W}^*
\]
is $\varphi \circ \Delta_k$, where we implicitly use the isomorphism
\[
\mathrm{Hom}_{\mathrm{D}^{\mathrm{b}}(X)}(\mathcal{W}, \mathrm{T}^k \mathcal{V}^*[-k]) \simeq \mathrm{Hom}_{\mathrm{D}^{\mathrm{b}}(X)}(\mathrm{T}^k \mathcal{V}[k], \mathcal{W}^*).
\]
\end{lemma}

\begin{proof}
The first point follows directly from Proposition \ref{matisse} (i). The second point follows from the fact that for any $\mathcal{A}$, $\mathcal{B}$ in $\mathrm{D}^{\mathrm{b}}(X)$, the isomorphism
\[
\mathrm{Hom}_{\mathrm{D}^{\mathrm{b}}(X)}(\mathcal{A}, \mathcal{B}) \simeq \mathrm{Hom}_{\mathrm{D}^{\mathrm{b}}(X)}(\mathcal{B}^*, \mathcal{A}^*).
\]
makes the diagram
\[
\xymatrix@C=40pt{
\mathcal{A} \ar[r]^-{\mathrm{co-ev}} \ar[dd] & \mathcal{B} \otimes \mathcal{B}^* \otimes \mathcal{A} \ar[d] \\
&  \mathcal{B} \otimes \mathcal{A}^* \otimes \mathcal{A} \ar[d]^-{\mathrm{ev}} \\
\mathcal{B} \ar[r] & \mathcal{B} 
}
\]
commute. Details are left to the reader.
\end{proof}

\subsubsection{Torsion and Atiyah class}
For any locally free sheaf $\mathcal{E}$ on $S$, let $\overline{\mathcal{E}}=j^* \mathcal{E}$. The exact sequence 
\[
0 \longrightarrow \mathcal{V} \otimes  \overline{\mathcal{E}} \longrightarrow \sigma_* \, \mathcal{E} \longrightarrow \overline{\mathcal{E}} \longrightarrow 0
\]
defines a morphism $
\tau_{\mathcal{E}} \colon \overline{\mathcal{E}} \longrightarrow \mathcal{V} \otimes \overline{\mathcal{E}}\,[1]
$
in the derived category $\mathrm{D}^{\mathrm{b}}(X)$, we call it the torsion of $\mathcal{E}$. It is easy to see that the vector bundle $\mathcal{E}$ is entirely determined by the couple $(\overline{\mathcal{E}}, \tau_{\mathcal{E}})$. In particular, $\tau_{\mathcal{E}}$ vanishes if and only if $\mathcal{E}$ is isomorphic to $\sigma^* \overline{\mathcal{E}}$

\begin{proposition} \label{brique1}
For any locally free sheaf $\mathcal{E}$ on $S$, the components of the morphism
\[
\overline{\mathcal{E}} \xlongrightarrow{\mathbb{L}j^* \, \mathrm{at}_S(\mathcal{E})} \mathbb{L} j^*(\Omega^1_S \otimes \mathcal{E} [1]) \longrightarrow  j^*(\Omega^1_S \otimes \mathcal{E} [1])  \simeq \mathcal{V} \otimes \overline{\mathcal{E}}\,[1] \oplus \Omega^1_X \otimes \overline{\mathcal{E}} \,[1]
\]
are $\tau_{\mathcal{E}}$ and $\mathrm{at}_X(\overline{\mathcal{E}})$.
\end{proposition}
\begin{proof}
There is a natural morphism $j^* \mathrm{P}^1_S(\mathcal{E}) \rightarrow \mathrm{P}^1_X(\overline{\mathcal{E}})$ making the diagram
\[
\xymatrix{
0 \ar[r] & j^* \Omega^1_S \otimes \overline{\mathcal{E}} \ar[r] \ar[d] & j^*\mathrm{P}^1_S(\mathcal{E})  \ar[d] \ar[r] & \overline{\mathcal{E}} \ar[d] \ar[r] & 0 \\
0 \ar[r] & \Omega^1_X \otimes \overline{\mathcal{E}} \ar[r]  & \mathrm{P}^1_X(\overline{\mathcal{E}}) \ar[r] & \overline{\mathcal{E}} \ar[r] & 0
}
\]
commutative. On the other hand, $j^* \mathrm{P}^1_S(\mathcal{E})/ (\Omega^1_X \otimes \overline{\mathcal{E}})$ is the cokernel of the map
\[
\mathcal{V} \otimes \overline{\mathcal{E}} \longrightarrow (\mathcal{V} \otimes \overline{\mathcal{E}}) \oplus \mathcal{E}
\]
which is the identity of the first factor and the natural inclusion on the second one. Hence it is isomorphic to $\sigma_* \mathcal{E}$ (by taking the difference of the two factors). This gives a commutative diagram
\[
\xymatrix{
0 \ar[r] & j^* \Omega^1_S \otimes \overline{\mathcal{E}} \ar[r] \ar[d] & j^*\mathrm{P}^1_S(\mathcal{E}) \ar[d] \ar[r] & \overline{\mathcal{E}} \ar[d] \ar[r] & 0 \\
0 \ar[r] & \mathcal{V} \otimes\overline{\mathcal{E}} \ar[r] &  \sigma_* \,\mathcal{E} \ar[r] & \overline{\mathcal{E}} \ar[r] & 0
}
\]
and the result follows.
\end{proof}

\subsection{Sheaves on a second order thickening}

\subsubsection{Setting and cohomological invariants}
As before, let us fix a pair $(X, \mathcal{V})$ where $\mathcal{V}$ is a locally free sheaf on $X$, and let $S$ be the split first order thickening of $X$ by $V$. We are interested by ring spaces $W$ which underlying topological space $X$ satisfying the following conditions:
\begin{equation} \label{wing}
\begin{cases}
\mathcal{O}_W \,\,\textrm{is locally isomorphic to}\,\, \mathcal{O}_X \oplus \mathcal{V} \oplus \mathrm{S}^2\mathcal{V}.\\
\textrm{There exists a map} \,\, S \rightarrow W \,\, \textrm{which is locally the quotient by} \,\, \mathrm{S}^2 \mathcal{V}. 
\end{cases}
\end{equation}
Let $k \colon X \rightarrow W$ be the composite map, and let us denote by $\langle \mathcal{V} \rangle$ the ideal sheaf of $X$ in $W$, which is a sheaf of $\mathcal{O}_S$-modules. 
We can attach to $W$ two cohomology classes:
\par \medskip
\begin{enumerate}
\item[--] The class $\alpha$ in $\mathrm{Ext}^1(\mathcal{V}, \mathrm{S}^2 \mathcal{V})$ is the extension class of the exact sequence \[
0 \rightarrow \mathrm{S}^2 \mathcal{V} \rightarrow \sigma_* \langle \mathcal{V} \rangle \rightarrow \mathcal{V} \rightarrow 0
\]
\item[--] The class $\beta$ in $\mathrm{Ext}^1(\Omega^1_X, \mathrm{S}^2 \mathcal{V})$ is the obstruction of lifting $\sigma$ to $W$. To see how this class is defined, it suffices to remark that the sheaf of retractions of $k$ inducing $\sigma$ on $S$ is an affine bundle directed by the vector bundle $\mathcal{D}er(\mathcal{O}_X,  \mathrm{S}^2 \mathcal{V})$, which is $\mathcal{H}om(\Omega^1_X,  \mathrm{S}^2 \mathcal{V})$.
\end{enumerate}
\begin{lemma} \label{bricfroid}
The map $W \rightarrow (\alpha, \beta)$ is a bijection between isomorphism classes of ring spaces $W$ satisfying \eqref{wing} and $\mathrm{Ext}^1(\mathcal{V}, \mathrm{S}^2 \mathcal{V}) \oplus \mathrm{Ext}^1(\Omega^1_X, \mathrm{S}^2 \mathcal{V})$.
\end{lemma}

\begin{proof}
The subsheaf of the sheaf of automorphisms of the ringed space $\{X, \mathcal{O}_X \oplus \mathcal{V} \oplus \mathrm{S}^2 \mathcal{V}\} $ that induce the indentity morphism after taking the quotient by the square zero ideal $\mathrm{S}^2\mathcal{V}$ is isomorphic to 
$
\mathcal{D}er(\mathcal{O}_X,  \mathrm{S}^2 \mathcal{V}) \oplus \mathcal{H}om(\mathcal{V}, \mathrm{S}^2 \mathcal{V}),
$
a couple $(D, \varphi)$ corresponding to the automorphism given by the $3 \times 3$ matrix
\[
\begin{pmatrix}
\mathrm{id} & 0 & 0 \\ 
0 & \mathrm{id} & 0 \\ 
D & \varphi & \mathrm{id}
\end{pmatrix}.
\] 
This gives the required result.
\end{proof}

\subsubsection{The second order HKR class}
Let $\mathcal{E}$ be a locally free sheaf on $X$. The functor that associates to any open set of $U$ the set of locally $\mathcal{O}_W$-free extensions of $\sigma^* \mathcal{E}$ to $W$ is an abelian gerbe, whose automorphism sheaf is $\mathcal{H}om(\mathcal{E}, \mathrm{S}^2 \mathcal{V} \otimes \mathcal{E})$. Hence this gerbe is classified by a cohomology class in $\mathrm{Ext}^2(\mathcal{E}, \mathrm{S}^2 \mathcal{V} \otimes \mathcal{E})$. 
\begin{definition}
For any locally free sheaf $\mathcal{E}$ on $X$, the class of the gerbe of locally free extensions of $\sigma^* \mathcal{E}$ on $W$ is called the second order HKR class of $\mathcal{E}$, and is denoted by $\gamma_{\sigma}(\mathcal{E})$.
\end{definition}
\begin{remark}
The terminology is justified as follows: in \cite{Arinkin-Caldararu}, the authors introduce the HKR class of a vector bundle on $X$ in the case $S$ is not globally split, it mesures the obstruction to lift the bundle from $X$ to $S$. Here we are defining the same kind of obstruction classes when $S$ is trivial, but $W$ is not.
\end{remark}
As in \cite{Arinkin-Caldararu}, the class $\gamma_{\sigma}(\mathcal{E})$ can be computed explicitly: 
\begin{proposition} \label{marre}
For any locally free sheaf $\mathcal{E}$ on $X$, the class $\gamma_{\sigma}(\mathcal{E})$ is obtained (up to a nonzero scalar) as the composition
\[
\mathcal{E} \xrightarrow{\mathrm{at}_X(\mathcal{E})} \Omega^1_X \otimes \mathcal{E} [1] \xrightarrow{\beta \,\otimes \, \mathrm{id}_{\mathcal{E}}} \mathrm{S}^2 \mathcal{V} \otimes \mathcal{E}[2]
\]
\end{proposition}

\begin{proof}
For the first HKR class, this is \cite[Prop. 2.11]{Arinkin-Caldararu}. We present an alternative and somehow more down to earth proof. Let us assume that the map $X \rightarrow W$ admits two retractions $\beta_1$ and $\beta_2$, and let $D=\beta_1-\beta_2$ be the corresponding element in $\mathcal{H}om(\Omega^1_X, \mathcal{V})$. Assume also that $\mathcal{E}$ admits a regular connexion $\nabla$ on $X$. We claim that the map $\Xi \colon \beta_1^* \mathcal{E} \rightarrow \beta_2^* \mathcal{E}$ defined by
\[
\Xi(f \otimes s)= f \otimes s + f (D \otimes \mathrm{id}_{\mathcal{E}})(\nabla s)
\]
is a well-defined isomorphism. Since $\mathrm{Hom}_{\mathcal{O}_W}(\beta_1^* \mathcal{E}, \beta_2^* \mathcal{E}) \simeq \mathrm{Hom}_{\mathcal{O}_X}(\mathcal{E}, \beta_{1*} \beta_2^* \mathcal{E})$, it suffices to show that
$
\Xi(1 \otimes gs)=\beta_1^*(g)\, .\, \Xi(s)
$
for any section $g$ of $\mathcal{O}_X$. We compute:
\begin{align*}
\Xi(1 \otimes gs)&=1 \otimes gs+(D \otimes \mathrm{id}_{\mathcal{E}})(g \nabla s+dg \otimes s) \\
&= \beta_2(g) \otimes s + g (D \otimes \mathrm{id}_{\mathcal{E}})(\nabla s)+D(dg)\otimes s\\
&= \beta_1(g)\otimes s + g (D \otimes \mathrm{id}_{\mathcal{E}})(\nabla s)+ D(dg) \otimes s\\
&=\beta_1(g) . (s+(D \otimes \mathrm{id}_{\mathcal{E}})(\nabla s))\\
&=\beta_1^*(g)\, .\, \Xi(s).
\end{align*}

Let us now fix a covering $(U_i)_{i \in I}$ of $X$, and assume that on each $U_i$,
there is a retraction $\beta_i$ of the map $X \rightarrow W$ and $\mathcal{E}$ admits a regular connexion $\nabla_i$.
We put $D_{ij}=\beta_i-\beta_j$. Then we have isomorphisms
\[
\begin{cases}
\Xi^i \colon \beta_i^* \mathcal{E}_{|U_{ij}} \xrightarrow{\sim} \beta_j^* \mathcal{E}_{|U_{ij}} \,\,&\textrm{depending on}\,\,\nabla_i \\
\Xi^j \colon \beta_j^* \mathcal{E}_{|U_{jk}} \xrightarrow{\sim} \beta_k^* \mathcal{E}_{|U_{jk}} \,\,&\textrm{depending on}\,\,\nabla_j \\
\Xi^k \colon \beta_k^* \mathcal{E}_{|U_{ki}} \xrightarrow{\sim} \beta_i^* \mathcal{E}_{|U_{ki}} \,\,&\textrm{depending on}\,\,\nabla_k \\
\end{cases}
\]
The composition $\Xi^k_{|U_{ijk}} \circ \Xi^j_{|U_{ijk}} \circ \Xi^i_{|U_{ijk}}$ yields an automorphism of $\beta_i^* \mathcal{E}_{|U_{ijk}}$, which corresponds to the element
\[
\lambda=(\mathrm{D}_{ij}\otimes \mathrm{id}_{\mathcal{E}}) \circ \nabla_i + (\mathrm{D}_{jk} \otimes \mathrm{id}_{\mathcal{E}}) \circ \nabla_j +  (\mathrm{D}_{ki}\otimes \mathrm{id}_{\mathcal{E}}) \circ \nabla_k
\]
of $\Gamma(U_{ijk}, \mathcal{H}om(\mathcal{E}, \mathrm{S}^2 \mathcal{V} \otimes \mathcal{E}))$. This element is a \u{C}hech representative of the class of the gerbe of locally free extensions of $\sigma^* \mathcal{E}$ on $W$. If $c_{ij}=(\mathrm{D}_{ij}\otimes \mathrm{id}_{\mathcal{E}}) \circ (\nabla_i-\nabla_j)$, then $(c_{ij})$ defines a $1$-cochain with values in $\mathcal{H}om(\mathcal{E}, \mathrm{S}^2 \mathcal{V} \otimes \mathcal{E})$ and we have
\begin{align*}
\lambda=c_{ij}&+c_{jk}+(\mathrm{D}_{ij}\otimes \mathrm{id}_{\mathcal{E}}) \circ (\nabla_j-\nabla_k) =\frac{2}{3}(c_{ij}+c_{jk}+c_{ki})\\
&+\frac{1}{3}((\mathrm{D}_{ij}\otimes \mathrm{id}_{\mathcal{E}}) \circ (\nabla_j-\nabla_k)+(\mathrm{D}_{jk}\otimes \mathrm{id}_{\mathcal{E}}) \circ (\nabla_k-\nabla_i)+(\mathrm{D}_{kj}\otimes \mathrm{id}_{\mathcal{E}}) \circ (\nabla_i-\nabla_j))
\end{align*}
which can be split us to some nonzero constant factors as the sum of the boundary of the cochain $(c_{ij})_{i, j}$ and the Yoneda product of the $1$-cocycles $(\nabla_i-\nabla_j)_{i, j}$ and $(D_{ij}\otimes \mathrm{id}_{\mathcal{E}})_{i, j}$ that represent $\mathrm{at}_X(\mathcal{E})$ and $\beta \otimes \mathrm{id}_{\mathcal{E}}$ respectively. This gives the required formula.
\end{proof}

\subsection{Quantized cycles}

\subsubsection{Setting, and tameness condition}
Let $Y$ be a smooth $\mathbf{k}$-scheme, and let $X$ be a smooth and closed subscheme of $Y$. We denote by $i \colon X \hookrightarrow Y$ the injection of $X$ in $Y$. 
\par \medskip
Let $S$ denote the first formal neighbourhood of $X$ in $Y$. Let us assume that the closed immersion $j \colon X \hookrightarrow S$ admits a retraction $\sigma \colon S \rightarrow X$ (that is $S$ is a globally trivial square zero extension of $X$ by $\mathrm{N}^*_{X/Y}$); this is equivalent to say that the conormal sequence of the pair $(X, Y)$ splits. In this case, we say that $(X, \sigma)$ is a quantized cycle in $Y$ (\textit{see} \cite{Grivaux-HKR}). 
\par \medskip
Assume that $(X, \sigma)$ is a quantized cycle, and let $W$ be the second formal neighborhood of $X$ in $Y$, and let $k \colon X \hookrightarrow W$ be the corresponding inclusion. According to Lemma \ref{bricfroid}, the ringed space $W$ is entirely encoded by two classes $\alpha$ and $\beta$ introduced in the previous section; $\alpha$ is the extension class of the exact sequence
\[
0 \rightarrow  \sigma_* \, \frac{\mathcal{I}_X^2}{\mathcal{I}_X^3} \rightarrow \sigma_* \, \frac{\mathcal{I}_X}{\mathcal{I}_X^3} \rightarrow  \sigma_* \, \frac{\mathcal{I}_X}{\mathcal{I}_X^2} \rightarrow 0
\]
that lives in $\mathrm{Ext}^2(\mathrm{N}^*_{X/Y}, \mathrm{S}^2 \mathrm{N}^*_{X/Y})$, and $\beta$ in the class in $\mathrm{Ext}^1(\Omega^1_X, \mathrm{S}^2 \mathrm{N}^*_{X/Y})$ that measures the obstruction to the existence of a retraction of $k$ that extends $\sigma$.

\begin{definition} \label{tiny} A $(X, \sigma)$ quantized cycle $(X, \sigma)$ is \textit{tame} if the locally free sheaf $\sigma^* \mathrm{N^*_{X/Y}}$ on $S$ extends to a locally free sheaf on $W$.
\end{definition}
\begin{remark}
If $k$ admits a retraction $q \colon W \rightarrow X$ such that $q_{|S}=\sigma$, that is if $\beta$ vanishes, then $(X, \sigma)$ is automatically tame: the locally free sheaf $q^* \mathrm{N}^*_{X/Y}$ on $W$ extends $\sigma^* \mathrm{N}^*_{X/Y}$. In this case, we say that $(X, \sigma)$ is \textit{2-split}.
\end{remark}

\subsubsection{Restriction of the Atiyah class}
The aim of this section is to describe another intrinsic description of the classes $\alpha$ and $\beta$.
\begin{proposition} \label{brique2}
The torsion of the locally free sheaf $(\Omega^1_Y)_{|S}$ is the morphism
\[
\mathrm{N}^*_{X/Y} \oplus \Omega^1_X \rightarrow \mathrm{N}^*_{X/Y} \otimes (\mathrm{N}^*_{X/Y} \oplus \Omega^1_X)\, [1]
\]
given by the $2 \times 2$ matrix $\begin{pmatrix}
\alpha & \beta \\ 
0 & 0
\end{pmatrix} $.
\end{proposition}

\begin{proof}
The locally free sheaf $(\Omega^1_Y)_{|S}$ depends only on the second formal neighbourhood $W$ of $X$ in $Y$. Let us consider an automorphism of the trivial ringed space $\mathcal{O}_X \oplus \mathrm{N}^*_{X/Y} \oplus \mathrm{S}^2 \mathrm{N}^*_{X/Y}$ given by a couple $(d, \varphi)$, where $d$ is a derivation from $\mathcal{O}_X$ to $\mathrm{S}^2 \mathrm{N}^*_{X/Y}$ and $\varphi$ is a linear morphism from $\mathrm{N}^*_{X/Y}$ to $\mathrm{S}^2 \mathrm{N}^*_{X/Y}$. Assume that we are in the local situation, so that we can take coordinates: $X=U \subset \mathbf{k}^n$ and $Y=U \times V$ where $V \subset \mathbf{k}^r$. Then we can represent the morphisms $d$ and $\varphi$ by sequences of regular maps $(Z^k_{i, j}(\mathbf{x}))_{1 \leq k \leq n, 1 \leq i, j \leq r}$ and $(\Lambda^{\ell}_{i, j}(\mathbf{x}))_{1 \leq i, j, \ell \leq r}$ that are symmetric in the indices $(i, j)$. The automorphism of $W$ can be lifted to an automorphism of $Y$ given by the formula
\[
(\mathbf{x}, \mathbf{t}) \rightarrow \left( \{\mathbf{x}^k+\sum_{i, j}Z^k_{i, j}(\mathbf{x})t^i t^j)\}_{1 \leq k \leq n},  \{\mathbf{t}^{\ell}+\sum_{i, j}\Lambda^{\ell}_{i, j}(\mathbf{x})t^i t^j)\}_{1 \leq \ell \leq r}\right)
\]
The result follows by computing the pullback of the forms $dx^k, t^i dx^k$, $dt^j$, $t^i dt^j$ restricted to $S$ by the above automorphism.
\end{proof}

\begin{corollary} \label{feu}
The composition
\[
\mathrm{N}^*_{X/Y} \oplus \Omega^1_X \simeq \Omega^1_{Y|X} \xrightarrow {\mathbb{L}i^* \mathrm{at}_Y(\Omega^1_Y)} \mathrm{S}^2 \mathrm{\Omega}^1_{Y | X} [1] \simeq \mathrm{S}^2 \mathrm{N}^*_{X/Y}[1] \oplus  \Omega^1_X \otimes \mathrm{N}^*_{X/Y} [1] \oplus \mathrm{S}^2 \Omega^1_X[1]
\]
is given by the matrix 
$\begin{pmatrix}
\alpha & \beta  \\ 
\mathrm{at}_X(\mathrm{N}^*_{X/Y}) &0\\
0 &  \mathrm{at}_X(\Omega^1_X)
\end{pmatrix} $.
\end{corollary}

\begin{proof}
This is obtained by putting together Proposition \ref{brique1} and Proposition \ref{brique2}, together with the functoriality of the Atiyah class.
\end{proof}

\subsubsection{Quantized HKR isomorphism}
As noticed in \cite{Arinkin-Caldararu}, the composition
\begin{equation} \label{C}
\mathrm{S} ({\mathrm{N}_{X/Y}}[-1]) \longrightarrow \mathrm{T}  ({\mathrm{N}}_{X/Y}[-1]) \simeq \mathcal{RH}om_{\mathcal{O}_S}(\mathcal{O}_X, \mathcal{O}_X) \longrightarrow \mathcal{RH}om_{\mathcal{O}_Y}(\mathcal{O}_X, \mathcal{O}_X)
\end{equation}
is an isomorphism in $\mathrm{D}^+(Y)$,
where the first map is the antisymetrization map. For quantized cycles, it is also possible to produce a left resolution of the sheaf $\mathcal{O}_X$ (the Atiyah-Kashiwara resolution) that computes directly $\mathcal{RH}om_{\mathcal{O}_Y}(\mathcal{O}_X, \mathcal{O}_X)$, this construction is done in \cite{Grivaux-HKR}. Both constructions are in fact compatible (\emph{see} \cite[Thm. 4.13]{Grivaux-HKR}). Let us now give the corresponding HKR isomorphism. For any non-negative integer $p$, we decompose the morphism $\Delta_p$ as the sum $\Delta_p^- + \Delta_p^+$ according to the decomposition 
\[
\mathrm{T}^p (\mathrm{N}^*_{X/Y}[1]) \simeq \widetilde{\Lambda}^p (\mathrm{N}^*_{X/Y}[1]) \oplus \mathrm{S}^p (\mathrm{N}^*_{X/Y} [1]).
\]
For any sheaf $\mathcal{F}$ on $X$, the functor $\mathcal{H}om_{\mathcal{O}_Y}(\, \star \, , \mathcal{F}) \colon \mathrm{Mod}(\mathcal{O}_Y) \rightarrow \mathrm{Mod}(\mathcal{O}_Y)$ factors through a functor 
\[
\mathcal{H}om^{\ell}_{\mathcal{O}_Y}(\, \star \, , \mathcal{F})  \colon \mathrm{Mod}(\mathcal{O}_Y) \rightarrow \mathrm{Mod}(\mathcal{O}_X).
\]

\begin{proposition} \label{buche}
For any locally free sheaves $\mathcal{E}_1$, $\mathcal{E}_2$ on X, there is a canonical isomorphism
\[
\mathcal{RH}om_{\mathcal{O}_Y}^{\ell}(\mathcal{E}_1, \mathcal{E}_2) \simeq \bigoplus_{p \in \mathbb{N}} \mathcal{RH}om(\Lambda^p \mathrm{N}^*_{X/Y} \otimes \mathcal{E}_1, \mathcal{E}_2)[-p] \simeq \mathrm{S}(\mathrm{N}_{X/Y}[-1]) \, \lltens \, \mathcal{RH}om(\mathcal{E}_1, \mathcal{E}_2)
\]
in $\mathrm{D}^{\mathrm{b}}(X)$ obtained by precomposing with $ \Delta_p^+ \otimes  \mathrm{id}_{\sigma^* \mathcal{E}_1}$.
\end{proposition}
We also give Kashiwara's \textit{dual} version:
\begin{proposition} \label{buche2}
For any locally free sheaves $\mathcal{E}_1$, $\mathcal{E}_2$ on X, there is a canonical isomorphism
\[
\mathcal{RH}om_{\mathcal{O}_Y}^{r}(\mathcal{E}_1, \mathcal{E}_2) \simeq \bigoplus_{p \in \mathbb{N}} \mathcal{RH}om(\mathcal{E}_1, \Lambda^p \mathrm{N}_{X/Y} \otimes \mathcal{E}_2)[-p] \simeq \mathrm{S}(\mathrm{N}_{X/Y}[-1]) \, \lltens \, \mathcal{RH}om(\mathcal{E}_1, \mathcal{E}_2)
\]
in $\mathrm{D}^{\mathrm{b}}(X)$ obtained by postcomposing by $ \Delta_p^{'+} \otimes  \mathrm{id}_{\sigma^* \mathcal{E}_2}$.
\end{proposition}

\section{Tame quantized cycles}

\subsection{Structure constants}

\subsubsection{Definition}
We fix a quantized analytic cycle $(X, \sigma)$ in $Y$. By Proposition \ref{buche}, there exist unique coefficients $(c_p^{'k})_{0 \leq k \leq p}$ such that:
\begin{enumerate} 
\item[--] Each $c_p^{k}$ belongs to $\mathrm{Hom}\,(\mathrm{S}^k (\mathrm{N}^*_{X/Y}[1]),\mathrm{T}^p (\mathrm{N}^*_{X/Y}[1]))$, that is $\mathrm{Ext}^{p-k}(\Lambda^k \mathrm{N}^*_{X/Y}, \mathrm{T}^p \mathrm{N}^*_{X/Y})$.
\item[--] The relation $\Delta_p=\sum_{k=0}^p c^{k}_ p \circ  \Delta_k^+ $ holds in $\mathrm{D}^{\mathrm{b}}(Y)$.
\end{enumerate}
Similarly, using Proposition \ref{buche2}, we can define \textit{dual} coefficients: there exist unique coefficients $(c_p^{' k})_{0 \leq k \leq p}$ such that:
\begin{enumerate} 
\item[--] Each $c_p^{' k}$ belongs to $\mathrm{Hom}\,(\mathrm{T}^p (\mathrm{N}_{X/Y}[-1]), \mathrm{S}^k (\mathrm{N}_{X/Y}[-1]))$, that is $\mathrm{Ext}^{p-k}(\mathrm{T}^p \mathrm{N}_{X/Y}, \Lambda^k \mathrm{N}_{X/Y})$.
\item[--] The relation $\Delta'_p=\sum_{k=0}^p \Delta_k^{'+} \circ c^{' k}_ p$ holds in $\mathrm{D}^{\mathrm{b}}(Y)$.
\end{enumerate}

\begin{lemma}
Via the isomorphism $\mathrm{Ext}^{p-k}(\mathrm{T}^p \mathrm{N}_{X/Y}, \Lambda^k \mathrm{N}_{X/Y}) \simeq \mathrm{Ext}^{p-k}(\Lambda^k \mathrm{N}^*_{X/Y}, \mathrm{T}^p \mathrm{N}^*_{X/Y})$, the coefficients $c_{p}^{' k}$ and $c_p^k$ are equal.
\end{lemma}

\begin{proof}
We take the relation  $\Delta'_p=\sum_{k=0}^p \Delta_k^{'+} \circ c^{' k}_ p $, take the tensor product by $\mathrm{id}_{\mathrm{T}^p \mathrm{N}_{X/Y}[-p]}$, and precompose with the co-evaluation map $\mathcal{O}_X \rightarrow \mathrm{id}_{\mathrm{T}^p \mathrm{N}_{X/Y}[-p]} \otimes \mathrm{id}_{\mathrm{T}^p \mathrm{N}^*_{X/Y}[p]}$. Applying Lemma \ref{bienpenible} with $\mathcal{W}=\mathrm{T}^p \mathrm{N}_{X/Y}[-p]$, we get
$\Delta_p=\sum_{k=0}^p   c^{' k}_ p \circ \Delta_k^{+}$, so $ c^{' k}_ p =c_k^p$.
\end{proof}

\begin{lemma} \label{Queyras} Let $p \geq 1$.
\begin{enumerate}
\item[(i)] The coefficient $c_p^0$ vanish.
\item[(ii)] If $1 \leq k \leq p-1$, the $\mathrm{S}^p (\mathrm{N}^*_{X/Y}[1])$ component of $c_p^k$ vanishes, so $c_p^k$ factors through $\widetilde{\Lambda}^p (\mathrm{N}^*_{X/Y}[1])$.
\item[(iii)] The coefficient $c_p^p$ is the canonical inclusion of $\mathrm{S}^p (\mathrm{N}^*_{X/Y}[1])$ in $\mathrm{T}^p (\mathrm{N}^*_{X/Y}[1])$.
\end{enumerate}
\end{lemma}

\begin{proof}
For (i), we apply $\sigma_*$ to the relation defining the $c_p^k$ and use Proposition \ref{matisse} (iv). To prove (ii), let $\pi_p$ be the projection from $T^p (\mathrm{N}^*_{X/Y}[1])$ to $S^p (\mathrm{N}^*_{X/Y}[1])$. Then
\[
\Delta_p^{+}=\sum_{k=0}^p \pi_p \circ c^{k}_ p \circ  \Delta_k^+
\]
so thanks to Proposition \ref{buche}, $\pi_p \circ c^{k}_ p=0$ if $0 \leq k \leq p-1$. For (iii) this is a purely local question, and we can use the Koszul complex as in \cite{Grivaux-HKR}.
\end{proof}

\subsubsection{The decomposition lemma}
Let us first compute the first nontrivial structure constant:
\begin{proposition} \label{muzelle}
The morphism $\Delta_2^- - \alpha \circ \Delta_1$ vanishes in $\mathrm{D}^{\mathrm{b}}(W)$. In particular, $c_2^1=\alpha$. 
\end{proposition}
\begin{proof}
Let us consider the exact sequence
\begin{equation} \label{oisans}
0 \rightarrow \frac{\mathcal{I}_X^2}{\mathcal{I}_X^3} \rightarrow \frac{\mathcal{I}_X}{\mathcal{I}_X^3} \rightarrow \frac{\mathcal{I}_X}{\mathcal{I}_X^2} \rightarrow 0.
\end{equation}
It defines a morphism $f$ from $\mathrm{N}^*_{X/Y}$ to $\mathrm{S}^2 \mathrm{N}^*_{X/Y}[1]$ in $\mathrm{D}^{\mathrm{b}}(S)$. Applying Proposition \ref{hkrbabe}, $f$ can be written as $u+ v\circ \Theta_{\mathrm{N}^*_{X/Y}}$ where $u$ is in $\mathrm{Ext}^1(\mathrm{N}_{X/Y}^*, \mathrm{S}^2 \mathrm{N}^*_{X/Y})$ and $v$ is in $\mathrm{Hom}(\mathrm{T}^2 \mathrm{N}^*_{X/Y}, \mathrm{S}^2 \mathrm{N}^*_{X/Y})$.
The morphism $u$ is the image under $\sigma_*$ of \eqref{oisans}, so it is $\alpha$. The morphism $v$ can be computed locally, it is the 
the natural symetrization map from $\mathrm{T}^2 \mathrm{N}^*_{X/Y}$ to $\mathrm{S}^2 \mathrm{N}^*_{X/Y}$. 
Let us now consider the diagram
\[
\xymatrix{
&&&0 \ar[d]& \\
0 \ar[r]& \dfrac{\mathcal{I}_X^2}{\mathcal{I}_X^3} \ar[r] \ar[d]&  \dfrac{\mathcal{I}_X}{\mathcal{I}_X^3} \ar[r] \ar[d]&  \dfrac{\mathcal{I}_X}{\mathcal{I}_X^2} \ar[r] \ar[d]& 0 \\
0 \ar[r]& \dfrac{\mathcal{I}_X^2}{\mathcal{I}_X^3} \ar[r]&  \dfrac{\mathcal{O}_Y}{\mathcal{I}_X^3} \ar[r]&  \dfrac{\mathcal{O}_Y}{\mathcal{I}_X^2} \ar[r] \ar[d]& 0 \\
&&&\dfrac{\mathcal{O}_{Y}}{\mathcal{I}_X} \ar[d]& \\
&&&0&
}
\]
It gives a commutative diagram
\[
\xymatrix{
\mathcal{O}_X \ar[r]^-{\Theta_{\mathcal{O}_X}}& \mathrm{N}^*_{X/Y} [1] \ar[r]^{f} \ar[d]& \mathrm{S}^2 \mathrm{N}^*_{X/Y} [2] \ar@{=}[d] \\
 & \mathcal{O}_S [1] \ar[r] & \mathrm{S}^2\, \mathrm{N}^*_{X/Y} [2] 
 }
\]
in $\mathrm{D}^{\mathrm{b}}(W)$. Since the composite arrow from $\mathcal{O}_X$ to $\mathcal{O}_S[1]$ vanishes, the morphism $f \circ \Theta_{\mathcal{O}_X}$ also vanishes. But
$ f \circ \Theta_{\mathcal{O}_X}=(u+v \circ \Theta_{\mathrm{N}^*_{X/Y}}) \circ \Theta_{\mathcal{O}_X} 
= \alpha \circ \Delta_1 -  \Delta^-_2$. This gives the result.
\end{proof}
As a corollary, we get our key technical result:
\begin{proposition}[Decomposition lemma] \label{key} Assume that the quantized cycle $(X, \sigma)$ is tame. For any non-negative integer $i$, we can write
\[
\Theta_{\mathrm{T}^{i+1} \mathrm{N}^*_{X/Y}} \circ \Theta_{\mathrm{T}^i \mathrm{N}^*_{X/Y}}=\mathfrak{A}_i - (\alpha  \otimes \mathrm{id}_{\mathrm{T}^i \mathrm{N}^*_{X/Y}}) \circ \Theta_{\mathrm{T}^i \mathrm{N}^*_{X/Y}} + \mathfrak{R}_i
\]
in the derived category $\mathrm{D}^{\mathrm{b}}(S)$, where $\mathfrak{A}_i$ factors through $\Lambda^2 \mathrm{N}^*_{X/Y} \otimes \mathrm{T}^i  \mathrm{N}^*_{X/Y} [2]$, and $\mathfrak{R}_i$ vanishes in the derived category $\mathrm{D}^{\mathrm{b}}(W)$.
\end{proposition}

\begin{proof}
Denote by $\mu$ the natural map from $S$ to $W$, and let $\mathcal{S}$ be a locally free extension of $\sigma^* \mathrm{N}^*_{X/Y}$ on $W$. Then 
\begin{align*}
&\Theta_{\mathrm{T}^{i+1} \mathrm{N}^*_{X/Y}} \circ \Theta_{\mathrm{T}^{i} \mathrm{N}^*_{X/Y}} = \Delta_2 \,\lltens{}_{\mathcal{O}_S}\, \mathrm{id}_{\sigma^* \mathrm{T}^i \mathrm{N}^*_{X/Y}} \\ 
&= \Delta_2^+ \,\lltens{}_{\mathcal{O}_S}\,\mathrm{id}_{\sigma^* \mathrm{T}^i \mathrm{N}^*_{X/Y}} +  (\alpha \circ \Delta_1) \,\lltens{}_{\mathcal{O}_S}\, \mathrm{id}_{\sigma^* \mathrm{T}^i \mathrm{N}^*_{X/Y}} + (\Delta_2^- - \alpha \circ \Delta_1) \,\lltens{}_{\mathcal{O}_S}\, \mathrm{id}_{\sigma^* \mathrm{T}^i \mathrm{N}^*_{X/Y}}.
\end{align*}
We claim that $(\Delta_2^- - \alpha \circ \Delta_1) \,\lltens{}_{\mathcal{O}_S}\, \mathrm{id}_{\sigma^* \mathrm{T}^i \mathrm{N}^*_{X/Y}}$ vanishes in $\mathrm{D}^{\mathrm{b}}(W)$. Indeed,
\begin{align*}
\mu_* \left( (\Delta_2^- - \alpha \circ \Delta_1) \,\lltens{}_{\mathcal{O}_S}\, \mathrm{id}_{\sigma^* \mathrm{T}^i \mathrm{N}^*_{X/Y}} \right) 
&=  \mu_* \left( (\Delta_2^- - \alpha \circ \Delta_1) \,\lltens{}_{\mathcal{O}_S}\, \mathbb{L} \mu^* \mathrm{id}_{\mathrm{T}^i \mathcal{S}} \right) \\
&=  \mu_* \left( \Delta_2^- - \alpha \circ \Delta_1 \right) \,\lltens{}_{\mathcal{O}_W}\, \mathrm{id}_{\mathrm{T}^i \mathcal{S}}.
\end{align*}
\end{proof}
\begin{remark} \raisebox{0.08cm}{\danger}
Proposition \ref{key} is false in general when the quantized cycle $(X, \sigma)$ is not tame: the morphism 
$ (\Delta_2^- - \alpha \circ \Delta_1) \,\lltens{}_{\mathcal{O}_S}\, \mathrm{id}_{\sigma^* \mathrm{T}^i \mathrm{N}^*_{X/Y}}$ can be nonzero in $\mathrm{D}^{\mathrm{b}}(W)$.
\end{remark} 
For $1 \leq p \leq p-1$, let
\[
\Psi_p \colon \mathrm{T}^p (\mathrm{N}^*_{X/Y}[1]) \rightarrow \bigoplus_{i=1}^{p-1} \mathrm{T}^{i-1}  (\mathrm{N}^*_{X/Y}[1]) \otimes 
\Lambda^2 (\mathrm{N}^*_{X/Y}[1]) \otimes  \mathrm{T}^{p-i-1}(\mathrm{N}^*_{X/Y}[1])
\]
be the map obtained by (graded) antisymetrization of two consecutive factors. We also denote by $\chi_{p, i}$ the components of $\Psi_p$.
\begin{corollary} \label{sky} Assume to be given a tame cycle $(X, \sigma)$ in $Y$. Then
\[
\Psi_{p} \circ c_{p}^k=\{ (\mathrm{id}_{\mathrm{T}^{i-1} \mathrm{N}^*_{X/Y}} \otimes \alpha \otimes \mathrm{id}_{\mathrm{T}^{p-i-1} \mathrm{N}^*_{X/Y}}) \circ c_{p-1}^k \}_{1 \leq i \leq p-1} \,\,\, \textrm{if} \,\,\, 1 \leq  k \leq p-1.
\]
\end{corollary} 

\begin{proof}
Using Proposition \ref{key}, we have in $\mathrm{D}^{\mathrm{b}}(S)$
\begin{align*}
\Delta_p&=(-1)^p \, \Theta_{\mathrm{T}^{p-1} \mathrm{N}^*_{X/Y}} \circ \cdots \circ   (\Theta_{\mathrm{T}^{i+1} \mathrm{N}^*_{X/Y}} \circ  \Theta_{\mathrm{T}^{i} \mathrm{N}^*_{X/Y}})  \circ \cdots \circ  \Theta_{\mathcal{O}_X} \\
&=(\mathrm{id}_{\mathrm{T}^{p-i-1} \mathrm{N}^*_{X/Y}} \otimes \alpha \otimes \mathrm{id}_{\mathrm{T}^i \mathrm{N}^*_{X/Y}}) \circ \Delta_{p-1}+(-1)^p \, \Theta_{\mathrm{T}^{p-1} \mathrm{N}^*_{X/Y}} \circ \cdots \circ (\mathfrak{U}_i+\mathfrak{R}_i) \circ \cdots \circ  \Theta_{\mathcal{O}_X}
\end{align*}
It implies that in $\mathrm{D}^{\mathrm{b}}(W)$, 
\begin{align*}
\chi_{p,i} \circ \Delta_p&=(\mathrm{id}_{\mathrm{T}^{p-i-1} \mathrm{N}^*_{X/Y}} \otimes \alpha \otimes \mathrm{id}_{\mathrm{T}^i \mathrm{N}^*_{X/Y}}) \circ \Delta_{p-1} \\
&=\sum_{k=1}^{p-1} (\mathrm{id}_{\mathrm{T}^{p-i-1} \mathrm{N}^*_{X/Y}} \otimes \alpha \otimes \mathrm{id}_{\mathrm{T}^i \mathrm{N}^*_{X/Y}}) \circ   c_{p-1}^k  \circ \Delta^+_k. 
\end{align*}
On the other hand
\[
\chi_{p,i} \circ \Delta_p=\sum_{k=1}^{p-1} \chi_{p,i} \circ c_p^k \circ \Delta^+_k.
\]
We conclude using Proposition \ref{buche}.
\end{proof}

\begin{corollary} \label{sky2}
If $(X, \sigma)$ is tame, $\Delta_p=\sum_{k=1}^p c_p^k \circ \Delta^+_k$ in $\mathrm{D}^{\mathrm{b}}(W)$.
\end{corollary}

\begin{proof}
We argue by induction on $p$. For $p=2$, this is Proposition \ref{muzelle}. Assume that the property holds at the order $p-1$. The morphism 
$\Delta_p-\sum_{k=1}^p c_p^k \circ \Delta^+_k$ takes values in $\widetilde{\Lambda}^p (\mathrm{N}^*_{X/Y}[1])$. Hence it suffices to prove that
$\Psi_p \circ (\Delta_p-\sum_{k=1}^p c_p^k \circ \Delta^+_k)$ vanishes in $\mathrm{D}^{\mathrm{b}}(W)$. Using the notation of the preceding proof, 
\[
\chi_{p, i} \circ (\Delta_p-\sum_{k=1}^p c_p^k \circ \Delta^+_k) =(\mathrm{id}_{\mathrm{T}^{p-i-1} \mathrm{N}^*_{X/Y}} \otimes \alpha \otimes \mathrm{id}_{\mathrm{T}^i \mathrm{N}^*_{X/Y}}) \circ \left(\Delta_{p-1}-\sum_{k=1}^{p-1} c_{p-1}^k \circ \Delta^+_k \right)
\]
in $\mathrm{D}^{\mathrm{b}}(W)$. Then we use the induction hypothesis.
\end{proof}

\subsubsection{Main result}
We can come to our main result.
\begin{theorem} \label{zeboss}
Let $(X, \sigma)$ be a tame quantized cycle in $Y$.
The class $\alpha$ defines a Lie coalgebra structure on $\mathrm{N}^*_{X/Y}[1]$, hence a Lie algebra structure on $\mathrm{N}_{X/Y}[-1]$. Besides, the objects $\mathcal{RH}om^{\ell}_{\mathcal{O}_Y}(\mathcal{O}_X, \mathcal{O}_X)$ and $\mathcal{RH}om^{r}_{\mathcal{O}_Y}(\mathcal{O}_X, \mathcal{O}_X)$ are naturally algebra objects in the derived category $\mathrm{D}^{\mathrm{b}}(X)$, and there are commutative diagrams
\[
\xymatrix{
\sigma_* \mathcal{RH}om_{\mathcal{O}_S}(\mathcal{O}_X, \mathcal{O}_X) \ar[r] \ar[dd]_-{\mathrm{HKR}} & \mathcal{RH}om^{\ell}_{\mathcal{O}_Y}(\mathcal{O}_X, \mathcal{O}_X) \ar[d]^-{\mathrm{HKR}} \\
&  \mathrm{S}(\mathrm{N}_{X/Y}[-1]) \ar[d]^-{\mathrm{PBW}} \\
\mathrm{T}(\mathrm{N}_{X/Y}[-1]) \ar[r]& \mathrm{U}(\mathrm{N}_{X/Y}[-1])
}
\] 
and
\[
\xymatrix{
\sigma_* \mathcal{RH}om_{\mathcal{O}_S}(\mathcal{O}_X, \mathcal{O}_X) \ar[r] \ar[dd]_-{\mathrm{dual \,\, HKR}} & \mathcal{RH}om^{r}_{\mathcal{O}_Y}(\mathcal{O}_X, \mathcal{O}_X) \ar[d]^-{\mathrm{dual \, \, HKR}} \\
&  \mathrm{S}(\mathrm{N}_{X/Y}[-1]) \ar[d]^-{\mathrm{PBW}} \\
\mathrm{T}(\mathrm{N}_{X/Y}[-1]) \ar[r]& \mathrm{U}(\mathrm{N}_{X/Y}[-1])
}
\] 
where all horizontal arrows are algebra morphisms.
\end{theorem}

\begin{proof}[Strategy of proof]
Let us first discuss the statement, as well as the main points involved in the proof.
\begin{enumerate}
\item[--] The object $\mathcal{RH}om_{\mathcal{O}_Y}(\mathcal{O}_X, \mathcal{O}_X)$ is always an algebra object in the category $\mathrm{D}^{\mathrm{b}}(Y)$, but this structure doesn't always come in full generality from an algebra object structure on $\mathcal{RH}om^{r}_{\mathcal{O}_Y}(\mathcal{O}_X, \mathcal{O}_X)$ or on $\mathcal{RH}om^{ \ell}_{\mathcal{O}_Y}(\mathcal{O}_X, \mathcal{O}_X)$ in $\mathrm{D}^{+}(X)$. Indeed, if $\mathcal{RH}om^{ \ell}_{\mathcal{O}_Y}(\mathcal{O}_X, \mathcal{O}_X)$ has such an algebra structure, the natural morphism
\[
\mathrm{N}_{X/Y}[-1] \rightarrow \mathcal{RH}om^{ \ell}_{\mathcal{O}_Y} (\mathrm{N}^*_{X/Y}[1], \mathcal{O}_X) \rightarrow \mathcal{RH}om^{ \ell}_{\mathcal{O}_Y}(\mathcal{O}_X, \mathcal{O}_X)
\]
obtained by precomposition with $\Delta_1$ yields a morphism $\mathrm{T}(\mathrm{N}_{X/Y}[-1]) \rightarrow \mathcal{RH}om^{ \ell}_{\mathcal{O}_Y}(\mathcal{O}_X, \mathcal{O}_X)$, and the composite morphism
\[
\mathrm{S}(\mathrm{N}_{X/Y}[-1]) \hookrightarrow \mathrm{T}(\mathrm{N}_{X/Y}[-1]) \rightarrow \mathcal{RH}om^{ \ell}_{\mathcal{O}_Y}(\mathcal{O}_X, \mathcal{O}_X)
\]
is an isomorphism in $\mathrm{D}^{+}(X)$, which is not always the case.
\vspace{0.2cm}
\item[--] 
Assuming that we have an algebra structure on $\mathcal{RH}om^{r}_{\mathcal{O}_Y}(\mathcal{O}_X, \mathcal{O}_X)$ and on $\mathcal{RH}om^{ \ell}_{\mathcal{O}_Y}(\mathcal{O}_X, \mathcal{O}_X)$ making the top row of each diagram if Theorem \ref{zeboss} multiplicative, the statement follows directly from Proposition \ref{muzelle} the reverse PBW Theorem (that is Theorem \ref{coupdepute}), without using the tameness condition at all: indeed, condition (A1) is \eqref{C}, and condition (A2) is Lemma \ref{Queyras}.
\end{enumerate}
Therefore, the main difficulty lies in the construction of this algebra structure on $\mathcal{RH}om^{r}_{\mathcal{O}_Y}(\mathcal{O}_X, \mathcal{O}_X)$ and on $\mathcal{RH}om^{ \ell}_{\mathcal{O}_Y}(\mathcal{O}_X, \mathcal{O}_X)$. We will provide three different proofs corresponding to different geometric contexts:
\par \bigskip
\fbox{\textbf{Case A: $\infty$-split}}\, Assume that $X$ admits a global retract in $Y$ that lifts $\sigma$. This is the easiest case, but it covers the case of the diagonal injection and is therefore sufficient to prove the results of Kapranov, Markarian and Ramadoss that are presented in the next section. The reader only interested in this can skip the two other cases.
\par \bigskip
\fbox{\textbf{Case B: 2-split}}\, Assume that $X$ admits a retract at order two that lifts $\sigma $ (that is $\beta=0$). In this case, we can adapt the former proof to the second formal neighborhood, using the second Corollary of the decomposition Lemma (Corollary \ref{sky2}).
\par \bigskip
\fbox{\textbf{Case C: tame}}\, The general case: $(X, \sigma)$ is tame. This requires the first Corollary of the decomposition Lemma (Corollary \ref{sky}) as well as the full strength of the categorical PBW Theorem (Theorem \ref{saroumane}).
\end{proof}

\begin{proof}[Proof of Theorem \ref{zeboss}]
We follow the aforementioned plan of proof, and discuss successively the three cases A, B and C. We recall the following notation:
\[
\xymatrix{
X \ar[r]_-{j}  \ar[rd]_-{k} \ar[d]_-{i} & S \ar@/_/[l]_-{\sigma}  \ar[d] \\
Y & \ar[l] W
}
\]
\par \medskip
\fbox{\textbf{Case A}}\, 
Let $f \colon Y \rightarrow X$ a retraction of $X$ in $Y$. There is a natural morphism of algebra objects 
\[
j_{S/Y*}\mathcal{RH}om_{\mathcal{O}_S}(\mathcal{O}_X, \mathcal{O}_X) \rightarrow \mathcal{RH}om_{\mathcal{O}_Y}(\mathcal{O}_X, \mathcal{O}_X)
\]
in $\mathrm{D}^{+}(Y)$, which gives a morphism of algebra objects
\begin{equation} \label{cristillan}
f_* j_{S/Y*}\mathcal{RH}om_{\mathcal{O}_S}(\mathcal{O}_X, \mathcal{O}_X) \rightarrow f_*\mathcal{RH}om_{\mathcal{O}_Y}(\mathcal{O}_X, \mathcal{O}_X)
\end{equation}
since $f_*$ is monoidal. Now it suffices to remark that 
\[
f_*\mathcal{RH}om_{\mathcal{O}_Y}(\mathcal{O}_X, \mathcal{O}_X) \simeq f_* i_* \mathcal{RH}om^{\ell}_{\mathcal{O}_Y}(\mathcal{O}_X, \mathcal{O}_X) \simeq \mathcal{RH}om^{\ell}_{\mathcal{O}_Y}(\mathcal{O}_X, \mathcal{O}_X),
\]
so that $\mathcal{RH}om^{\ell}_{\mathcal{O}_Y}(\mathcal{O}_X, \mathcal{O}_X)$ inherits naturally from an algebra structure and the morphism \eqref{cristillan} becomes an algebra morphism
\[
\sigma_* \mathcal{RH}om_{\mathcal{O}_S}(\mathcal{O}_X, \mathcal{O}_X) \rightarrow \mathcal{RH}om^{\ell}_{\mathcal{O}_Y}(\mathcal{O}_X, \mathcal{O}_X).
\]
The same trick works for the functor $\mathcal{RH}om^{r}$.  This settles Case A.
\par \bigskip
\fbox{\textbf{Case B}}\, Let us consider the map $\mathfrak{p} \colon \sigma_* \mathcal{RH}om_{\mathcal{O}_S}(\mathcal{O}_X, \mathcal{O}_X) \rightarrow \mathcal{RH}om_{\mathcal{O}_X}^{\ell}(\mathcal{O}_X, \mathcal{O}_X)$. It admits a section, given by the composition
\[
\mathcal{RH}om_{\mathcal{O}_X}^{\ell}(\mathcal{O}_X, \mathcal{O}_X) \simeq \mathrm{S}(\mathrm{N}^*_{X/Y}[1]) \hookrightarrow \mathrm{T}(\mathrm{N}_{X/Y}[-1]) \simeq \sigma_* \mathcal{RH}om_{\mathcal{O}_S}(\mathcal{O}_X, \mathcal{O}_X).
\]
Hence $\mathfrak{p}$ admits a kernel $\mathscr{K}$, which can be explicitly described as follows: $\mathscr{K}$ is isomorphic to $\widetilde{\Lambda} (\mathrm{N}_{X/Y}[-1])$, and the (split) embedding of $\mathscr{K}$ in $\sigma_* \mathcal{RH}om_{\mathcal{O}_S}(\mathcal{O}_X, \mathcal{O}_X)$
is obtained by applying $\sigma_*$ to the composition
\[
\bigoplus_{p \geq 0} j_* \widetilde{\Lambda}^p (\mathrm{N}_{X/Y}[-1]) \rightarrow \bigoplus_{p \geq 0} \mathcal{RH}om_{\mathcal{O}_S}(\widetilde{\Lambda}^p (\mathrm{N}^*_{X/Y}[1]), \mathcal{O}_X) \rightarrow  \mathcal{RH}om_{\mathcal{O}_S}(\mathcal{O}_X, \mathcal{O}_X)
\]
where the last map is obtained componentwise by precomposing with $\Delta_p-\sum_{k=1}^p c_p^k \Delta_k^+$, which is a morphism in $\mathrm{Hom}_{\mathrm{D}^{\mathrm{b}}(S)}(\mathcal{O}_X, \widetilde{\Lambda}^p (\mathrm{N}^*_{X/Y}[1])$. Assume now that there exists a retraction $q \colon W \rightarrow X$ that extends the first order retraction $\sigma$. Then the composition
\[
\mathscr{K} \rightarrow \sigma_* \mathcal{RH}om_{\mathcal{O}_S}(\mathcal{O}_X, \mathcal{O}_X) \rightarrow q_* \mathcal{RH}om_{\mathcal{O}_W}(\mathcal{O}_X, \mathcal{O}_X)
\]
is obtained by applying $q_*$ to the map 
\[
\bigoplus_{p \geq 0} j_{X/W*}\, \widetilde{\Lambda}^p (\mathrm{N}_{X/Y}[-1]) \rightarrow \bigoplus_{p \geq 0} \mathcal{RH}om_{\mathcal{O}_W}(\widetilde{\Lambda}^p (\mathrm{N}^*_{X/Y}[1]), \mathcal{O}_X) \rightarrow  \mathcal{RH}om_{\mathcal{O}_W}(\mathcal{O}_X, \mathcal{O}_X).
\]
But this map is identically zero, since according to proposition \ref{sky2}, $\Delta_p-\sum_{k=1}^p c_p^k \Delta_k^+$ vanishes in the derived category $\mathrm{D}^{\mathrm{b}}(W)$. We can now conclude: the map
\[
\sigma_* \mathcal{RH}om_{\mathcal{O}_S}(\mathcal{O}_X, \mathcal{O}_X) \rightarrow q_* \mathcal{RH}om_{\mathcal{O}_W}(\mathcal{O}_X, \mathcal{O}_X)
\]
is a morphism of algebra objects and $\mathscr{K}$ is a split sub-object that maps to zero, so that the composition
\[
\mathscr{K} \,\lltens{}_{\mathcal{O}_X}\, \sigma_* \mathcal{RH}om_{\mathcal{O}_S}(\mathcal{O}_X, \mathcal{O}_X) \rightarrow  \sigma_* \mathcal{RH}om_{\mathcal{O}_S}(\mathcal{O}_X, \mathcal{O}_X) \rightarrow q_*  \mathcal{RH}om_{\mathcal{O}_W}(\mathcal{O}_X, \mathcal{O}_X)
\]
is zero. It follows that the composition
\[
\mathscr{K} \,\lltens{}_{\mathcal{O}_X}\, \sigma_* \mathcal{RH}om_{\mathcal{O}_S}(\mathcal{O}_X, \mathcal{O}_X) \rightarrow  \sigma_* \mathcal{RH}om_{\mathcal{O}_S}(\mathcal{O}_X, \mathcal{O}_X) \rightarrow  \mathcal{RH}om_{\mathcal{O}_Y}^{\ell}(\mathcal{O}_X, \mathcal{O}_X)
\]
is zero. Therefore, in the decomposition
\[
\sigma_* \mathcal{RH}om_{\mathcal{O}_S}(\mathcal{O}_X, \mathcal{O}_X) \simeq \mathscr{K} \oplus  \mathcal{RH}om_{\mathcal{O}_Y}^{\ell}(\mathcal{O}_X, \mathcal{O}_X),
\]
the object $\mathscr{K}$ is an ideal object, so that $\mathcal{RH}om_{\mathcal{O}_Y}^{\ell}(\mathcal{O}_X, \mathcal{O}_X)$ inherits of a natural algebra structure, for which $\mathfrak{p}$ is a multiplicative morphism. This finishes the proof. Again, the whole proof works in the same way for the functor $\mathcal{RH}om_{\mathcal{O}_Y}^r$.
\par \bigskip
\fbox{\textbf{Case C}}\, If we consider $\alpha$ as a morphism from $\mathrm{\Lambda}^2 (\mathrm{N}^*_{X/Y}[1])$ to $\mathrm{N}^*_{X/Y}[1]$ in the opposite category of $\mathrm{D}^{\mathrm{b}}(X)$, we notice that the induction relations provided by Corollary \ref{sky} are exactly 
the same as the ones proved in Proposition \ref{chalvet} (this is why we took the same notation $c_p^k$). Hence $(\mathrm{N}^*_{X/Y}[1], \alpha)$ is a Lie algebra object in the opposite derived category of $X$.
Now according to the second part of the categorical PBW Theorem (Theorem \ref{saroumane}), we can define an algebra structure on $\mathrm{S}(\mathrm{N}^*_{X/Y}[1])$ using the coefficients $c_p^k$ and there is a natural multiplicative morphism from $\mathrm{T}(\mathrm{N}^*_{X/Y}[1])$ to $\mathrm{S}(\mathrm{N}^*_{X/Y}[1])$ endowed with this structure. To conclude, it suffices to notice that the following diagram
\[
\xymatrix{
\sigma_* \mathcal{RH}om_{\mathcal{O}_S}(\mathcal{O}_X, \mathcal{O}_X) \ar[r] \ar[d]^-{\sim} & \mathcal{RH}om^{\ell}_{\mathcal{O}_Y}(\mathcal{O}_X, \mathcal{O}_X) \ar[d]^-{\sim} \\
\mathrm{T}(\mathrm{N}^*_{X/Y}[1]) \ar[r] & \mathrm{S}(\mathrm{N}^*_{X/Y}[1])
}
\]
is commutative, which is nothing but the fact that the ``geometric'' coefficients $c_{p}^k$ are the same as the ``algebraic'' coefficients $c_p^k$, that is Corollary \ref{sky}.
\end{proof}

\subsection{The results of Kapranov, Markarian, Ramadoss and Yu}

In this section, we provide a new light on the foundational result in this theory: the construction of the Lie algebra structure on $\mathrm{T}_X[-1]$, due to Kapranov \cite{Kapranov} and Markarian \cite{Markarian}, and the computation of its universal enveloping algebra, due to Markarian \cite{Markarian}, and Ramadoss \cite{Ramadoss}. Then we prove Ramadoss formula \cite{Ramadoss2} that computes the big Chern classes of a vector bundle introduced by Kapranov in \cite{Kapranov}.
\subsubsection{The Lie algebra  \texorpdfstring{$\mathrm{T}_X[-1]$}{TX[-1]}}
Given a smooth scheme/manifold $X$, we consider the special case of the quantized cycle $(\Delta_X, \mathrm{pr}_1)$ in the product $X \times X$.
\begin{theorem}[\cite{Kapranov}, \cite{Markarian}, \cite{Ramadoss}] \label{silex}$ $
\begin{enumerate}
\item[--] The object $(\Omega^1_X[1], \mathrm{at}_{\Omega^1_X[1]})$ is a Lie coalgebra in $\mathrm{D}^{\mathrm{b}}(X)$. 
\vspace{0.1cm}
\item[--] The ring object $\mathrm{pr}_{1*} \mathcal{RH}om_{\mathcal{O}_{X \times X}}(\mathcal{O}_X, \mathcal{O}_X)$ is isomorphic to $\mathrm{U}(\mathrm{T}_X[-1])$.
\vspace{0.1cm}
\item[--] The map $\,\mathrm{pr}_{1*} \mathcal{RH}om_{\mathcal{O}_{\overline{X}}}(\mathcal{O}_X, \mathcal{O}_X) \rightarrow \mathrm{pr}_{1*} \mathcal{RH}om_{\mathcal{O}_{X \times X}}(\mathcal{O}_X, \mathcal{O}_X)$ identifies with the natural projection map $\mathrm{T}(\mathrm{T}_X[-1]) \rightarrow \mathrm{U}(\mathrm{T}_X[-1])$.
\end{enumerate}
\end{theorem}
\begin{proof}
The only thing we must prove is that $\alpha$ identifies with the Atiyah class of $\Omega^1_X$, which is well-known: this follows from looking at the diagram
\[
\xymatrix{
0 \ar[r] & \mathrm{pr}_{1*} \dfrac{\mathcal{I}_{\Delta_X}^2}{\mathcal{I}_{\Delta_X}^3} \ar[r] \ar[d] &  \mathrm{pr}_{1*}\dfrac{\mathcal{I}_{\Delta_X}}{\mathcal{I}_{\Delta_X}^3} \ar[r] \ar[d]& \mathrm{pr}_{1*} \dfrac{\mathcal{I}_{\Delta_X}}{\mathcal{I}_{\Delta_X}^2} \ar[r] \ar[d]& 0 \\
0 \ar[r]& \mathrm{pr}_{1*}\left(\dfrac{\mathcal{I}_{\Delta_X}}{\mathcal{I}_{\Delta_X}^2} \otimes \mathrm{pr}_2^* \Omega^1_X\right) \ar[r] &  \mathrm{pr}_{1*}\left(\dfrac{\mathcal{O}_{X \times X}}{\mathcal{I}_{\Delta_X}^2} \otimes \mathrm{pr}_2^* \Omega^1_X\right) \ar[r] &  \mathrm{pr}_{1*}\left(\dfrac{\mathcal{O}_{X \times X}}{\mathcal{I}_{\Delta_X}} \otimes \mathrm{pr}_2^* \Omega^1_X\right) \ar[r]&0
}
\]
where the vertical maps are given by differentation with respect to the second variable (so that they are all linear). The right vertical map is 
the isomorphism given by the quantization $\mathrm{pr}_1$, and the left vertical map is the symetrization morphism. 
\end{proof}

\begin{remark}
Any object $\mathcal{F}$ in $\mathrm{D}^{\mathrm{b}}(X)$ defines a representation of the Lie algebra $\mathrm{T}_X[-1]$, which is obtained by the chain of morphisms
\[
\mathrm{T}_X[-1] \otimes \mathcal{F} \xrightarrow{\mathrm{id} \otimes \mathrm{at}_{\mathcal{F}}} \mathrm{T}_X[-1] \otimes \Omega^ 1_X[1] \otimes \mathcal{F} \xrightarrow{\mathrm{ev} \otimes \mathrm{id}} \mathcal{F} 
\]
For any $\mathcal{F}$, $\mathcal{G}$ in $\mathrm{D}^{\mathrm{b}}(X)$ viewed as representations of $\mathrm{T}_X[-1]$, the naturality of Atiyah classes implies that
\[
\mathrm{Hom}_{\mathrm{Rep}\left(\mathrm{T}_X[-1] \right)}(\mathcal{F}, \mathcal{G})=\mathrm{Hom}_{\mathrm{D}^{\mathrm{b}}(X)}(\mathcal{F}, \mathcal{G}).
\]
In other words, $\mathrm{D}^{\mathrm{b}}(X)$ embeds as a full subcategory of $\mathrm{Rep}\left(\mathrm{T}_X[-1] \right)$.
\end{remark}
Theorem \ref{silex} allows to give a Lie-theoretic interpretation of the tameness condition (Definition \ref{tiny}) for quantized cycles. To a quantized cycle $(X, \sigma)$ in $Y$, we have an exact sequence
\[
0 \rightarrow \mathrm{T}X[-1] \rightarrow \mathrm{T}Y[-1]_{| X} \rightarrow \mathrm{N}_{X/Y}[-1] \rightarrow 0
\]
and $\sigma$ provides a splitting of this sequence.
\begin{theorem}
The triplet $(\mathrm{T}X[-1], \mathrm{T}Y[-1]_{| X}, \mathrm{N}_{X/Y}[-1])$ is a reductive pair of Lie objects, as defined in \S \ref{mignon}. Besides this pair is tame in the sense of Definition \ref{defalg} if and only if $(X, \sigma)$ is tame in the sense of Definition \ref{tiny}. In this case, the dual of $\alpha$ defines a Lie structure on $\mathrm{N}_{X/Y}[-1]$.
\end{theorem}

\begin{proof}
We switch from Lie algebras objects to Lie coalgebras objects, so that we see $\Omega^1_{Y|X}[1]$ as a Lie coalgebra in $\mathrm{D}^{\mathrm{b}}(X)$. This Lie coalgebra is described by Corollary \ref{feu}: first the diagram
\[
\xymatrix{
\Omega_{Y|X}^1[1] \ar[r] \ar[d]& \Omega_{Y|X}^1[1] \otimes \Omega_{Y|X}^1[1] \ar[r]& \Omega_{X}^1[1] \otimes \Omega_{X}^1[1] \ar@{=}[d] \\
\Omega_{X}^1[1] \ar[rr] &&  \Omega_{X}^1[1] \otimes \Omega_{X}^1[1]
}
\]
commutes, so the morphism $\Omega^1_{Y|X} \rightarrow \Omega^1_X$ is a morphism of Lie coalgebras objects. Next the tameness of the pair (in the sens of Definition \ref{defalg}) can be explicited as follows: if consider the morphism
\[
\Omega^1_X[1] \rightarrow \Omega^1_{Y|X}[1] \rightarrow \Omega^1_{Y|X}[1] \otimes \Omega^1_{Y|X}[1] \rightarrow \mathrm{N}^*_{X/Y}[1] \otimes \mathrm{N}^*_{X/Y}[1],
\]
tameness means the vanishing of the composite morphism
\[
\mathrm{N}^*_{X/Y}[1] \rightarrow \Omega^1_X[1] \otimes \mathrm{N}^*_{X/Y}[1] \rightarrow ( \mathrm{N}^*_{X/Y}[1] \otimes \mathrm{N}^*_{X/Y}[1]) \otimes \mathrm{N}^*_{X/Y}[1].
\]
The first morphism is the class $\beta$, and the second one is $(\beta \otimes \mathrm{id}_{\mathrm{N}^*_{X/Y}[1]}) \circ \mathrm{at}_{\mathrm{N}^*_{X/Y}[1]}$. Thanks to Proposition \ref{marre}, this is $\gamma_{\sigma}(\mathrm{N}^*_{X/Y})$, and we are done. For the last point, we remark that the composition
\[
\mathrm{N}^*_{X/Y}[1] \rightarrow \Omega^1_{Y|X}[1] \rightarrow \Omega^1_{Y|X}[1] \otimes \Omega^1_{Y|X}[1] \rightarrow \mathrm{N}^*_{X/Y}[1] \otimes \mathrm{N}^*_{X/Y}[1]
\]
is exactly $\alpha$.
\end{proof}

\subsubsection{Big Chern classes}
Let us recall that for any vector bundle $E$ on $X$, the big Chern classes $\hat{c}_p(E)$ of $E$ live in $\mathrm{H}^p(X, \mathrm{T}^p \Omega^1_X)$, they are obtained by composing the morphisms $\mathrm{at}_E, \mathrm{id}_{\Omega^1_X} \otimes \mathrm{at}_E, \dots, \mathrm{id}_{\Omega^{p-1}_X} \otimes \mathrm{at}_E$ and then taking the trace on $E$ (without antisymmetrizing on the factor $\mathrm{T}^p \Omega^1_X$).

\begin{theorem}
For any vector bundle $E$ on $X$, we have $\hat{c}_p(E)=\sum_{p=1}^k {c_p^k}(X) \circ c_k(E) $ where $c_p^k$ are the universal elements in $\mathrm{Ext}^{p-k}(\Omega^k_ X, \mathrm{T}_X^{\otimes p})$ associated to the Lie algebra $\mathrm{T}_X[-1]$.
\end{theorem}
\begin{proof}
On $X \times X$, we have $\Delta_p=\sum_{k=1}^p c_p^k \circ \Delta_k^+$. This gives 
\[
\mathrm{pr}_{2*} (\Delta_p \otimes \mathrm{id}_{\mathrm{pr_1^* E}})=\sum_{k=1}^p (c_p^k\otimes \mathrm{id}_{E})\circ  \mathrm{pr}_{2*} (\Delta_k^+     \otimes \mathrm{id}_{\mathrm{pr_1^* E}}).
\]
The result follows by taking the trace on $E$.
\end{proof}
\begin{remark}
As explained in \cite{Ramadoss2}, the total Chern class of a vector bundle $E$ can be interpreted in terms of the representation of the Lie algebra $\mathrm{T}_X[-1]$. Indeed $E$ defines a representation of $\mathrm{T}_X[-1]$, whith is (the dual of) the Atiyah class of $E$, hence a map from $\mathrm{U}(\mathrm{T}_X[-1])$ to $\mathrm{End}\,(E)$. Its trace defines a map from $\mathrm{U}(\mathrm{T}_X[-1])$ to $\mathcal{O}_X$, which is exactly $\sum_p c_p(E)$ via the isomorphism
\[
\mathrm{Hom}_{\mathrm{D}^{\mathrm{b}}(X)}(\mathrm{U}(\mathrm{T}_X[-1]), \mathcal{O}_X) \simeq \bigoplus_{p} \mathrm{H}^p(X, \Omega^p_X).
\]

\end{remark}

\subsubsection{The quantized cycle class}
Let us recall the definition of the quantized cycle class introduced in \cite{Grivaux-HKR}. For any quantized analytic cycle $(X, \sigma)$ in $Y$, we consider the composition

\begin{equation} \label{shuriken}
\omega_{X/Y} \simeq \mathcal{RH}om^{\textrm{r}}(\mathcal{O}_X, \mathcal{O}_Y) \rightarrow \mathcal{RH}om^{\textrm{r}}_{\mathcal{O}_Y}(\mathcal{O}_X, \mathcal{O}_X) \simeq \mathrm{S}(\mathrm{N}_{X/Y}[-1])
\end{equation}
where the last isomorphism is the dual HKR isomorphism. Let $d$ be the codimension of $X$ in $Y$.
\begin{proposition} \label{zurich}
Assume that $(X, \sigma)$ is tame. Then the morphism \eqref{shuriken} is a $d$-torsion morphism for the Lie algebra $\mathrm{N}_{X/Y}[-1]$.
\end{proposition}

\begin{proof}
We must prove that the composition
\[
\mathrm{N}_{X/Y} [-1] \otimes \omega_{X/Y} \rightarrow \mathrm{N}_{X/Y} [-1] \otimes  \mathrm{S}(\mathrm{N}_{X/Y}[-1]) \simeq 
 \mathrm{N}_{X/Y} [-1] \otimes  \mathrm{U}(\mathrm{N}_{X/Y}[-1]) \rightarrow \mathrm{U}(\mathrm{N}_{X/Y}[-1])
\]
vanishes, where the last map is given the multiplication in the algebra $\mathrm{U}(\mathrm{N}_{X/Y}[-1])$. We can rewrite this map (using duality) as a morphism from $\mathrm{U}(\mathrm{N}_{X/Y}[-1])$ to $\mathrm{N}^*_{X/Y}[1] \otimes \mathrm{U}(\mathrm{N}_{X/Y}[-1])$, and the question reduces to the vanishing of the map
\[
\omega_{X/Y} \rightarrow   \mathrm{S}(\mathrm{N}_{X/Y}[-1])  \simeq \mathrm{U}(\mathrm{N}_{X/Y}[-1]) \rightarrow \mathrm{N}^*_{X/Y}[1] \otimes \mathrm{U}(\mathrm{N}_{X/Y}[-1])
\]
Using Theorem \ref{zeboss}, we have a commutative diagram
\[
\xymatrix{
\mathcal{RH}om^{r}_{\mathcal{O}_Y}(\mathcal{O}_X, \mathcal{O}_X)  \ar[r]^-{\Delta_1 \circ } \ar[d]_-{\sim} & \mathcal{RH}om^{r}_{\mathcal{O}_Y}(\mathcal{O}_X, \mathrm{N}^*_{X/Y}[1]) \ar[d]^-{\sim} \\
\mathrm{U}(\mathrm{N}_{X/Y}[-1]) \ar[r] & \mathrm{N}^*_{X/Y}[1] \otimes \mathrm{U}(\mathrm{N}_{X/Y}[-1])
}
\]
Hence the morphism we want to look at is (modulo isomorphism on the target) the composition
\[
\omega_{X/Y} \simeq  \mathcal{RH}om^{\textrm{r}}(\mathcal{O}_X, \mathcal{O}_Y) \rightarrow  \mathcal{RH}om^{\textrm{r}}(\mathcal{O}_X, \mathcal{O}_X) \rightarrow  \mathcal{RH}om^{\textrm{r}}(\mathcal{O}_X, \mathrm{N}^*_{X/Y}[1]).
\]
We can factor the first arrow through  $\mathcal{RH}om^{\textrm{r}}(\mathcal{O}_X, \mathcal{O}_S)$, and the composition
\[
\mathcal{RH}om^{\textrm{r}}(\mathcal{O}_X, \mathcal{O}_S) \rightarrow  \mathcal{RH}om^{\textrm{r}}(\mathcal{O}_X, \mathcal{O}_X) \rightarrow  \mathcal{RH}om^{\textrm{r}}(\mathcal{O}_X, \mathrm{N}^*_{X/Y}[1])
\]
vanishes, since it is obtained by composing two successive arrows of a distinguished triangle.
\end{proof}

\begin{theorem} \cite{Yu}
Let $(X, \sigma)$ be a tame quantized cycle in $Y$. Via the isomorphism 
\[
\mathrm{RHom}_{\mathrm{D}^{\mathrm{b}}(Y)}(\omega_{X/Y}, \mathrm{S}(\mathrm{N}_{X/Y}[1])) \simeq \mathbb{H}^0(X, \mathrm{S}(\mathrm{N}^*_{X/Y}[1])) \simeq \mathrm{S}(\mathrm{N}^*_{X/Y}[1])^{\mathrm{N}_{X/Y} [-1]},
\]
the image of the morphism \eqref{shuriken} is the Duflo element of $\mathrm{N}_{X/Y}[-1].$
\end{theorem}

\begin{proof}
Thanks to Proposition \ref{zurich} and Theorem \ref{serpette}, the morphism $ \eqref{shuriken}$ is obtained by contracting the morphism
\[
\omega_{X/Y} \simeq \mathrm{S}^d (\mathrm{N}_{X/Y}[-1]) \rightarrow \mathrm{S} (\mathrm{N}_{X/Y}[-1])
\]
by the Duflo element of $\mathrm{N}_{X/Y}[-1]$. This gives the result.
\end{proof}

\section{The Ext algebras}
\subsection{Definitions}

Let $(X, Y, \sigma)$ be a fixed quantized cycle. For any $k$, we denote by $X^{(k)}_Y$ is the $k^{\mathrm{th}}$ formal neighborhood of $X$ in $Y$.
\begin{definition}
$\mathscr{A}^{(k)}_{X/Y}$ is the algebra  $\bigoplus_{n=0}^{\infty} \mathrm{Ext}^n_{X^{(k)}_Y}(\mathcal{O}_X, \mathcal{O}_X)$, 
the algebra structure being given by the Yoneda product. 
\end{definition}

For any $k$, there are natural algebra morphisms
\[
\mathscr{A}^{(0)}_{X/Y} \hookrightarrow \mathscr{A}^{(1)}_{X/Y} \longrightarrow  \mathscr{A}^{(2)}_{X/Y} \longrightarrow \cdots \longrightarrow \mathscr{A}^{(k)}_{X/Y} \longrightarrow \cdots \longrightarrow   \mathscr{A}^{(\infty)}_{X/Y}
\]
Note that all the algebras $\mathscr{A}^{(k)}_{X/Y}$ are naturally graded by the integer $n$, we call this grading the degree grading.
Thanks to the Proposition \eqref{buche}, the algebra $\mathscr{A}^{(1)}_{X/Y}$ is canonically  isomorphic (as an algebra) to the algebra 
\[
\bigoplus_{n=0}^{\infty} \bigoplus_{p=0}^n \mathrm{Ext}^{n-p}(\mathrm{T}^p \mathrm{N}^*_{X/Y}, \mathcal{O}_X)
\]
via the map that attaches to any $\varphi$ in $\mathrm{Ext}^{n-p}(\mathrm{T}^p \mathrm{N}^*_{X/Y}, \mathcal{O}_X)$ the element $\varphi \circ \Delta_p$.
As a corollary, the algebra $\mathscr{A}^{(1)}_{X/Y}$ carries another natural grading given by the integer $p$, which is completely different from the degree grading; we call this grading the Lie grading. Elements of depth zero correspond to the sub-algebra $\mathscr{A}^{(0)}_{X/Y}$.
\par \medskip
Using the sequence \eqref{C}, the composition
\[
\bigoplus_{n=0}^{\infty} \bigoplus_{p=0}^n \mathrm{Ext}^{n-p}(\Lambda^p \mathrm{N}^*_{X/Y}, \mathcal{O}_X) \longrightarrow \bigoplus_{n=0}^{\infty} \bigoplus_{p=0}^n \mathrm{Ext}^{n-p}(\mathrm{T}^p \mathrm{N}^{*}_{X/Y}, \mathcal{O}_X) \simeq \mathscr{A}^{(1)}_{X/Y} \longrightarrow \mathscr{A}^{(\infty)}_{X/Y}
\]
is an isomorphism of $\mathbf{k}$-vector spaces. Hence the map $
\mathscr{A}^ {(1)}_{X/Y} \longrightarrow \mathscr{A}^ {(\infty)}_{X/Y}
$
is surjective, and there is an isomorphism 
\[
\mathscr{A}_{X/Y}^{(\infty)} \simeq \bigoplus_{n=0}^{\infty} \bigoplus_{p=0}^n \mathrm{Ext}^{n-p}(\Lambda^p \mathrm{N}^*_{X/Y}, \mathcal{O}_X).
\]
of $\mathbf{k}$-vector spaces obtained by attaching to any $\varphi$ in $\mathrm{Ext}^{n-p}(\Lambda^p \mathrm{N}^*_{X/Y}, \mathcal{O}_X)$ the push forward of the element $\varphi \circ \Delta^-_p$ from $S$ to $Y$. The integer $p$ defines a grading on $\mathscr{A}_{X/Y}^{(\infty)}$ but it no longer respects the algebra structure, however the corresponding ascending filtration does. We call this filtration the Lie filtration.

\subsection{Enrichment} \label{Macron}

For any scheme (or complex manifold) $Z$, the category $\mathrm{D}^{\mathrm{b}}(Z)$ is enriched over $\mathrm{D}^{\mathrm{b}}(\mathbf{k})$ in the following way: 
\[
\mathrm{D}^{\mathrm{b}}(Z)(\mathcal{F}, \mathcal{G}):=\mathrm{R} \Gamma (X, \mathcal{RH}om_{\mathrm{D}^{\mathrm{b}}(Z)}(\mathcal{F}, \mathcal{G}))\,.
\]

\begin{remark}
$ $ \par
\begin{enumerate}
\item[--] The enrichment is symmetric monoidal. In particular, the tensor product functor is an enriched functor.  
\item[--] As $\mathbf{k}$ is a field, we can replace $\mathrm{D}^{\mathrm{b}}(\mathbf{k})$ by the equivalent category $\mathbf{k}^ {\mathbb{Z}}$. Through this equivalence we have
\[
\mathrm{D}^{\mathrm{b}}(Z)(\mathcal{F}, \mathcal{G}) \simeq \bigoplus_{n \in \mathbb{Z}} \mathrm{Hom}_{\mathrm{D}^{\mathrm{b}}(Z)}(\mathcal{F}, \mathcal{G}[n]).
\]
\item[--] The Ext algebras admit a very simple description using this formalism, they are given by the formula
\[
\mathscr{A}^{(k)}_{X/Y}=\mathrm{D}^{\mathrm{b}}\left(X^{(k)}_Y\right)(\mathcal{O}_X,\mathcal{O}_X)\,.
\]
\end{enumerate}
\end{remark}

Let $E$ an object in $\mathrm{D}^{\mathrm{b}}(Z)$, which we view as a module over the Lie algebra object $\mathrm{T}_X[-1]$ using the Atiyah class. 
There are two possible definitions of the $V$-invariants $E^V$ of $E$: 
\begin{enumerate}
\item[--] The standard one: $E^V:=\mathrm{Hom}_{\mathrm{D}^{\mathrm{b}}(Z)}(\mathcal O_Z,E)$, 
where $\mathcal O_Z$ is equipped with the trivial $\mathrm{T}_X[-1]$-module structure. 
\item[--] The enriched one: $E^V:=\mathrm{D}^{\mathrm{b}}(Z)(\mathcal O_Z,E)$. 
\end{enumerate}
Below we always use the enriched version. 

\subsection{Structure theorem} \label{yaaa}
If $(X, \sigma)$ is tame, we can explicitly describe the Ext algebra $\mathscr{A}_{X/Y}^{(\infty)}$. 
We have an exact sequence
\[
0 \rightarrow \mathrm{T}_X[-1] \rightarrow (\mathrm{T}_Y[-1])_{|X} \rightarrow \mathrm{N}_{X/Y}[-1] \rightarrow 0.
\]
Recall that $\mathrm{N}_{X/Y}[-1]$ is naturally endowed with a Lie structure (given by $\alpha$). 

\begin{theorem} \label{swan}
Assume that $(X, \sigma)$ is a tame quantized cycle in $Y$. Then using the corresponding Lie structure on $\mathrm{N}_{X/Y}[-1]$, the algebra
$\mathscr{A}^{(\infty)}_{X/Y}$ is naturally isomorphic to $\mathrm{U}(\mathrm{N}_{X/Y}[-1])^{\mathrm{T}_{X}[-1]}$. Besides, there is a commutative diagram
\[
\xymatrix{\mathscr{A}^{(1)}_{X/Y} \ar[r] \ar[d] & \mathscr{A}^{(\infty)}_{X/Y}  \ar[d] \\
\left(\mathrm{T}(\mathrm{N}_{X/Y}[-1])\right)^{\mathrm{T}_X[-1]} \ar[r] & \left(\mathrm{U}(\mathrm{N}_{X/Y}[-1])\right)^{\mathrm{T}_X[-1]}
}
\]
\end{theorem}
\begin{proof}
We have already proved that there is an isomorphism of algebra objects between $\mathcal{RH}om^{\ell}_{\mathcal O_Y}(\mathcal O_X,\mathcal O_X)$ and 
$\mathrm{U}(\mathrm{N}_{X/Y}[-1])$. Applying the derived global section functor we get an isomorphism of graded algebras between 
$\mathscr{A}^{(\infty)}_{X/Y}$ and $\mathrm{R}\Gamma\left(X,\mathrm{U}(\mathrm{N}_{X/Y}[-1])\right)$. 
Finally, observe that the $\mathrm{T}_{X}[-1]$-module structure of $\mathrm{U}(\mathrm{N}_{X/Y}[-1])\cong \mathrm{S}(\mathrm{N}_{X/Y}[-1])$ is given by the Atiyah class of $\mathrm{S}(\mathrm{N}_{X/Y}[-1])$. Therefore the algebra $\mathrm{R}\Gamma\left(X,\mathrm{U}(\mathrm{N}_{X/Y}[-1])\right)$ is indeed 
$\mathrm{U}(\mathrm{N}_{X/Y}[-1])^{\mathrm{T}_{X}[-1]}$. 
\end{proof}

\subsection{The algebra  \texorpdfstring{$\mathscr{A}^{(2)}_{X/Y}$}{A2}}
In this section, we describe completely the image of $\mathscr{A}^{(1)}_{X/Y}$ in $\mathscr{A}^{(2)}_{X/Y}$ for tame quantized cycles. The result, which seems quite surprising at first sight, runs as follows:
\begin{theorem} \label{serpette2}
Assume that $(X, \sigma)$ is tame in $Y$. The surjective morphism from $\mathscr{A}^{(2)}_{X/Y}$ to $\mathscr{A}^{(\infty)}_{X/Y}$ admits a canonical section.
\end{theorem}

\begin{proof}
Let $\mathscr{R}_{X/Y}$ be the kernel of the map $\mathscr{A}^{(1)}_{X/Y} \rightarrow \mathscr{A}^{(\infty)}_{X/Y}$. The kernel of the map $\mathscr{A}^{(1)}_{X/Y} \rightarrow \mathscr{A}^{(2)}_{X/Y}$ is a subalgebra of $\mathscr{R}_{X/Y}$, and we must prove that any element of $\mathscr{R}_{X/Y}$ maps to zero in $\mathscr{A}^{(2)}_{X/Y}$. Elements of $\mathscr{R}_{X/Y}$ are of the form $\sum_{p \geq 0} \alpha_p \circ (\Delta_p-\sum_{k=1}^p c_p^k \circ \Delta^+_k)$ for $\alpha_p$ in $\mathrm{Ext}^*(\bigotimes^p \mathrm{N}^*_{X/Y}, \mathcal{O}_X)$. Hence the result follows directly from Corollary \ref{sky2}.

\end{proof}

\section{Conclusion and perspectives}

\subsection{State of the art in the tame case}

In this paper we introduced a tameness condition for a quantized cycle $(X,\sigma)$ in $Y$. Under this assumption we proved that the shifted normal sheaf $\mathrm{N}_{X/Y}[-1]$ is endowed with a Lie bracket and that the sheaf of derived $\mathcal O_Y$-linear endomorphisms of $\mathcal O_X$ is isomorphic to the universal enveloping algebra of $\mathrm{N}_{X/Y}[-1]$. Using this, we were able to: 
\begin{itemize}
\item[--] describe explicitly the Ext algebra $\mathrm{Ext}^\bullet_Y(\mathcal O_X,\mathcal O_X)$, giving simple and conceptual proofs in the diagonal case of results of Ramadoss and Kapranov.
\item[--] identify the quantized cycle class with the Duflo element of the Lie algebra object $\mathrm{N}_{X/Y}[-1]$, recovering and reinterpreting in Lie algebraic terms a result Yu. 
\end{itemize}

\subsection{Beyond the tame case}

At first sight, the diagonal cycle $X\hookrightarrow X\times X$ admits two distinguished quantizations (\textit{i.e.} first order retractions): the two projections $\mathrm{pr}_1$ and $\mathrm{pr}_2$. However, since the space of quantizations of a cycle is an affine space, this gives a whole line of quantizations, namely $t \mathrm{pr}_1 + (1-t) \mathrm{pr}_2$ for $t$ in the base field $\mathbf{k}$.
\par \medskip
In general these quantizations are never tame except for $t=0$ or $t=1$, that is for the two projections. However, the value $t=\displaystyle \frac{1}{2}$ is special since the corresponding quantized cycle fits in a very interesting family of quantized cycles: every fixed-point locus $X=Y^{\iota}$ of an involution $\iota$ of $Y$ defines naturally a quantized cycle. In our former example, $Y=X \times X$ and $\iota (x, x')=(x', x)$. 
In Lie algebraic terms this corresponds to symmetric pairs. We will study further on this case in future work.

\subsection{Link with Duflo's conjecture for symmetric pairs}
If $\mathfrak{g}$ is a finite-dimensional Lie algebra (over a characteristic zero field) then Duflo's theorem \cite{duflo_operateurs_1977} states that the composition
\[
\mathrm{S}(\mathfrak{g}) \xrightarrow{\sqrt{\mathfrak{d}_{\mathfrak{g}}} \,\lrcorner} \mathrm{S}(\mathfrak{g}) \xrightarrow{\mathrm{PBW}} \mathrm{U}(\mathfrak{g}) 
\] 
induces an isomorphism of algebras between the algebra of invariants\footnote{This algebra is sometimes called the algebra of Casimirs of $\mathfrak{g}$.} $\mathrm{S}(\mathfrak{g})^{\mathfrak{g}}$ and the center $\mathrm{Z}\big(\mathrm{U}(\mathfrak{g})\big)$ of $\mathrm{U}(\mathfrak{g})$. Here, $\mathfrak{d}_{\mathfrak{g}}$ is the Duflo element of $\mathfrak{g}$ in $\widehat{S}(\mathfrak{g}^*)$ introduced in \S \ref{dudu}. This theorem is believed to be true in any reasonable categorical context, that is for any dualizable Lie algebra object in a symmetric monoidal category with the hypotheses we put to get the categorical PBW theorem. However, this goal is out of reach in full generality (aside from the abelian case due to Kontsevich), even though it has been achieved in the enriched\footnote{\textit{See} \S \ref{Macron}} monoidal category $\mathrm{D}^{\mathrm{b}}(X)$ for the Lie algebra object $\mathrm{T}_X[-1]$ in \cite{calaque_hochschild_2010}. For historical considerations around Duflo's theorem and its geometric counterpart, we refer the reader to the monograph \cite{calaque_lectures_2011}.
\par \medskip
Duflo's theorem has a conjectural extension for reductive pairs (\textit{see} \cite{duflo_open_problems}), which runs as follows: given such a pair $(\mathfrak{g}, \mathfrak{h})$, there should be an isomorphism of algebras
\[
\mathrm{ZP}\left\{ \left( \frac{\mathrm{S}(\mathfrak{g})}{\mathrm{S}(\mathfrak{g}) \mathfrak{h}_{\rho}} \right)^{\mathfrak{h}} \right\}\simeq \mathrm{Z}\left\{\left( \frac{\mathrm{U}(\mathfrak{g})}{\mathrm{U}(\mathfrak{g}) \mathfrak{h}} \right)^{\mathfrak{h}}\right\}.
\]
where the notation is explained below:
\begin{enumerate}
\item[--] $\rho$ is half the character of the $\mathfrak{h}$-representation $\mathfrak{g}/\mathfrak{h}$.
\vspace{0.2cm}
\item[--] $\mathfrak{h}_{\rho}=\{h+\rho(h)|h\in\mathfrak{h}\}\subset \mathrm{S}(\mathfrak{g})$
\vspace{0.2cm}
\item[--] $\left(\displaystyle \frac{\mathrm{S}(\mathfrak{g})}{\mathrm{S}(\mathfrak{g}) \mathfrak{h}_{\rho}} \right)^{\mathfrak{h}}$ is naturally a Poisson algebra\footnote{It is the algebra of regular functions on the reduction along the coisotropic subspace $\{\lambda\in\mathfrak{g}^*|\lambda_{|\mathfrak{h}}=\rho\}$ in the linear Poisson space $\mathfrak{g}^*$.}.
\vspace{0.2cm}
\item[--] $\mathrm{ZP}$ denotes the Poisson center. 
\end{enumerate}
It is interesting to understand its geometrization. To do this, we consider the reductive pair $(\mathrm{T}_Y[-1]_{|X},\mathrm{T}_X[-1])$ attached to a quantized cycle. Of course, in a categorical context, $\displaystyle \left(\frac{\mathrm{U}(\mathfrak{g})}{\mathrm{U}(\mathfrak{g}) \mathfrak{h}}\right)^{\mathfrak{h}}$ has to be understood as $\mathrm{Hom}_{\mathfrak{g}}\big(\mathrm{Ind}_{\mathfrak{h}}^{\mathfrak{g}} \, \mathbf{1}_{\mathcal{C}},\mathrm{Ind}_{\mathfrak{h}}^{\mathfrak{g}} \, \mathbf{1}_{\mathcal{C}}\big)$. 
\par \medskip
In the tame case, the results of \S \ref{chili} and \S  \ref{yaaa} play a crucial role. Indeed, combining Theorem \ref{harpe} and Theorem \ref{yolo}, the left hand side of Duflo's conjecture is $\mathrm{U}(\mathrm{N}_{X/Y}[-1])^{\mathrm{T}_X[-1]}$ endowed with the opposite algebra structure. Thanks to Theorem \ref{swan},  it is the center of $\mathscr{A}_{X/Y}^{\infty}$.

\appendix
\section{Operadic proof of the categorical PBW} \label{AAAA}
In this appendix we sketch a proof of the following folklore result: the PBW theorem holds in any Karoubian symmetric monoidal category in which countable sums exist and such that the monoidal product preserves these in each variable. We also describe the multiplication by elements of degree one in the universal envelopping algebra. The proof makes use of linear operads, for which we recommend the reference \cite{LV}. 
\begin{notation}
For a set $\mathcal V$ of variables and a given bijection $x:\{1,\dots,k\}\to \mathcal V$ we will allow ourselves to write a $k$-ary operation $m$ in variable notation: $m=M\left((x(1),\dots, x(k)\right)$. In this notation the symmetric group action reads $m\cdot\sigma=M\left(x(\sigma(1)),\dots, x(\sigma(k))\right)$,  and the operad composition is expressed in terms of substitution: given another operation $n=N\big(y(1),\dots,y(l)\big)$ written in variable notation, the partial composition $m\circ_i n$ is
\[
m\circ_i n=M\left(x(1),\dots,x(i-1),N\left(y(1),\dots,y(l)\right),x(i+1),\dots,x(k)\right).
\]
\end{notation}

\subsection{Symmetric multibraces}

Let us denote by $\mathcal{SMB}$ the operad of symmetric multibrace algebras from \cite[\S4.5]{LR}: it is the operad generated by operations $m_{p,q}$ (where $p,q\geq1$) of arity $p+q$, such that: 
\begin{itemize}
\item[(a)] $m_{p,q}$ is $(p,q)$-symmetric: $m_{p,q}\cdot \sigma=m_{p,q}$ if $\sigma\in\mathfrak{S}_p\times\mathfrak{S}_q$. 
\item[(b)] The operations $m_{p,q}=M_{p,q}(x_1,\dots,x_p,y_1,\dots,y_q)$ satisfy the relation $\mathcal{SR}$ from in \cite[Proposition 4.3]{LR}. 
\end{itemize}
Here and below, we use the convention that $m_{p,0}$ (resp.~$m_{0,q}$) is the identity operation if $p=1$ (resp.~$q=1$) and $0$ otherwise. 

\subsection{Interpretation in terms of bialgebra structure}

Let $\mathcal C$ be a symmetric monoidal $\mathbf{k}$-linear category that is Karoubian and with arbitrary countable sums. This is sufficient for $\mathrm{S}(V)$ to exist, given any object $V$. 
Observe that $\mathrm{S}(V)$ carries a coaugmented cocommutative coassociative counital coalgebra structure: 
\begin{enumerate}
\item[--] The coproduct $\nabla:\mathrm{S}(V)\to\mathrm{S}(V)\otimes\mathrm{S}(V)$ given as follows: on $\mathrm{S}^n(V)$, 
\[
\nabla=\sum_{p=0}^n \binom{n}{p} \,\nabla_{p,n-p}\,,
\]
where $\nabla_{p,q}$ is the inclusion $\mathrm{S}^n(V)\hookrightarrow\mathrm{S}^p(V)\otimes\mathrm{S}^q(V)$. 
\item[--] The counit $\epsilon:\mathrm{S}(V)\to\mathbf{1}_{\mathcal{C}}$ is the projection onto the direct factor $\mathrm{S}^0(V)=\mathbf{1}_{\mathcal{C}}$. 
\item[--] The coaugmentation $1:\mathbf{1}_{\mathcal{C}}\to\mathrm{S}(V)$ is the direct factor inclusion of $\mathrm{S}^0(V)=\mathbf{1}_{\mathcal{C}}$. 
\item[--] We will also need the reduced coproduct $\bar{\nabla}:=\nabla-(\mathrm{id}\otimes1+1\otimes \mathrm{id})$. 
\end{enumerate}

\par \medskip

On $\mathrm{S}^n(V)$, the iterated reduced coproduct $\bar{\nabla}^{(k)}$ is the sum, over all $(k+1)$-tuples $(n_0,\dots,n_k)$ such that $n_0 + \ldots + n_k=n$, of $\binom{n}{n_1,\dots, n_k}$ times the inclusions $\mathrm{S}^n(V)\hookrightarrow \mathrm{S}^{n_0}(V)\otimes\cdots\otimes \mathrm{S}^{n_k}(V)$. In particular, on $\mathrm{S}^n(V)$, the following properties hold: 
\begin{itemize}
\item[--] If $k=n$ then $\bar{\nabla}^{(n)}$ is $n!$ times the inclusion of $S^n(V)$ in $V^{\otimes n}$. 
\item[--] If $k>n$ then $\bar{\nabla}^{(k)}=0$. 
\end{itemize}

\par\medskip

It has been proven in \cite[Section 4]{LR} that an $\mathcal{SMB}$-algebra structure on $V$ is the same as an associative product $m$ on $\mathrm{S}(V)$ such that $m$ is a coaugmented counital coalgebra morphism: $\nabla\circ m=m^{\otimes 2}\circ(23)\circ\nabla^{\otimes2}$ and $\epsilon\circ m=\epsilon\otimes\epsilon$. Indeed, being a coaugmented counital coalgebra morphism imposes, by cofreeness of $\mathrm{S}(V)$, that $m$ is completely determined by the structure maps $m_{p,q}$ defined as the composition 
\[
\mathrm{S}^p(V)\otimes \mathrm{S}^q(V)\overset{m}{\longrightarrow} \mathrm{S}(V)\twoheadrightarrow \mathrm{S}^1(V)=V. 
\]
Associativity of $m$ is equivalent to condition (b) on the structure maps $m_{p,q}$. 

\subsection{A morphism of operads $\mathcal{L}ie\to\mathcal{SMB}$}

Recall that the Lie operad is the linear operad generated by a skew-symmetric binary operation obeying the Jacobi identity. 
In arity $n$, $\mathcal{L}ie(n)$ consists of (linear combinations of) Lie monomials in $x_1,\dots,x_n$ in which every $x_i$ appears exactly once\footnote{In other words, the free Lie $\mathbf{k}$-algebra $\textsc{Lie}_n$ in $n$ variables $x_1,\dots,x_n$ carries an obvious $\mathbb{Z}_+^n$-grading, and $\mathcal{L}ie(n)=\textsc{Lie}_n^{(1,\dots,1)}$. The operadic structure is given by substitution.}. 
Hence the generating operation is $[x_1,x_2]\in\mathcal{L}ie(2)$. 

\medskip

There is an obvious morphism of operads $f:\mathcal{L}ie\to\mathcal{SMB}$ sending $[x_1,x_2]$ to $\displaystyle \frac{m_{1,1}-m_{1,1}\cdot(12)}{2}$, so that any symmetric multibrace algebra is naturally a Lie algebra. 

\subsection{A morphism of operads $\mathcal{SMB}\to\mathcal{L}ie$}

In this section, we construct a less obvious morphism of operads $\mathcal{SMB}\to\mathcal{L}ie$. 

\medskip

We use the classical PBW theorem for Lie algebras over a field of characteristic zero, and apply it to the case of the free Lie $\mathbf{k}$-algebra $\textsc{Lie}_{p+q}$ in $p+q$ generators $x_1,\dots,x_p,y_1,\dots,y_q$. 
We have an algebra structure $*$ on the cocommutative coalgebra $\mathrm{S}(\textsc{Lie}_{p+q})\simeq\mathrm{U}(\textsc{Lie}_{p+q})$ that is defined \textit{via} the exponential formula $\mathrm{exp}(u)* \mathrm{exp(v)}=\mathrm{exp} \left(\textsc{bch}(u,v)\right)$, where $\textsc{bch}(u,v)$ is the Baker--Campbell-Hausdorff series, which is a Lie series in two variables. 

\medskip

This product is easily seen to be compatible with the coproduct\footnote{This is the classical statement that $\mathrm{U}(\mathfrak{g})$ is a bi-algebra, see for instance \cite[Chap. 2 \S1.4 Proposition 7]{Bourbaki}.}, leading to a symmetric multibrace algebra structure on $\textsc{Lie}_{p+q}$. One can check from the exponential formula for the product that the output expression
$M_{p,q}(x_1\dots x_p,y_1\dots y_q)$ is the part inside $\textsc{bch}(x_1+\cdots+x_p,y_1+\cdots+y_q)$ consisting in monomials where every $x_i$ and every $y_j$ appears exactly once\footnote{In other words, $M_{p,q}(x_1\dots x_p,y_1\dots y_q)=\textsc{bch}(x_1+\cdots+x_p,y_1+\cdots+y_q)^{(1,\dots,1)}$.}, and thus lies in $\mathcal{L}ie(p+q)$. 

\par \medskip

It is almost tautological\footnote{More precisely, taking the multi-degree $(1,\dots,1)$ part of 
\[
\textsc{bch}\big(\textsc{bch}(x_1+\cdots+x_p,y_1+\cdots+y_q),z_1+\cdots+z_r\big)=
\textsc{bch}\big(x_1+\cdots+x_p,\textsc{bch}(y_1+\cdots+y_q,z_1+\cdots+z_r)\big)\,.
\]
precisely gives the relation $(\mathcal{SR}_{pqr})$ from \cite[Proposition 4.3]{LR}. } that sending $m_{p,q}$ in $\mathcal{SMB}(p+q)$ to $M_{p,q}(x_1\dots x_p,y_1\dots y_q)$ in $\mathcal{L}ie_{p+q}$ indeed defines a morphism of operads $g:\mathcal{SMB}\to\mathcal{L}ie$. One can check that $g\circ f$ is the identity. 

\subsection{The categorical PBW theorem} \label{audiA5}

Form the above we get that every Lie algebra object $(V,\alpha)$ in a symmetric monoidal category inherits a symmetric brace algebra structure. 
In other words, there is an associative unital product $m_\star:\mathrm{S}(V)^{\otimes2}\to \mathrm{S}(V)$ that is a morphism of coaugmented counital coalgebra. 
From the compatibility with the coproduct, the counit, and the coaugmentation, it follows immediately that $m_\star$ satisfies the following properties: 
\begin{enumerate}
\item[--] $m_\star$ is filtered: $m_\star$ restricted to $\mathrm{S}^{\leq i}(V)\otimes \mathrm{S}^{\leq j}(V)$ factors through $\mathrm{S}^{\leq i+j}(V)\subset \mathrm{S}(V)$. 
\item[--] $m_\star$ is a deformation of $m_0$: on $\mathrm{S}^i(V)\otimes \mathrm{S}^j(V)$, $m_\star-m_0$ factors through $\mathrm{S}^{<i+j}(V)\subset \mathrm{S}(V)$.
\item[--] On $V^{\otimes 2}$, $m_\star-m_\star\cdot(12)=\alpha$. 
\end{enumerate}
The first property follows from the fact that $\mathrm{S}^{\leq k}(V)$ is the kernel of $\bar{\nabla}^{(k+1)}$, and that the multiplication $m_\star$ is a morphism of coalgebras. The second property follows from the fact that $\bar\nabla^{(i+j)}$ is $(i+j)!$ times the inclusion $\mathrm{S}^{i+j}(V)\subset V^{\otimes i+j}$. Only the third property uses the morphism $\mathcal{SMB}\to\mathcal{L}ie$: it follows from the fact that $m_{1,1}$ is sent to $[x_1,x_2]$. 

\subsection{The derivative of the multiplication map} \label{audiA6}

Our aim in this section is to prove that the restriction $\varphi$ of the product $m_\star$ to $\mathrm{S}(V)\otimes V\subset \mathrm{S}(V)\otimes\mathrm{S}(V)$ is equal to $\psi:=m_0\circ \displaystyle \frac{\omega}{1-\exp(-\omega)}$, where $\omega$ is defined at the beginning of \S \ref{oural}. Our strategy runs as follows:
\par \medskip 
\begin{enumerate}
\item[\textbf{Step 1}]\,\,\, Show that $\psi$ is compatible with the coproduct\footnote{Such an identity somehow makes $\psi$ into a ``family of coderivations'' (parametrized by $V$). }: 
\begin{eqnarray*}
\nabla\circ\psi & \overset{?}{=} & \psi^{\otimes 2}\circ(23)\circ\nabla^{\otimes2} \\
& = & \psi^{\otimes 2}\circ(23)\circ\left(\nabla\otimes(\mathrm{id}\otimes1+1\otimes \mathrm{id})\right) \qquad (\nabla_{|V}=\mathrm{id}\otimes1+1\otimes \mathrm{id})\\
& = & \left(\mathrm{id}+(12)\right)(\mathrm{id}\otimes\psi)\circ(\nabla\otimes \mathrm{id}) \qquad (\nabla=(12)\nabla)\,.
\end{eqnarray*}
\item[\textbf{Step 2}]\,\,\, Using the first step, one gets that the morphism $\psi$ is completely determined by its structure maps 
$\psi_{p,1}:\mathrm{S}^p(V)\otimes V\to V$. To end the proof one shall show that $\psi_{p,1}=\varphi_{p,1}$. 
\end{enumerate}

\begin{proof}[Proof of Step 1]
To prove that $\psi$ is compatible with the coproduct, it is sufficient to prove that so are all the $q_k:=m_0\circ\omega^{\circ k}$; that is
\begin{equation}\label{equationsuperhot}
\nabla\circ q_k \overset{?}{=} (\mathrm{id}+(12))(\mathrm{id}\otimes q_k)\circ(\nabla\otimes \mathrm{id})\,.
\end{equation}
On the one hand, on $\mathrm{S}^n(V)\otimes V$ we have by definition
\[
(\mathrm{id}+(12))(\mathrm{id}\otimes q_k)\circ(\nabla\otimes \mathrm{id}) = 
(\mathrm{id}+(12))(\mathrm{id}\otimes m_0)\circ(\mathrm{id}\otimes\omega^{\circ k})\circ(\nabla\otimes \mathrm{id})\,.
\]

On the other hand, on $\mathrm{S}^n(V)\otimes V$ we also have:
\begin{eqnarray*}
\nabla\circ q_k & = & \nabla\circ m_0\circ\omega^{\circ k} \\
& = & m_0^{\otimes2}\circ(23)\circ\nabla^{\otimes2}\circ\omega^{\circ k} \\
& = & (\mathrm{id}+(12))\circ(\mathrm{id}\otimes m_0)\circ(\nabla\otimes \mathrm{id})\circ\omega^{\circ k}\,.
\end{eqnarray*}

Hence, to prove the desired identity \eqref{equationsuperhot} it is sufficient to prove that 
\[
(\nabla\otimes \mathrm{id})\circ\omega^{\circ k}=(\mathrm{id}\otimes\omega^{\circ k})\circ(\nabla\otimes \mathrm{id}).
\]
We now compute the two sides of the above identity, starting with the left-hand-one: 
\[
(\nabla\otimes \mathrm{id})\circ\omega^{\circ k}=n(n-1)\cdots (n-k+1)\sum_{p+q'=n-k} \, \binom{n-k}{p} \, \nabla_{p,q'}\circ\underbrace{(\mathrm{id}\otimes\alpha)\cdots(\mathrm{id}\otimes\alpha)}_{k\textrm{times}}.
\]
The right-hand-side reads: 
\begin{eqnarray*}
(\mathrm{id}\otimes\omega^{\circ k})\circ(\nabla\otimes \mathrm{id}) & = & \sum_{p+q=n} q(q-1)\cdots(q-k+1) \,\binom{n}{p}\,\underbrace{(\mathrm{id}\otimes\alpha)\cdots(\mathrm{id}\otimes\alpha)}_{k\textrm{ times}}\circ\nabla_{p,q} \\
& = & \sum_{p+q=n} q(q-1) \cdots(q-k+1) \,\binom{n}{p} \,\nabla_{p,q-k}\circ\underbrace{(\mathrm{id}\otimes\alpha)\cdots(\mathrm{id}\otimes\alpha)}_{k\textrm{times}}.
\end{eqnarray*}
We finally conclude by identifying coefficients (fixing $p$ and $q$, and having $q'=q-k$):
\[
n(n-1)\cdots (n-k+1)\,\binom{n-k}{p}\,=\frac{n!(n-k)!}{(n-k)!p!(q-k)!}=\frac{n!}{p!(q-k)!}=\frac{q!n!}{(q-k)!p!q!}=q(q-1)\cdots(q-k+1)\,\binom{n}{p}.
\]
\end{proof}
\begin{proof}[Proof of Step 2]
Observe that we have the following identity in $\textsc{Lie}_{p+1}$: 
\[
M_{p,1}(x_1,\dots, x_p,y)=\textsc{bch}(x_1+\cdots+x_n,y)^{(1,\dots,1)}
=(-1)^p\frac{B_p}{p!}\sum_{\sigma\in\mathfrak{S}_p}\big[x_{\sigma(1)},\dots,[x_{\sigma(p)},y]\dots\big].
\]
Hence we get that 
\[
\varphi_{p,1}=(-1)^p\frac{B_p}{p!}\omega^{\circ p}=\psi_{p,1}.
\]
\end{proof}

\bibliographystyle{plain}
\bibliography{biblio}

\begin{thebibliography}{10}

\bibitem{Arinkin-Caldararu}
Dima Arinkin and Andrei C{\u{a}}ld{\u{a}}raru.
\newblock When is the self-intersection of a subvariety a fibration?
\newblock {\em Adv. Math.}, 231(2):815--842, 2012.

\bibitem{Bourbaki}
Nicolas Bourbaki.
\newblock {\em Algebra {I}. {C}hapters 1--3}.
\newblock Elements of Mathematics (Berlin). Springer-Verlag, Berlin, 1998.
\newblock Translated from the French, Reprint of the 1989 English translation.

\bibitem{CV}
D.~Calaque and M.~Van~den Bergh.
\newblock {Hochschild cohomology and {A}tiyah classes}.
\newblock {\em Adv. Math.}, 224(5):1839--1889, 2010.

\bibitem{CalaqueCT2}
Damien Calaque, Andrei C{\u{a}}ld{\u{a}}raru, and Junwu Tu.
\newblock {P}{B}{W} for an inclusion of {L}ie algebras.
\newblock {\em J. of Algebra}, 378:64--79, 2013.

\bibitem{CalaqueCT}
Damien Calaque, Andrei C{\u{a}}ld{\u{a}}raru, and Junwu Tu.
\newblock On the {L}ie algebroid of a derived self-intersection.
\newblock {\em Adv. Math.}, 262:751--783, 2014.

\bibitem{calaque_lectures_2011}
Damien Calaque and Carlo~A. Rossi.
\newblock {\em Lectures on {Duflo} isomorphisms in {Lie} algebra and complex
  geometry}.
\newblock {EMS} {Series} of {Lectures} in {Mathematics}. European Mathematical
  Society (EMS), Z\"{u}rich, 2011.

\bibitem{calaque_hochschild_2010}
Damien Calaque and Michel Van~den Bergh.
\newblock Hochschild cohomology and {Atiyah} classes.
\newblock {\em Advances in Mathematics}, 224(5):1839--1889, 2010.

\bibitem{DM}
Pierre Deligne and John~W. Morgan.
\newblock Notes on supersymmetry (following {J}oseph {B}ernstein).
\newblock In {\em Quantum fields and strings: a course for mathematicians,
  {V}ol. 1, 2 ({P}rinceton, {NJ}, 1996/1997)}, pages 41--97. Amer. Math. Soc.,
  Providence, RI, 1999.

\bibitem{duflo_operateurs_1977}
Michel Duflo.
\newblock Op\'{e}rateurs diff\'{e}rentiels bi-invariants sur un groupe de
  {Lie}.
\newblock {\em Annales scientifiques de l'\'{E}cole normale sup\'{e}rieure},
  10(2):265--288, 1977.

\bibitem{duflo_open_problems}
Michel Duflo.
\newblock In Toshio Oshima, editor, {\em Open problems in representation theory
  of Lie groups}, pages 1--5, Katata, Japon, August 25--30 1986.
\newblock Conference on Analysis on homogeneous spaces.

\bibitem{conjecture}
Julien Grivaux.
\newblock On a conjecture of {K}ashiwara relating {C}hern and {E}uler classes
  of {O}-modules.
\newblock {\em J. Differential Geom.}, 90(2):267--275, 2012.

\bibitem{Grivaux-HKR}
Julien Grivaux.
\newblock The {H}ochschild-{K}ostant-{R}osenberg isomorphism for quantized
  analytic cycles.
\newblock {\em Int. Math. Res. Not. IMRN}, (4):865--913, 2014.

\bibitem{Grivaux-chernloc}
Julien {Grivaux}.
\newblock {Derived geometry of the first formal neighborhood of a smooth
  analytic cycle}.
\newblock {\em ArXiv 1505.04414}, 2015.

\bibitem{HKRoriginal}
G.~Hochschild, Bertram Kostant, and Alex Rosenberg.
\newblock Differential forms on regular affine algebras.
\newblock {\em Trans. Amer. Math. Soc.}, 102:383--408, 1962.

\bibitem{Kapranov}
M.~Kapranov.
\newblock Rozansky-{W}itten invariants via {A}tiyah classes.
\newblock {\em Compositio Math.}, 115(1):71--113, 1999.

\bibitem{KS1}
Masaki Kashiwara and Pierre Schapira.
\newblock {\em {Deformation quantization modules}}.
\newblock Ast\'erisque. Soci\'et\'e math\'ematique de France, 2012.

\bibitem{Kontsevich}
Maxim Kontsevich.
\newblock Deformation quantization of {P}oisson manifolds.
\newblock {\em Lett. Math. Phys.}, 66(3):157--216, 2003.

\bibitem{LR}
Jean-Louis Loday and Maria Ronco.
\newblock Combinatorial hopf algebras.
\newblock {\em Clay Mathematics Proceedings Volume}, 11, 2010.

\bibitem{LV}
Jean-Louis Loday and Bruno Vallette.
\newblock {\em Algebraic Operads}, volume 346 of {\em Grundlehren des
  mathematischen Wissenschaften}.
\newblock Springer-Verlag Berlin Heidelberg, 2012.

\bibitem{Markarian}
Nikita Markarian.
\newblock The {A}tiyah class, {H}ochschild cohomology and the {R}iemann-{R}och
  theorem.
\newblock {\em J. Lond. Math. Soc. (2)}, 79(1):129--143, 2009.

\bibitem{Ramadoss2}
Ajay~C. Ramadoss.
\newblock The big {C}hern classes and the {C}hern character.
\newblock {\em Internat. J. Math.}, 19(6):699--746, 2008.

\bibitem{Ramadoss}
Ajay~C. Ramadoss.
\newblock The relative {R}iemann-{R}och theorem from {H}ochschild homology.
\newblock {\em New York J. Math.}, 14:643--717, 2008.

\bibitem{toledo_duality_1978}
Domingo Toledo and Yue Lin~L. Tong.
\newblock Duality and intersection theory in complex manifolds. {I}.
\newblock {\em Mathematische Annalen}, 237(1):41--77, 1978.

\bibitem{Yu}
S.~{Yu}.
\newblock {Todd class via homotopy perturbation theory}.
\newblock {\em ArXiv 1510.07936}, 2015.

\end{thebibliography}

\end{document}